\documentclass[11pt,leqno]{article}
\usepackage{enumitem}
\usepackage[doc]{optional}
\usepackage{color}
\usepackage{float}
\usepackage{soul}
\usepackage{graphicx}
\usepackage{lmodern,bm}

\usepackage{url}

\usepackage{pstricks}
\usepackage{pst-plot}
\usepackage{pstricks-add}

\definecolor{labelkey}{rgb}{0,0.08,0.45}
\definecolor{refkey}{rgb}{0,0.6,0.0}
\definecolor{Brown}{rgb}{0.45,0.0,0.05}
\definecolor{lime}{rgb}{0.00,0.8,0.0}
\definecolor{lblue}{rgb}{0.5,0.5,0.99}
\definecolor{lblue}{rgb}{0.8,0.85,1.00}
\definecolor{anotherblue}{rgb}{.8, .8,1}
\definecolor{violet}{rgb}{0.9,0.6,0.9}
\definecolor{greenyellow}{rgb}{0.53,0.99,0.18}
\definecolor{Lyellow}{rgb}{0.99,0.99,0.87}
\definecolor{Lgray}{rgb}{0.93,0.93,0.93}

\usepackage{mathpazo}
\usepackage{amsmath}
\usepackage{amssymb}
\usepackage{theorem}

\usepackage{rotating}
\usepackage{graphicx}
\usepackage{graphics}
\usepackage{longtable}
\usepackage{fancybox}
\usepackage{mathrsfs}
\usepackage{arcs}
\usepackage[dvips]{geometry}

\makeatletter
\newcommand{\vast}{\bBigg@{4}}
\newcommand{\Vast}{\bBigg@{5}}
\makeatother

\newcommand{\nnn}{\ensuremath{{n\in{\mathbb N}}}}

\newcommand{\thalb}{\ensuremath{\tfrac{1}{2}}}
\newcommand{\menge}[2]{\big\{{#1}~\big |~{#2}\big\}}
\newcommand{\mmenge}[2]{\bigg\{{#1}~\bigg |~{#2}\bigg\}}

\newcommand{\Menge}[2]{\left\{{#1}~\Big|~{#2}\right\}}

\newcommand{\To}{\ensuremath{\rightrightarrows}}

\newcommand{\fenv}[1]%
{\ensuremath{\,\overrightarrow{\operatorname{env}}_{#1}}}
\newcommand{\benv}[1]%
{\ensuremath{\,\overleftarrow{\operatorname{env}}_{#1}}}

\newcommand{\scal}[2]{\left\langle{#1},{#2}  \right\rangle}
\newcommand{\tscal}[2]{\langle{#1},{#2}  \rangle}
\newcommand{\bscal}[2]{\big\langle{#1},{#2}  \big\rangle}

\newcommand{\exi}{\ensuremath{\exists\,}}
\newcommand{\zeroun}{\ensuremath{\left]0,1\right[}}
\newcommand{\RR}{\ensuremath{\mathbb R}}
\newcommand{\ZZ}{\ensuremath{\mathbb Z}}
\newcommand{\RP}{\ensuremath{\mathbb R}_+}
\newcommand{\RM}{\ensuremath{\mathbb R}_-}
\newcommand{\RPP}{\ensuremath{\mathbb R}_{++}}
\newcommand{\RMM}{\ensuremath{\mathbb R}_{--}}

\newcommand{\RX}{\ensuremath{\,\left]-\infty,+\infty\right]}}

\newcommand{\NN}{\ensuremath{\mathbb N}}

\newcommand{\argmin}{\ensuremath{\operatorname*{argmin}}}

\newcommand{\gr}{\ensuremath{\operatorname{gr}}}

\newcommand{\reli}{\ensuremath{\operatorname{ri}}}
\newcommand{\inte}{\ensuremath{\operatorname{int}}}

\newcommand{\bd}{\ensuremath{\operatorname{bdry}}}
\newcommand{\epi}{\ensuremath{\operatorname{epi}}}

\newcommand{\ran}{\ensuremath{\operatorname{ran}}}

\newcommand{\conv}{\ensuremath{\operatorname{conv}}}
\newcommand{\cone}{\ensuremath{\operatorname{cone}}}
\newcommand{\lspan}{\ensuremath{\operatorname{span}}}
\newcommand{\aff}{\ensuremath{\operatorname{aff}}}
\newcommand{\pa}{\ensuremath{\operatorname{par}}}
\newcommand{\cconv}{\ensuremath{\overline{\operatorname{conv}}\,}}
\newcommand{\ccone}{\ensuremath{\overline{\operatorname{cone}}\,}}

\newcommand{\Id}{\ensuremath{\operatorname{Id}}}

\newcommand{\minf}{\ensuremath{-\infty}}
\newcommand{\pinf}{\ensuremath{+\infty}}

\newcommand{\ball}[2]{\operatorname{ball}({#1};{#2})}
\newcommand{\sphere}[2]{\operatorname{sphere}({#1};{#2})}

\newcommand{\wt}[1]{\widetilde{#1}}
\newcommand{\nc}[2]{N^{#2}_{#1}}
\newcommand{\fnc}[1]{N^{\text{\rm Fr\'e}}_{#1}}
\newcommand{\cnc}[1]{N^{\text{\rm conv}}_{#1}}
\newcommand{\pnX}[1]{N^{\text{\rm prox}}_{#1}} 
\newcommand{\pn}[2]{\widehat{N}^{#2}_{#1}} 

\makeatletter 
\renewcommand\p@enumii{}
\makeatother

\def\ve{\varepsilon}
\def\dd{\delta}

\def\k{{\sigma}}

\def\dn{\downarrow}

\newcommand{\mcA}{\ensuremath{\mathcal A}}
\newcommand{\mcB}{\ensuremath{\mathcal B}}

\newcommand{\sgn}{\ensuremath{\operatorname{sgn}}}

{\begin{list}{}{%
\settowidth{\labelwidth}{\textrm{#1~}}%
\setlength{\leftmargin}{\labelwidth+\labelsep}}}
{\end{list}}
\newtheorem{theorem}{Theorem}[section]
\newtheorem{lemma}[theorem]{Lemma}
\newtheorem{corollary}[theorem]{Corollary}
\newtheorem{proposition}[theorem]{Proposition}
\newtheorem{definition}[theorem]{Definition}
\theoremstyle{plain}{\theorembodyfont{\rmfamily}
}
\theoremstyle{plain}{\theorembodyfont{\rmfamily}
}
\theoremstyle{plain}{\theorembodyfont{\rmfamily}
}
\theoremstyle{plain}{\theorembodyfont{\rmfamily}
\newtheorem{example}[theorem]{Example}}
\newtheorem{fact}[theorem]{Fact}
\theoremstyle{plain}{\theorembodyfont{\rmfamily}
\newtheorem{remark}[theorem]{Remark}}

\renewcommand{\emptyset}{\varnothing}
\renewcommand{\limsup}{\varlimsup}

\newcommand{\boxedeqn}[1]{%
    \[\fbox{%
        \addtolength{\linewidth}{-2\fboxsep}%
        \addtolength{\linewidth}{-2\fboxrule}%
        \begin{minipage}{\linewidth}%
        \begin{equation}#1\\[+4mm]\end{equation}%
        \end{minipage}%
      }\]%
  }

\newcounter{count}

\allowdisplaybreaks 

\begin{document}

\title{\textrm{
Restricted normal cones and the \\method of alternating
projections}}

\author{
Heinz H.\ Bauschke\thanks{Mathematics, University of British
Columbia, Kelowna, B.C.\ V1V~1V7, Canada. E-mail:
\texttt{heinz.bauschke@ubc.ca}.},
D.\ Russell Luke\thanks{Institut f\"ur Numerische und Angewandte Mathematik,\
Universit\"at G\"ottingen,\ Lotzestr.~16--18, 37083 G\"ottingen, Germany. E-mail: \texttt{r.luke@math.uni-goettingen.de}.}, Hung M.\
Phan\thanks{Mathematics, University of British Columbia, Kelowna,
B.C.\ V1V~1V7, Canada. E-mail:  \texttt{hung.phan@ubc.ca}.}, ~and
Xianfu\ Wang\thanks{Mathematics, University of British Columbia,
Kelowna, B.C.\ V1V~1V7, Canada. E-mail:
\texttt{shawn.wang@ubc.ca}.}}

\date{May 2, 2012}

\maketitle

\vskip 8mm

\begin{abstract} \noindent
The method of alternating projections (MAP) is a common 
method for solving feasibility problems. While employed traditionally
to subspaces or to convex sets, little was known about the
behavior of the MAP in the nonconvex case until 2009, when Lewis,
Luke, and Malick derived local linear convergence results provided
that a condition involving normal cones holds and at least one of the
sets is superregular (a property less restrictive than convexity).
However, their results failed to capture very simple classical convex
instances such as two lines in three-dimensional space.

In this paper, we extend and develop the Lewis-Luke-Malick framework
so that not only any two linear subspaces but also any two closed
convex sets whose relative interiors meet are covered. We also allow
for sets that are more structured such as unions of convex sets.
The key tool required is the restricted normal cone, which is a
generalization of the classical Mordukhovich normal cone. We
thoroughly study restricted normal cones from the viewpoint of
constraint qualifications and regularity. Numerous examples are provided
to illustrate the theory.
\end{abstract}

{\small \noindent {\bfseries 2010 Mathematics Subject
Classification:} {Primary 49J52, 49M20;
Secondary 47H09, 65K05, 65K10, 90C26.
}}

\noindent {\bfseries Keywords:}
Constraint qualification,
convex set,
Friedrichs angle,
linear convergence,
method of alternating projections,
normal cone,
projection operator,
restricted normal cone,
superregularity.


\section{Introduction}

Throughout this paper, we assume that

\boxedeqn{
\text{$X$ is a Euclidean space
}
}
(i.e., finite-dimensional real Hilbert space)
with inner product $\scal{\cdot}{\cdot}$, induced norm $\|\cdot\|$,
and induced metric $d$.

Let $A$ and $B$ be nonempty closed subsets of $X$.
We assume first that $A$ and $B$ are additionally \emph{convex}
and that $A\cap B\neq\varnothing$.
In this case, the \emph{projection operators} $P_A$ and $P_B$
(a.k.a.\ projectors or nearest point mappings)
corresponding to $A$ and $B$, respectively,
are single-valued with full domain.
In order to find a point in the intersection $A$ and $B$,
it is very natural to simply alternate the operator $P_A$ and $P_B$
resulting in the famous \emph{method of alternating projections
(MAP)}. Thus, given a starting point $b_{-1}\in X$,
sequences $(a_n)_\nnn$ and $(b_n)_\nnn$ are generated as
follows:
\begin{equation}
(\forall\nnn)\qquad
a_{n} := P_Ab_{n-1},\quad b_n := P_Ba_n.
\end{equation}
In the present consistent convex setting, both sequences have
a common limit in $A\cap B$.
Not surprisingly, because of its elegance and usefulness,
the MAP has attracted many famous mathematicians, including
John~von~Neumann and Norbert~Wiener and it has been independently
rediscovered repeatedly. It is out of scope of this article to
review the history of the MAP, its many extensions, and its rich
and convergence theory;
the interested reader is referred to, e.g.,
\cite{BC2011},
\cite{CensorZenios},
\cite{Deutsch},
and the references therein.

Since $X$ is finite-dimensional and $A$ and $B$ are closed,
the convexity of $A$ and $B$ is actually not needed in order
to guarantee existence of nearest points. This gives rise
to \emph{set-valued} projection operators which for convenience
we also denote by $P_A$ and $P_B$. Dropping the convexity assumption,
the MAP now generates sequences via
\begin{equation}
(\forall\nnn)\qquad
a_{n} \in P_Ab_{n-1},\quad b_n \in P_Ba_n.
\end{equation}
This iteration is much less understood than its much older convex cousin.
For instance, global convergence to a point in $A\cap B$ cannot be
guaranteed anymore \cite{CombTrus90}.
Nonetheless, the MAP is widely applied to applications in engineering 
and the physical sciences for finding a point in $A\cap B$
(see, e.g., \cite{StarkYang}).
Lewis, Luke, and Malick achieved a break-through result in 2009,
when there are no normal vectors that are opposite
and at least one of the sets is superregular
(a property less restrictive than convexity).
Their proof techniques were quite different from the well known convex
approaches; in fact, the Mordukhovich normal cone was a central tool
in their analysis.
However, their results were not strong enough to handle
well known convex and linear scenarios. For instance, the linear
convergence of the MAP for two lines in $\RR^3$ cannot be obtained
in their framework.

\emph{The goal of this paper is to extend the results by Lewis, Luke
and Malick to make them applicable in more general settings. We unify their
theory with classical convex convergence results. Our principal tool
is a new normal cone which we term the \emph{restricted normal cone}.
A careful study of restricted normal cones and their applications is
carried out.
We also allow for constraint sets that are \emph{unions} of
superregular (or even convex) sets. We shall  recover the known
optimal convergence rate for the MAP when studying two linear subspaces.
}  In a parallel paper \cite{BLPW12b} we apply the tools developed here
to the important problem of sparsity optimization with affine constraints.  

The remainder of the paper is organized as follows.
In Section~\ref{s:prelim}, we collect various auxiliary results
that are useful later and to make the later analysis less cluttered.
The restricted normal cones are introduced in
Section~\ref{s:normalcone}.
Section~\ref{s:normalandaffine} focuses on normal cones
that are restricted by affine subspaces; the results achieved
are critical in the inclusion of convex settings to the linear
convergence framework.
Further examples and results are provided
in Section~\ref{s:further} and Section~\ref{s:nosuper},
where we illustrate that the restricted normal cone cannot be obtained
by intersections with various natural conical supersets.
Section~\ref{s:CQ1} and Section~\ref{s:CQ2} are devoted
to constraint qualifications which describe how
well the sets $A$ and $B$ relate to each other.
In Section~\ref{s:superregularity}, we discuss regularity and
superregularity, notions that extend the idea of convexity,
for sets and collections of sets.
We are then in a position to provide
in Section~\ref{s:application}
our main results dealing with
the local linear convergence of the MAP.

\subsection*{Notation}

The notation employed in this article
is quite standard and follows largely
\cite{Borzhu05},
\cite{Boris1},
\cite{Rock70},
and \cite{Rock98};
these books also provide exhaustive information on variational
analysis.
The real numbers are $\RR$,
the integers are $\ZZ$, and
$\NN := \menge{z\in\ZZ}{z\geq 0}$.
Further, $\RP := \menge{x\in\RR}{x\geq 0}$,
$\RPP := \menge{x\in \RR}{x>0}$ and
$\RM$ and $\RMM$ are defined analogously.
Let $R$ and $S$ be subsets of $X$.
Then the closure of $S$ is $\overline{S}$,
the interior of $S$ is $\inte(S)$,
the boundary of $S$ is $\bd(S)$,
and the smallest affine and linear subspaces containing $S$
are $\aff S$ and $\lspan S$, respectively.
The linear subspace parallel to $\aff S$ is
$\pa S := (\aff S)-S=(\aff S)-s$, for every $s\in S$.
The relative interior of $S$, $\reli(S)$, is the interior of $S$
relative to $\aff(S)$.
The negative polar cone of $S$ is
$S^\ominus=\menge{u\in X}{\sup\scal{u}{S}\leq 0}$.
We also set $S^\oplus := -S^\ominus$
and $S^\perp := S^\oplus \cap S^\ominus$.
We also write
$R\oplus S$ for $R+S:=\menge{r+s}{(r,s)\in R\times S}$
provided that $R\perp S$, i.e., $(\forall (r,s)\in R\times S)$
$\scal{r}{s}=0$.
We write $F\colon X\To X$, if $F$ is a mapping from $X$ to its power set,
i.e., $\gr F$, the graph of $F$, lies in $X\times X$.
Abusing notation slightly, we will write $F(x) = y$ if $F(x)=\{y\}$.
A nonempty subset $K$ of $X$ is a cone
if $(\forall\lambda\in\RP)$ $\lambda K := \menge{\lambda k}{k\in
K}\subseteq K$.
The smallest cone containing $S$ is denoted $\cone(S)$;
thus, $\cone(S) := \RP\cdot S := \menge{\rho s}{\rho\in\RP,s\in S}$
if $S\neq\varnothing$ and $\cone(\varnothing):=\{0\}$.
The smallest convex and closed and convex subset containing $S$
are $\conv(S)$ and $\cconv(S)$, respectively.
If $z\in X$ and $\rho\in\RPP$, then
$\ball{z}{\rho} := \menge{x\in X}{d(z,x)\leq \rho}$ is
the closed ball centered at $z$ with radius $\rho$ while
$\sphere{z}{\rho} := \menge{x\in X}{d(z,x)= \rho}$ is
the (closed) sphere centered at $z$ with radius $\rho$.
If $u$ and $v$ are in $X$,
then $[u,v] := \menge{(1-\lambda)u+\lambda v}{\lambda\in [0,1]}$ is
the line segment connecting $u$ and $v$.

\section{Auxiliary results}

\label {s:prelim}

In this section, we fix some basic notation used throughout this
article. We also collect several auxiliary results that will be
useful in the sequel.

\subsection*{Projections}

\begin{definition}[distance and projection]
Let $A$ be a nonempty subset of $X$.
Then
\begin{equation}
d_A \colon X\to\RR\colon x\mapsto \inf_{a\in A}d(x,a)
\end{equation}
is the \emph{distance function} of the set $A$ and
\begin{equation}
P_A\colon X\To X\colon
x\mapsto \menge{a\in A}{d_A(x)=d(x,a)}
\end{equation}
is the corresponding \emph{projection}.
\end{definition}

\begin{proposition}[existence]
\label{p:0224a}
Let $A$ be a nonempty closed subset of $X$.
Then $(\forall x\in X)$ $P_A(x)\neq\emptyset$.
\end{proposition}
\begin{proof}
Let $z\in X$. The function
$f\colon X\to\RR\colon x\mapsto \|x-z\|^2$ is continuous and
$\lim_{\|x\|\to\pinf} f(x)=\pinf$.
Let $(x_n)_\nnn$ be a sequence in $A$ such that
$f(x_n)\to\inf f(A)$. Then $(x_n)_\nnn$ is bounded.
Since $A$ is closed and $f$ is continuous, every cluster point of
$(x_n)_\nnn$ is a minimizer of $f$ over the set $A$, i.e., an element in
$P_Az$.
\end{proof}
\begin{example}[sphere]
\label{ex:projS}
Let $z\in X$ and $\rho\in\RPP$.
Set $S := \sphere{z}{\rho}$.
Then
\begin{equation}
\label{e:0308a}
(\forall x\in X)\quad P_S(x)=\begin{cases}
z+\rho\frac{x-z}{\|x-z\|},&\text{if $x\neq z$;}\\
S, &\text{otherwise.}
\end{cases}
\end{equation}
\end{example}
\begin{proof} 
Let $x\in X$. The formula is clear when $x=z$, so we assume $x\neq z$.
Set
\begin{equation}
c := z+\rho\frac{x-z}{\|x-z\|} \in S,
\end{equation}
and let $s = z+\rho b\in S\smallsetminus \{c\}$,
i.e., $\|b\|=1$ and $b\neq (x-z)/\|x-z\|$.
Hence, using that $|\|u\|-\|v\||<\|u-v\|$
$\Leftrightarrow$ $\scal{u}{v}<\|u\|\|v\|$ and because of Cauchy-Schwarz,
we obtain
\begin{subequations}
\begin{align}
\|x-c\| &= \big|\|x-z\|-\rho\big|
=\big|\|x-z\|-\|\rho b\|\big|
=\big|\|x-z\|-\|s-z\|\big|\\
&<\|x-s\|.
\end{align}
\end{subequations}
We have thus established \eqref{e:0308a}.
\end{proof}

In view of Proposition~\ref{p:0224a},
the next result is in particular applicable to
the union of finitely many nonempty closed subsets of $X$.

\begin{lemma}[union]
\label{l:unionproj}
Let $(A_i)_{i\in I}$ be a collection of nonempty subsets of $X$,
set $A := \bigcup_{i\in I}A_i$,
let $x\in X$, and suppose that $a\in P_A(x)$.
Then there exists $i\in I$ such that $a\in P_{A_i}(x)$.
\end{lemma}
\begin{proof}
Indeed, since $a\in A$, there exists $i\in I$ such that
$a\in A_i$.
Then
$d(x,a)=d_A(x)\leq d_{A_i}(x)\leq d(x,a)$.
Hence $d(x,a)=d_{A_i}(x)$, as claimed.
\end{proof}

The following result is well known.

\begin{fact}[projection onto closed convex set]
\label{f:convproj}
Let $C$ be a nonempty closed convex subset of $X$,
and let $x$, $y$ and $p$ be in $X$.
Then the following hold:
\begin{enumerate}
\item
\label{f:convproj1}
$P_C(x)$ is a singleton.
\item
\label{f:convproj2}
$P_C(x)=p$ if and only if $p\in C$ and $\sup\scal{C-p}{x-p}\leq 0$.
\item
\label{f:convproj3}
$\|P_C(x)-P_C(y)\|^2 + \|(\Id-P_C)(x)-(\Id-P_C)(y)\|^2\leq\|x-y\|^2$.
\item
\label{f:convproj4}
$\|P_C(x)-P_C(y)\|\leq\|x-y\|$.
\end{enumerate}
\end{fact}
\begin{proof}
\ref{f:convproj1}\&\ref{f:convproj2}:
\cite[Theorem~3.14]{BC2011}.
\ref{f:convproj3}:
\cite[Proposition~4.8]{BC2011}.
\ref{f:convproj4}:
Clear from \ref{f:convproj3}.
\end{proof}


\subsection*{Miscellany}

\begin{lemma}\label{l:cone01}
Let $A$ and $B$ be subsets of $X$, and let
$K$ be a cone in $X$.
Then the following hold:
\begin{enumerate}
\item\label{l:cone01i}
$\cone(A\cap B)\subseteq\cone A\cap \cone B$.
\item\label{l:cone01ii} $\cone(K\cap B)= K \cap \cone B$.
\end{enumerate}
\end{lemma}
\begin{proof}
\ref{l:cone01i}: Clear.
\ref{l:cone01ii}:
By \ref{l:cone01i}, $\cone(K\cap B)
\subseteq (\cone K) \cap (\cone B) = K \cap \cone B$.
Now assume that $x\in (K\cap \cone B)\smallsetminus\{0\}$.
Then there exists $\beta>0$ such that $x/\beta \in B$.
Since $K$ is a cone, $x/\beta\in K$.
Thus $x/\beta\in K\cap B$ and therefore $x \in \cone(K\cap B)$.
\end{proof}

Note that the inclusion in
Lemma~\ref{l:cone01}\ref{l:cone01i} may be strict:
indeed, consider the case when
$X=\RR$, $A := \{1\}$, and $B=\{2\}$.


\begin{lemma}[a characterization of convexity]
\label{l:PA&cvex}
Let $A$ be a nonempty closed subset of $X$.
Then the following are equivalent:
\begin{enumerate}
\item\label{l:PA&cvex-i} $A$ is convex.
\item\label{l:PA&cvex-ii} $P^{-1}_A(a)-a$ is a cone,
for every $a\in A$.
\item\label{l:PA&cvex-iii}
$P_A(x)$ is a singleton, for every $x\in X$.
\end{enumerate}
\end{lemma}
\begin{proof}
``\ref{l:PA&cvex-i}$\Rightarrow$\ref{l:PA&cvex-ii}'':
Indeed, it is well known in convex analysis (see, e.g.,
\cite[Proposition~6.17]{Rock98}) that for
every $a\in A$,
$P_A^{-1}(a)-a$ is equal to the normal cone (in the sense of convex
analysis) of $A$ at $a$.

``\ref{l:PA&cvex-ii}$\Rightarrow$\ref{l:PA&cvex-iii}'':
Let $x\in X$.
By Proposition~\ref{p:0224a}, $P_Ax\neq\varnothing$.
Take $a_1$ and $a_2$ in $P_Ax$.
Then $\|x-a_1\|=\|x-a_2\|$ and
$x-a_1\in P_A^{-1}a_1-a_1$.
Since $P_A^{-1}a-a$ is a cone, we have
$2(x-a_1)\in P_A^{-1}a_1-a_1$.
Hence $y := 2x-a_1\in P_A^{-1}a_1$
and $y-x=x-a_1$.
Thus,
\begin{subequations}
\begin{align}
\scal{y-a_2}{a_1-a_2}&=\scal{(y-x)+(x-a_2)}{(a_1-x)+(x-a_2)}\\
&=\scal{y-x}{a_1-x}+\scal{y-x}{x-a_2}+\scal{x-a_2}{a_1-x}+\|x-a_2\|^2\\
&=\scal{x-a_1}{a_1-x}+\scal{x-a_1}{x-a_2}+\scal{x-a_2}{a_1-x}+\|x-a_2\|^2\\
&=-\|x-a_1\|^2 + \|x-a_2\|^2\\
&=0.
\end{align}
\end{subequations}
Since $a_1\in P_Ay$, it follows that
\begin{subequations}
\label{e:0224a}
\begin{align}
\|y-a_1\|^2&=\|y-a_2\|^2+2\scal{y-a_2}{a_2-a_1}+\|a_1-a_2\|^2\\
&=\|y-a_2\|^2+\|a_1-a_2\|^2\\
&\geq \|y-a_2\|^2\\
&\geq\|y-a_1\|^2.
\end{align}
\end{subequations}
Hence equality holds throughout \eqref{e:0224a}.
Therefore, $a_1=a_2$.

``\ref{l:PA&cvex-iii}$\Rightarrow$\ref{l:PA&cvex-i}``:
This classical result due to Bunt and to Motzkin
on the convexity of Chebyshev sets is well known;
for proofs, see, e.g., \cite[Chapter~12]{Deutsch} or
\cite[Corollary~21.13]{BC2011}.
\end{proof}



\begin{proposition}
\label{p:onecone} Let $S$ be a convex set. Then the following are
equivalent.
\begin{enumerate}
\item
\label{p:oneconei} $0\in\reli S$.
\item
\label{p:oneconeii} $\cone S = \lspan S$.
\item
\label{p:oneconeiii} $\ccone S = \lspan S$.
\end{enumerate}
\end{proposition}
\begin{proof}
Set $Y = \lspan S$. Then \ref{p:oneconei} $\Leftrightarrow$ $0$
belongs to the interior of $S$ relative to $Y$.

``\ref{p:oneconei}$\Rightarrow$\ref{p:oneconeii}'': There exists
$\dd>0$ such that for every $y\in Y\smallsetminus\{0\}$, $\dd
y/\|y\| \in S$. Hence $y\in\cone S$.

``\ref{p:oneconeii}$\Rightarrow$\ref{p:oneconei}'': For every $y\in
Y$, there exists $\dd>0$ such that $\dd y\in S$. Now
\cite[Corollary~6.4.1]{Rock70} applies in $Y$.

``\ref{p:oneconeii}$\Leftrightarrow$\ref{p:oneconeiii}'': Set
$K=\cone S$, which is convex. By \cite[Corollary~6.3.1]{Rock70}, we
have $\reli K = \reli Y$ $\Leftrightarrow$ $\overline{K} =
\overline{Y}$ $\Leftrightarrow$ $\reli Y \subseteq K \subseteq
\overline{Y}$. Since $\reli Y = Y = \overline{Y}$, we obtain the
equivalences: $\reli K = Y$ $\Leftrightarrow$ $\overline{K}=Y$
$\Leftrightarrow$ $K=Y$.
\end{proof}

\section{Restricted normal cones: basic properties}
\label{s:normalcone}

Normal cones are fundamental objects in variational analysis;
they are used to construct subdifferential operators, and they
have found many applications in
optimization, optimal control, nonlinear analysis, convex analysis,
etc.; see, e.g.,
\cite{BC2011},
\cite{Borzhu05},
\cite{clsw98},
\cite{Loewenbook}, 
\cite{Boris1},
\cite{Rock70},
\cite{Rock98}.
One of the key building blocks is the Mordukhovich (or limiting)
normal cone $N_{A}$, which is obtained by limits of
proximal normal vectors.
In this section, we propose a new, very flexible,
normal cone of $A$, denoted by $\nc{A}{B}$, by
constraining the proximal normal vectors to a set $B$.

\begin{definition}[normal cones]
\label{d:NCone}
Let $A$ and $B$ be nonempty subsets of $X$, and let $a$ and
$u$ be in $X$.
If $a\in A$, then various normal cones of $A$ at $a$ are defined as follows:
\begin{enumerate}
\item
\label{d:pnB} The \emph{$B$-restricted proximal normal cone} of $A$
at $a$ is
\begin{equation}\label{e:pnB}
\pn{A}{B}(a):= \cone\Big(\big(B\cap P_A^{-1}a\big)-a\Big) =
\cone\Big(\big(B-a\big)\cap \big(P_A^{-1}a-a\big)\Big).
\end{equation}
\item
\label{d:pnX} The (classical) \emph{proximal normal cone} of $A$ at
$a$ is
\begin{equation}\label{e:pnX}
\pnX{A}(a):=\pn{A}{X}(a)= \cone\big(P_A^{-1}a-a\big).
\end{equation}
\item\label{d:nc}
The \emph{$B$-restricted normal cone} $\nc{A}{B}(a)$ is implicitly
defined by $u\in\nc{A}{B}(a)$ if and only if there exist sequences
$(a_n)_\nnn$ in $A$ and $(u_n)_\nnn$ in $\pn{A}{B}(a_n)$ such that
$a_n\to a$ and $u_n\to u$.
\item\label{d:fnc} The \emph{Fr\'{e}chet normal cone}
$\fnc{A}(a)$ is implicitly defined by $u\in \fnc{A}(a)$ if and only
if $(\forall \ve>0)$ $(\exi\dd>0)$ $(\forall x\in A\cap
\ball{a}{\delta})$ $\scal{u}{x-a}\leq\ve\|x-a\|$.
\item\label{d:cnc}
The \emph{normal convex from convex analysis} $\cnc{A}(a)$ is
implicitly defined by $u\in\cnc{A}(a)$ if and only if
$\sup\scal{u}{A-a}\leq 0$.
\item\label{d:Mnc}
The \emph{Mordukhovich normal cone} $N_A(a)$ of $A$ at $a$ is
implicitly defined by $u\in N_A(a)$ if and only if there exist
sequences $(a_n)_\nnn$ in $ A$ and $(u_n)_\nnn$ in $\pnX{A}(a_n)$
such that $a_n\to a$ and $u_n\to u$.
\end{enumerate}
If $a\notin A$, then all normal cones are defined to be empty.
\end{definition}

\bigskip

\begin{center}
\psset{xunit=1.2cm, yunit=1.0cm}
\begin{pspicture}
(-3,-4)(3.2,1.5)

\pscustom{
\psline[showpoints=false]{-}(-3,1.5)(0,0)(3,1.5)
\gsave
\psline(0,1.5)
\fill[fillstyle=solid,fillcolor=Lyellow]
\grestore}
\rput(-0.5,1){\psframebox*[framearc=.4]{$A$}}
\rput(0,0.3){\psframebox*[framearc=.4]{$a$}}

\pscustom{
\psline[showpoints=false]{<->}(-2,-4)(0,0)(2,-4)
\gsave
\psline(1,-4)
\fill[fillstyle=solid,opacity=0.2,fillcolor=Lgray]
\grestore}
\rput(0,-2.5){\psframebox*[framearc=.4]{$\pnX{A}(a)$}}

\psline[linecolor=gray,showpoints=false]{-}(-0.3,0.15)(-0.45,-0.15)(-0.15,-0.3)
\psline[linecolor=gray,showpoints=false]{-}(0.3,0.15)(0.45,-0.15)(0.15,-0.3)


\rput(-2,-1){\begin{tabular}{c} The proximal\\ normal cone \end{tabular}}
\psarcn{->}(-1,-1.6){1}{90}{0}

\end{pspicture}
\begin{pspicture}
(-3,-4)(4,1.5)

\pscustom{
\psline[showpoints=false]{-}(-3,1.5)(0,0)(3,1.5)
\gsave
\psline(0,1.5)
\fill[fillstyle=solid,fillcolor=Lyellow]
\grestore}
\rput(-0.5,1){\psframebox*[framearc=.4]{$A$}}
\rput(0,0.3){\psframebox*[framearc=.4]{$a$}}

\pscustom{
\psline[showpoints=false]{<->}(-2,-4)(0,0)(2,-4)
\gsave
\psline(1,-4)
\fill[fillstyle=solid,opacity=0.2,fillcolor=Lgray]
\grestore}

\psline[linecolor=gray,showpoints=false]{-}(-0.3,0.15)(-0.45,-0.15)(-0.15,-0.3)
\psline[linecolor=gray,showpoints=false]{-}(0.3,0.15)(0.45,-0.15)(0.15,-0.3)

\pscustom{
\psline[linecolor=red,showpoints=false]{<->}(0.5,-4)(0.02,-0.05)(2,-4)

\gsave
\psline(1,-4)
\fill[fillstyle=solid,opacity=0.2,fillcolor=lightgray,linestyle=none]
\grestore}
\rput(1.1,-3.5){\psframebox*[framearc=.4]{$\pn{A}{B}(a)$}}

\psellipse[fillstyle=solid,fillcolor=greenyellow,opacity=0.2](1.565,-2.5)(1.25,0.3)
\rput(2,-2.5){\psframebox*[framearc=.4]{$B$}}
\psline[linecolor=red,showpoints=false]{<->}(0.5,-4)(0.02,-0.05)(2,-4)


\rput(2.3,-0.5){\begin{tabular}{c} The restricted\\ proximal normal cone \end{tabular}}
\psarc{->}(1.1,-1){1}{90}{180}

\rput(2.50,-1.7){$P^{-1}_A(a)\cap B$}
\psarc{->}(1.5,-2.5){1}{90}{180}

\end{pspicture}
\end{center}



\begin{remark}
Some comments regarding Definition~\ref{d:NCone} are in order.
\begin{enumerate}
\item Clearly, the restricted proximal normal cone generalizes
the notion of the classical proximal normal cone. The name
``restricted'' stems from the fact that the pre-image $P_A^{-1}a$ is
restricted to the set $B$.
\item See \cite[Example~6.16]{Rock98} and
\cite[Subsection~2.5.2.D on page~240]{Boris1} for further
information regarding the classical proximal normal cone, including
the fact that
\begin{equation}
\label{e:120209a} u\in\pnX{A}(a) \quad\Leftrightarrow\quad a\in
A\;\text{and}\; (\exi\delta>0) (\forall x\in A)\;\;
\scal{u}{x-a}\leq\delta\|x-a\|^2.
\end{equation}
This also implies that: $\pnX{A}(a)+(A-a)^\ominus\subseteq\pnX{A}(a)$.
\item
Note that $\gr\nc{A}{B} = (A\times X)\cap \overline{\gr\pn{A}{B}}$.
Put differently, $\nc{A}{B}(a)$ is the outer (or upper Kuratowski)
limit of $\pn{A}{B}(x)$ as $x\to a$ in $A$, written
\begin{equation}
\label{e:0224b}
\nc{A}{B}(a)= \limsup_{x\to a\atop x\in A}\pn{A}{B}(x).
\end{equation}
See also \cite[Chapter~4]{Rock98}.
\item
See \cite[Definition~1.1]{Boris1} or \cite[Definition~6.3]{Rock98}
(where this is called the regular normal cone) for further
information regarding $\fnc{A}(a)$.
\item
The Mordukhovich normal cone 
is also known as
the basic or limiting normal cone. Note that $N_A=\nc{A}{X}$ and
$\gr N_A = (A\times X)\cap \overline{\gr \pn{A}{X}} = (A \times
X)\cap \overline{\gr\pnX{A}}$ and once again $N_A(a)$ is the outer
(or upper Kuratowski) limit of $\pn{A}{X}(x)$ or $\pnX{A}(x)$ as
$x\to a$ in $A$. See also \cite[page~141]{Boris1} for historical
notes.
\end{enumerate}
\end{remark}

The next result presents useful characterizations of the Mordukhovich
normal cone.

\begin{proposition}[characterizations of the Mordukhovich normal cone]
\label{p:Nequi} Let $A$ be a nonempty closed subset of $X$, let
$a\in A$, and let $u\in X$. Then the following are equivalent:
\begin{enumerate}
\item\label{p:Ne-i}
$u\in N_A(a)$.
\item\label{p:Ne-ii}
There exist sequences $(\lambda_n)_\nnn$ in $\RP$, $(b_n)_\nnn$ in
$X$, $(a_n)_\nnn$ in $A$ such that $a_n\to a$,
$\lambda_n(b_n-a_n)\to u$, and $(\forall\nnn)$ $a_n\in P_Ab_n$.
\item
\label{p:Ne-iii} There exist sequences $(\lambda_n)_\nnn$ in $\RP$,
$(x_n)_\nnn$ in $X$, $(a_n)_\nnn$ in $A$ such that $x_n\to a$,
$\lambda_n(x_n-a_n)\to u$, and $(\forall\nnn)$ $a_n\in P_Ax_n$.
(This also implies $a_n\to a$.)
\item
\label{p:Ne-iv} There exist sequences $(a_n)_\nnn$ in $A$ and
$(u_n)_\nnn$ in $X$ such that $a_n\to a$, $u_n\to u$, and
$(\forall\nnn)$ $u_n\in\fnc{A}(a_n)$.
\end{enumerate}
\end{proposition}
\begin{proof}
``\ref{p:Ne-i}$\Leftrightarrow$\ref{p:Ne-ii}'': Clear from
Definition~\ref{d:NCone}\ref{d:Mnc}.

``\ref{p:Ne-iii}$\Leftrightarrow$\ref{p:Ne-iv}'': Noting that the
definition of $N_A(a)$ in \cite{Boris1} is the one given in
\ref{p:Ne-iv}, we see that this equivalence follows from
\cite[Theorem~1.6]{Boris1}.

``\ref{p:Ne-ii}$\Rightarrow$\ref{p:Ne-iii}'': Let
$(\lambda_n)_\nnn$, $(a_n)_\nnn$, and $(b_n)_\nnn$ be as in
\ref{p:Ne-ii}. For every $\nnn$, since $a_n\in P_Ab_n$,
\cite[Example~6.16]{Rock98} implies that $a_n\in P_A[a_n,b_n]$. Now
let $(\ve_n)_\nnn$ be a sequence in $\zeroun$ such that $\ve_n
a_n\to 0$ and $\ve_n b_n\to 0$. Set
\begin{equation}
(\forall\nnn)\quad x_n=(1-\ve_n)a_n+\ve_nb_n=a_n+\ve_n(b_n-a_n) \in
[a_n,b_n].
\end{equation}
Then $x_n\to a$ and $(\forall\nnn)$ $a_n\in P_Ax_n$. Furthermore,
$(\lambda_n/\ve_n)_\nnn$ lies in $\RP$ and
\begin{equation}
(\lambda_n/\ve_n)(x_n-a_n) = \lambda_n(b_n-a_n)\to u.
\end{equation}

``\ref{p:Ne-iii}$\Rightarrow$\ref{p:Ne-ii}'': Let
$(\lambda_n)_\nnn$, $(x_n)_\nnn$, and $(a_n)_\nnn$ be as in
\ref{p:Ne-iii}. Since $x_n\to a$ and $a\in A$, we deduce that $0\leq
\|x_n-a_n\|=d_A(x_n)\leq \|x_n-a\|\to 0$. Hence $x_n-a_n\to 0$ which
implies that $a_n-a=a_n-x_n+x_n-a\to 0+0=0$. Therefore,
\ref{p:Ne-ii} holds with $(b_n)_\nnn=(x_n)_\nnn$.
\end{proof}

Here are some basic properties of the restricted normal cone and its
relation to various classical cones.

\begin{lemma}[basic inclusions among the normal cones]
\label{l:NsubsetN}
Let $A$ and $B$ be nonempty subsets of $X$, and let $a\in A$.
Then the following hold:
\begin{enumerate}
\item
\label{l:NsubsetNi} $\cnc{A}(a)\subseteq\pnX{A}(a)$.
\item
\label{l:NsubsetNi+}
$\pn{A}{B}(a) =
\cone((B-a)\cap(P_A^{-1}a-a))\subseteq
(\cone(B-a))\cap\pnX{A}(a)$.
\item
\label{l:NsubsetNii} $\pn{A}{B}(a)\subseteq\pn{A}{X}(a)=\pnX{A}(a)$
and $\nc{A}{B}(a)\subseteq N_A(a)$.
\item
\label{l:NsubsetNiv} $\pn{A}{B}(a)\subseteq\nc{A}{B}(a)$.
\item
\label{l:NsubsetNiii} If $A$ is closed, then $\pnX{A}(a) \subseteq
\fnc{A}(a)$.
\item
\label{l:NsubsetNv} If $A$ is closed, then $\fnc{A}(a)\subseteq
N_A(a)$.
\item
\label{l:NsubsetNvi} If $A$ is closed and convex, then
$\pn{A}{X}(a)=\pnX{A}(a)=\fnc{A}(a)=\cnc{A}(a)=N_A(a)$.
\item
\label{l:NsubsetNri} If $a\in\reli(A)$, then
$\pn{A}{\aff(A)}(a)=\nc{A}{\aff(A)}(a)=\{0\}$.
\item
\label{l:NsubsetNvii} $(\aff(A)-a)^\bot\subseteq (A-a)^\ominus$.
\item
\label{l:NsubsetNviii} $(A-a)^\ominus\cap \cone(B-a) \subseteq
\pn{A}{B}(a)\subseteq \cone(B-a)$.
\end{enumerate}
\end{lemma}
\begin{proof}
\ref{l:NsubsetNi}: Take $u\in\cnc{A}(a)$ and fix an arbitrary
$\delta>0$. Then $(\forall x\in A)$ $\scal{u}{x-a}\leq 0 \leq
\delta\|x-a\|^2$. In view of \eqref{e:120209a}, $u\in\pnX{A}(a)$.

\ref{l:NsubsetNi+}:
In view of Lemma~\ref{l:cone01}, the definitions yield
\begin{subequations}
\begin{align}
\pn{A}{B}(a) &= \cone\big((B\cap P_A^{-1}a)-a\big)
= \cone\big((B-a)\cap (P_A^{-1}a-a)\big)\\
&\subseteq \cone\big((B-a)\cap \cone(P_A^{-1}a-a)\big)
=\cone\big((B-a)\cap\pnX{A}(a)\big)\\
&=\cone(B-a)\cap \pnX{A}(a).
\end{align}
\end{subequations}

\ref{l:NsubsetNii}, \ref{l:NsubsetNiv} and \ref{l:NsubsetNvii}: This
is obvious.

\ref{l:NsubsetNiii}: Assume that $A$ is closed and take
$u\in\pnX{A}(a)$. By \eqref{e:120209a}, there exists $\rho>0$ such
that $(\forall x\in A)$ $\scal{u}{x-a}\leq\rho\|x-a\|^2$. Now let
$\ve>0$ and set $\delta = \ve/\rho$.
If $x\in A\cap \ball{a}{\delta}$,
then $\scal{u}{x-a}\leq\rho\|x-a\|^2 \leq \rho\delta
\|x-a\|=\ve\|x-a\|$. Thus, $u\in\fnc{A}(a)$.

\ref{l:NsubsetNv}: This follows from Proposition~\ref{p:Nequi}.

\ref{l:NsubsetNvi}: Since $A$ is closed, it follows from
\ref{l:NsubsetNi}, \ref{l:NsubsetNiii}, and \ref{l:NsubsetNv} that
\begin{equation}
\cnc{A}(a) \subseteq \pnX{A}(a) \subseteq \fnc{A}(a) \subseteq
N_A(a).
\end{equation}
On the other hand, by \cite[Proposition~1.5]{Boris1},
$N_A(a)\subseteq \cnc{A}(a)$ because $A$ is convex.

\ref{l:NsubsetNri}: By assumption, $(\exi\dd>0)$
$\ball{a}{\dd}\cap \aff(A)\subseteq A$.
Hence $\aff(A)\cap P_A^{-1}a = \{a\}$ and thus
$\pn{A}{\aff(A)}(a)=\{0\}$. Since $a\in\reli(A)$, it follows that
$(\forall x\in \ball{a}{\dd/2}\cap\aff(A))$ $ \pn{A}{\aff(A)}(x)=\{0\}$.
Therefore, $\nc{A}{\aff(A)}(a)=\{0\}$.

\ref{l:NsubsetNviii}: Take $u\in
((A-a)^\ominus\cap\cone(B-a))\smallsetminus\{0\}$, say
$u=\lambda(b-a)$, where $b\in B$ and $\lambda>0$. Then $0\geq
\sup\scal{A-a}{u}=\lambda\sup\scal{A-a}{b-a}=
\sup\lambda\scal{\cconv A-a}{b-a}$.
By Fact~\ref{f:convproj}\ref{f:convproj2},
$a=P_{\cconv A}b$ and hence $a=P_Ab$. It follows that
$u\in\cone((B\cap P_A^{-1}a)-a)$. The left inclusion thus holds. The
right inclusion is clear.
\end{proof}



\begin{remark}[on closedness of normal cones]
Let $A$ be a nonempty subset of $X$, let $a\in A$, and let $B$ be a
subset of $X$. Then $\nc{A}{B}(a)$, $N_A(a)$, and $\cnc{A}(a)$ are
obviously closed---this is also true for  $\fnc{A}(a)$ but requires
some work (see \cite[Proposition~6.5]{Rock98}). On the other hand,
the classical proximal normal cone $\pnX{A}(a) = \pn{A}{X}(a)$ is
not necessarily closed (see, e.g., \cite[page~213]{Rock98}), and
hence neither is $\pn{A}{B}(a)$. For a concrete example, suppose
that $X=\RR^2$, that $A=\{(0,0)\}$, that $B=\RR\times\{1\}$ and that
$a=(0,0)$. Then $\pn{A}{B}(a)=\big(\RR\times\RPP\big)\cup\{(0,0)\}$, which is
not closed; however, the classical proximal normal cone
$\pnX{A}(a)=\RR^2$ is closed.
\end{remark}


The sphere is a nonconvex set for which all classical normal cones
coincide:

\begin{example}[classical normal cones of the sphere]
\label{ex:ncS}
Let $z\in X$ and $\rho\in\RPP$.
Set $S := \sphere{z}{\rho}$
and let $s\in S$.
Then
$\pnX{S}(s) = \pn{S}{X}(s)=\fnc{S}(s)=N_S(s)=\RR(s-z)$.
\end{example}
\begin{proof}
By Example~\ref{ex:projS}, we have
$P_S^{-1}(s) = z + \RP(s-z)$ and
so $P_S^{-1}(s)-s = \left[-1,\pinf\right[\cdot(s-z)$.
Hence, using Lemma~\ref{l:NsubsetN}\ref{l:NsubsetNiii}\&\ref{l:NsubsetNv},
we have
\begin{subequations}
\begin{align}
\pnX{S}(s) &= \pn{S}{X}(s)= \RR(s-z) \subseteq \fnc{S}(s) \subseteq
N_S(s)\\
&=\varlimsup_{s'\to S\atop s'\in S}\pnX{S}(s')
=\varlimsup_{s'\to S\atop s'\in S}\RR(s'-z) = \RR(s-z)\\
&=\pnX{S}(s),
\end{align}
\end{subequations}
as announced.
\end{proof}


Here are some elementary yet useful calculus rules.

\begin{proposition}
\label{p:elementary} Let $A$, $A_1$, $A_2$, $B$, $B_1$, and $B_2$ be
nonempty subsets of $X$, let $c\in X$,
and suppose that $a\in A\cap A_1\cap A_2$.
Then the following hold:
\begin{enumerate}
\item\label{p:ele-i}
If $A$ and $B$ are convex, then $\pn{A}{B}(a)$ is convex.
\item\label{p:ele-ii}
$\pn{A}{B_1\cup B_2}(a)=\pn{A}{B_1}(a)\cup\pn{A}{B_2}(a)$ and
$\nc{A}{B_1\cup B_2}(a)=\nc{A}{B_1}(a)\cup\nc{A}{B_2}(a)$.
\item\label{p:ele-iii}
If $B\subseteq A$, then $\pn{A}{B}(a)=\nc{A}{B}(a)=\{0\}$.
\item\label{p:ele-iv}
If $A_1\subseteq A_2$, then $\pn{A_2}{B}(a)\subseteq
\pn{A_1}{B}(a)$.
\item\label{p:ele-v}
$-\pn{A}{B}(a)=\pn{-A}{-B}(-a)$, $-\nc{A}{B}(a)=\nc{-A}{-B}(-a)$,
and $-N_{A}(a)=N_{-A}(-a)$.
\item\label{p:ele-vi}
$\pn{A}{B}(a)=\pn{A-c}{B-c}(a-c)$ and $\nc{A}{B}(a)=\nc{A-c}{B-c}(a-c)$.
\end{enumerate}
\end{proposition}
\begin{proof}
It suffices to establish the conclusions for the restricted proximal
normal cones since the restricted normal cone results follows by
taking closures (or outer limits). \ref{p:ele-i}: We assume that
$B\cap P_A^{-1}a\neq\varnothing$, for otherwise the conclusion is
clear. Then $P_A^{-1}(a) = P_{\overline{A}}^{-1}a =
(\Id+N_{\overline{A}})a$ is convex (as the image of the maximally
monotone operator $\Id+N_{\overline{A}}$ at $a$). Hence $(B\cap
P_A^{-1}a)-a$ is convex as well, and so is its conical hull, which
is $\pn{A}{B}(a)$. \ref{p:ele-ii}: Since $((B_1\cup
B_2)\cap P_A^{-1}a)-a = ((B_1\cap P_A^{-1}a)-a)\cup((B_2\cap P_A^{-1}a)-a)$, the
result follows by taking the conical hull. \ref{p:ele-iii}: Clear,
because $(B\cap P_A^{-1}a)- a$ is either empty or equal to $\{0\}$.
\ref{p:ele-iv}: Suppose $\lambda(b-a)\in \pn{A_2}{B}(a)$, where
$\lambda\geq 0$, $b\in B$, and $a\in P_{A_2}b$. Since $a\in
A_1\subseteq A_2$, we have $a\in P_{A_1}b$. Hence $\lambda(b-a)\in
\pn{A_1}{B}(a)$. \ref{p:ele-v}: This follows by using elementary
manipulations and the fact that $P_{-A} = (-\Id)\circ P_A\circ
(-\Id)$. \ref{p:ele-vi}: This follows readily from the fact that
$P^{-1}_{A-c}(a-c) = P_A^{-1}(a)-c$.
\end{proof}

\begin{remark}
\label{r:120212a} The restricted normal cone counterparts of items
\ref{p:ele-i} and \ref{p:ele-iv} are false in general; see
Example~\ref{ex:120212a} (and also
Example~\ref{ex:ice}\ref{ex:iceiii}) below.
\end{remark}


The Mordukhovich normal cone (and hence also the Clarke normal cone
which contains the Mordukhovich normal cone) strictly contains
$\{0\}$ at boundary points (see \cite[Corollary~2.24]{Boris1} or
\cite[Exercise~6.19]{Rock98});
however, the restricted normal cone can be $\{0\}$ at boundary
points as we illustrate next.

\begin{example}[restricted normal cone at boundary points]
\label{ex:cf} Suppose that $X=\RR^2$, set
$A:=\ball{0}{1}=\menge{x\in\RR^2}{\|x\|\leq1}$ and
$B:=\RR\times\{2\}$, and let $a=(a_1,a_2)\in A$.
Then
\begin{equation}
\pn{A}{B}(a)=
\begin{cases}
\RP a,&\text{if $\|a\|=1$ and $a_2>0$;}\\
\{(0,0)\},&\text{otherwise.}
\end{cases}
\end{equation}
Consequently,
\begin{equation}
\nc{A}{B}(a)=
\begin{cases}
\RP a,&\text {if $\|a\|=1$ and $a_2\geq 0$;}\\
\{(0,0)\},&\text{otherwise.}
\end{cases}
\end{equation}
Thus the restricted normal cone is $\{(0,0)\}$ for all boundary
points in the lower half disk that do not ``face'' the set $B$.
\end{example}

\begin{remark}
In contrast to Example~\ref{ex:cf}, we shall see in
Corollary~\ref{c:N=0}\ref{c:N=0ii} below that if $A$ is closed, $B$
is the affine hull of $A$, and $a$ belongs to the relative boundary
of $A$, then the restricted normal cone $\nc{A}{B}(a)$ strictly
contains $\{0\}$.
\end{remark}


\section{Restricted normal cones and affine subspaces}

\label{s:normalandaffine}

In this section, we consider the case when
the restricting set is a suitable affine subspace.
This results in further calculus rules and a characterization
of interiority notions.

The following four lemmas are useful in the derivation of the main
results in this section.

\begin{lemma}\label{l:aff-lspan}
Let $A$ and $B$ be nonempty subsets of $X$,
and suppose that $c\in A\cap B$.
Then
\begin{equation}
\aff(A\cup B)-c=\lspan(B-A).
\end{equation}
\end{lemma}
\begin{proof}
Since $c\in A\cap B \subseteq A\cup B$, it is clear that the
$\aff(A\cup B)-c$ is a subspace. On the one hand, if $a\in A$ and
$b\in B$, then $b-a = 1\cdot b + (-1)\cdot a + 1\cdot c - c
\in\aff(A\cup B)-c$. Hence $B-A\subseteq\aff(A\cup B)-c$ and thus
$\lspan(B-A)\subseteq\aff(A\cup B)-c$. On the other hand, if
$x\in\aff(A\cup B)$, say $x=\sum_{i\in I} \lambda_i a_i+\sum_{j\in
J} \mu_j b_j$, where each $a_i$ belongs to $A$, each $b_j$ belongs
to $B$, and $\sum_{i\in I}\lambda_i +\sum_{j\in J}\mu_j=1$, then
$x-c = \sum_{i\in I}(-\lambda_i)(c-a_i) + \sum_{j\in
I}\mu_j(b_j-c)\in\lspan(B-A)$. Thus $\aff(A\cup B)-c\subseteq
\lspan(B-A)$.
\end{proof}


\begin{lemma}\label{l:P(b+u)}
Let $A$ be a nonempty subset of $X$, let $a\in A$,
and let $u\in(\aff(A)-a)^\bot$.
Then
\begin{equation}
(\forall x\in X)\quad P_A(x+u)=P_A(x).
\end{equation}
\end{lemma}
\begin{proof}
Let $x\in X$. For every $b\in A$, we have
\begin{subequations}
\begin{align}
\|u+x-b\|^2&=\|u\|^2+2\scal{u}{x-b}+\|x-b\|^2\\
&=\|u\|^2+2\scal{u}{x-a}+2\scal{u}{a-b}+\|x-b\|^2\\
&=\|u\|^2+2\scal{u}{x-a}+\|x-b\|^2.
\end{align}
\end{subequations}
Hence $P_{A}(x+u)=\argmin_{b\in A}\|u+x-b\|^2=\argmin_{b\in
A}\|x-b\|^2 = P_Ax$, as announced.
\end{proof}
\begin{lemma}\label{l:PAPL}
Let $A$ be a nonempty subset of $X$, and let $L$ be an affine
subspace of $X$ containing $A$. Then
\begin{equation}
P_A=P_A\circ P_L.
\end{equation}
\end{lemma}
\begin{proof}
Let $a\in A$ and $x\in X$, and set $b=P_Lx$. Using
\cite[Corollary~3.20(i)]{BC2011}, we have $x-b\in (L-a)^\bot\subset
(\aff(A)-a)^\bot$. In view of Lemma~\ref{l:P(b+u)}, we deduce that
$(P_A\circ P_L)x=P_A(b)=P_A(b+(x-b))=P_Ax$.
\end{proof}

\begin{lemma}\label{l:per}
Let $A$ be a nonempty subset of $X$, let $a\in A$, and let $L$ be an
affine subspace of $X$ containing $A$. Then the following hold:
\begin{enumerate}
\item
\label{l:peri} $\pn{A}{L}(a) \bot (L-a)^\bot$.
\item
\label{l:perii} $\nc{A}{L}(a) \bot (L-a)^\bot$.
\end{enumerate}
\end{lemma}
\begin{proof}
Observe that $L-a=\pa(A)$ does not depend on the concrete choice of
$a\in A$. \ref{l:peri}: Using
Lemma~\ref{l:NsubsetN}\ref{l:NsubsetNviii}, we see that
$\pn{A}{L}(a)\subseteq\cone(L-a)\subseteq\lspan(L-a)\perp
(\lspan(L-a))^\perp = (L-a)^\perp = (\pa A)^\perp$. \ref{l:perii}:
By \ref{l:peri}, $\ran\pn{A}{L}\subseteq \pa A$. Since
$\ran\nc{A}{L}\subseteq\overline{\ran\pn{A}{L}}$, it follows that
$\ran\nc{A}{L}\subseteq \pa A = L-a$.
\end{proof}



For a normal cone restricted to certain affine subspaces, it is
possible to derive precise relationships to the Mordukhovich normal
cone.

\begin{theorem}[restricted vs Mordukhovich normal cone]
\label{p:pnA(L)} Let $A$ and $B$ be nonempty subsets of $X$,
suppose that $a\in A$,
and let $L$ be an affine subspace of $X$ containing $A$.
Then the following hold:
\begin{subequations}
\label{e:pnAL}
\begin{align}
\pn{A}{X}(a)&=\pn{A}{L}(a)\oplus(L-a)^\bot=\pn{A}{X}(a)+(L-a)^\bot,\label{e:pnALa}\\
\pn{A}{L}(a)&=\pn{A}{X}(a)\cap (L-a),\label{e:pnALb}\\
N_A(a)&=\nc{A}{L}(a)\oplus(L-a)^\bot=N_A(a)+(L-a)^\bot,\label{e:pnALc}\\
\nc{A}{L}(a)&=N_A(a)\cap (L-a).\label{e:pnALd}
\end{align}
\end{subequations}
Consequently, the following hold as well:
\begin{subequations}
\begin{align}
\pn{A}{X}(a)&=\pn{A}{\aff(A)}(a)\oplus(\aff(A)-a)^\bot=\pn{A}{X}(a)+(\aff(A)-a)^\bot,\label{e:pnALipa}\\
\pn{A}{\aff(A)}(a)&=\pn{A}{X}(a)\cap (\aff(A)-a),\label{e:pnALipb}\\
N_A(a)&=\nc{A}{\aff(A)}(a)\oplus\big(\aff(A)-a\big)^\bot=N_A(a)+\big(\aff(A)-a\big)^\bot,\label{e:pnALipc}\\
\nc{A}{\aff(A)}(a)&=N_A(a)\cap \big(\aff(A)-a\big),\label{e:pnALipd}\\
a\in A\cap B\;\;\Rightarrow\;\;
\nc{A}{\aff(A\cup B)}(a)&=N_A(a)\cap \lspan(A-B)\label{e:pnALipe}.
\end{align}
\end{subequations}
\end{theorem}
\begin{proof}
\eqref{e:pnALa}: Take $u\in\pn{A}{X}(a)$. Then there exist
$\lambda\geq 0$, $x\in X$, and $a\in P_Ax$ such that
$\lambda(x-a)=u$. Set $b=P_Lx$. By Lemma~\ref{l:PAPL}, we have $a\in
P_A x=(P_A\circ P_L)x= P_A b$. Using
\cite[Corollary~3.20(i)]{BC2011}, we thus deduce that
$\lambda(b-a)\in \pn{A}{L}(a)$ and $\lambda(x-b)\in
(L-b)^\perp=(L-a)^\perp$. Hence $u=\lambda(b-a)+\lambda(x-b)\in
\pn{A}{L}(a)+ (L-a)^\perp =\pn{A}{L}(a)\oplus (L-a)^\perp$ by
Lemma~\ref{l:per}\ref{l:peri}. We have thus shown that
\begin{equation}
\label{e:0211a} \pn{A}{X}(a) \subseteq \pn{A}{L}(a)\oplus
(L-a)^\perp.
\end{equation}
On the other hand, Lemma~\ref{l:NsubsetN}\ref{l:NsubsetNii} implies
that $\pn{A}{L}(a)\subseteq\pn{A}{X}(a)$ and thus
\begin{equation}
\pn{A}{L}(a)+(L-a)^\perp \subseteq \pn{A}{X}(a)+(L-a)^\perp.
\end{equation}
Altogether,
\begin{equation}
\pn{A}{X}(a) \subseteq \pn{A}{L}(a)\oplus (L-a)^\perp \subseteq
\pn{A}{X}(a)+(L-a)^\perp.
\end{equation}
To complete the proof of \eqref{e:pnALa}, it thus suffices to show
that $\pn{A}{X}(a)+ (L-a)^\bot\subseteq \pn{A}{X}(a)$. To this end,
let $u\in\pn{A}{X}(a)$ and
$v\in(L-a)^\bot\subseteq(\aff(A)-a)^\bot$. Then there exist
$\lambda\geq0$, $b\in X$, and $a\in P_Ab$ such that
$u=\lambda(b-a)$. If $\lambda=0$, then $u=0$ and
$u+v=v\in(\aff(A)-a)^\perp \subseteq (A-a)^\ominus = (A-a)^\ominus
\cap X = (A-a)^\ominus \cap \cone(X-a) \subseteq \pn{A}{X}(a)$ by
Lemma~\ref{l:NsubsetN}\ref{l:NsubsetNvii}\&\ref{l:NsubsetNviii}.
Thus, we assume that $\lambda>0$. By Lemma~\ref{l:P(b+u)}, we have
$a\in P_A b = P_A(b+\lambda^{-1}v)$. Hence
$b+\lambda^{-1}v-a\in\pn{A}{X}(a)$ and therefore
$\lambda(b+\lambda^{-1}v-a) = \lambda(b-a)+v=u+v\in\pn{A}{X}(a)$, as
required.

\eqref{e:pnALb}: By
Lemma~\ref{l:NsubsetN}\ref{l:NsubsetNii}\&\ref{l:NsubsetNviii},
$\pn{A}{L}(a)\subseteq\pn{A}{X}(a)\cap(L-a)$. Now let $u\in
\pn{A}{X}(a)\cap(L-a)$. By \eqref{e:pnALa}, we have $u=v+w$, where
$v\in\pn{A}{L}(a)\subseteq L-a$ and $w\in(L-a)^\perp$. On the other
hand, $w=u-v\in (L-a)-(L-a)=L-a$. Altogether $w\in
(L-a)\cap(L-a)^\perp=\{0\}$. Hence $u=v\in\pn{A}{L}(a)$.

\eqref{e:pnALc}: Let $u\in N_A(a)$. By definition, there exist
sequences $(a_n)_\nnn$ in $A$ and $(u_n)_\nnn$ in $X$ such that
$a_n\to a$, $u_n\to u$, and $(\forall\nnn)$ $u_n\in\pn{A}{X}(a_n)$.
By \eqref{e:pnALa}, there exists a sequence $(v_n,w_n)_\nnn$ such
that $(a_n,v_n)_\nnn$ lies in $\gr\pn{A}{L}$, $(w_n)_\nnn$ lies in
$(L-a)^\perp$, and $(\forall\nnn)$ $u_n=v_n+w_n$ and $v_n\perp w_n$.
Since $\|u\|^2 \leftarrow \|u_n\|^2 = \|v_n\|^2 + \|w_n\|^2$, the
sequences $(v_n)_\nnn$ and $(w_n)_\nnn$ are bounded. After passing
to subsequences and relabeling if necessary, we assume $(v_n)_\nnn$
and $(w_n)_\nnn$ are convergent, with limits $v$ and $w$,
respectively. It follows that $v\in\nc{A}{L}(a)$ and
$w\in(L-a)^\perp$; consequently, $u=v+w\in\nc{A}{L}(a) \oplus
(L-a)^\perp$ by Lemma~\ref{l:per}\ref{l:perii}. Thus
$N_A(a)\subseteq \nc{A}{L}(a) \oplus (L-a)^\perp$. On the other
hand, by Lemma~\ref{l:NsubsetN}\ref{l:NsubsetNii},
$\nc{A}{L}(a)\oplus(L-a)^\bot\subseteq N_A(a)+(L-a)^\bot$.
Altogether,
\begin{equation}
N_A(a)\subseteq \nc{A}{L}(a)\oplus(L-a)^\bot\subseteq
N_A(a)+(L-a)^\bot.
\end{equation}
It thus suffices to prove that $N_A(a)+ (L-a)^\bot\subseteq N_A(a)$.
To this end, take $u\in N_A(a)$ and $v\in (L-a)^\perp$. Then there
exist sequences $(a_n)_\nnn$ in $A$ and $(u_n)_\nnn$ in $X$ such
that $a_n\to a$, $u_n\to u$, and $(\forall\nnn)$
$u_n\in\pn{A}{X}(a_n)$. For every $\nnn$, we have $L-a=L-a_n$ and
hence $u_n+v\in \pn{A}{X}(a_n) + (L-a_n)^\perp = \pn{A}{X}(a_n)$ by
\eqref{e:pnALa}. Passing to the limit, we conclude that $u+v\in
N_A(a)$.

\eqref{e:pnALd}: First, take $u\in \nc{A}{L}(a)$. On the one hand,
by Lemma~\ref{l:NsubsetN}\ref{l:NsubsetNii}, $u\in N_A(a)$. On the
other hand, by Lemma~\ref{l:per}\ref{l:perii}, $u\in
(L-a)^{\perp\perp} = L-a$. Altogether, we have shown that
\begin{equation}
\label{e:0211b} \nc{A}{L}(a)\subseteq N_A(a)\cap(L-a).
\end{equation}
Conversely, take $u\in N_A(a)\cap(L-a)\subseteq N_A(a)$. By
\eqref{e:pnALc}, there exist $v\in \nc{A}{L}(a)$ and $w\in
(L-a)^\perp$ such that $u=v+w$ and $v\perp w$. By \eqref{e:0211b},
$v\in L-a$. Hence $w = u-v\in (L-a)-(L-a)=L-a$. Since $w\in
(L-a)^\perp$, we deduce that $w=0$. This implies $u=v\in
\nc{A}{L}(a)$. Therefore, $N_A(a)\cap (L-a)\subseteq \nc{A}{L}(a)$.

``Consequently'' part: Consider \eqref{e:pnAL} when $L=\aff(A)$ or
$L=\aff(A\cup B)$, and recall Lemma~\ref{l:aff-lspan} in the latter
case.
\end{proof}

An immediate consequence of Theorem~\ref{p:pnA(L)} (or of the definitions)
is the following result.

\begin{corollary}[the $X$-restricted and the Mordukhovich normal
cone coincide] \label{c:Boris=ncX} Let $A$ be a nonempty subset of
$X$, and let $a\in A$. Then
\begin{equation}
\nc{A}{X}(a)=N_A(a).
\end{equation}
\end{corollary}

The next two results provide some useful calculus rules.

\begin{corollary}[restricted normal cone of a sum]\label{c:cal1}
Let $C_1$ and $C_2$ be nonempty closed convex subsets of $X$, let
$a_1\in C_1$, let $a_2\in C_2$, and let $L$ be an affine subspace of
$X$ containing $C_1+C_2$. Then
\begin{equation}
\nc{C_1+C_2}{L}(a_1+a_2)=\nc{C_1}{L-a_2}(a_1)\cap\nc{C_2}{L-a_1}(a_2).
\end{equation}
\end{corollary}
\begin{proof}
Set $C=C_1+C_2$ and $a=a_1+a_2$. Then \eqref{e:pnALd} and
\cite[Exercise~6.44]{Rock98} yield
\begin{subequations}
\label{e:0213d}
\begin{align}
\nc{C}{L}(a)&=\nc{C}{}(a)\cap(L-a)=\nc{C_1}{}(a_1)\cap\nc{C_2}{}(a_2)\cap(L-a)\\
&=\big(\nc{C_1}{}(a_1)\cap(L-a)\big)\cap\big(\nc{C_2}{}(a_2)\cap(L-a)\big).
\end{align}
\end{subequations}
Note that $L-a$ is a linear subspace of $X$ containing $C_1-a_1$ and
$C_2-a_2$. Thus, $L-a_2=L-a+a_1$ is an affine subspace of $X$
containing $C_1$, and $L-a_1=L-a+a_2$ is an affine subspace of $X$
containing $C_2$. By \eqref{e:pnALd},
\begin{equation}
\label{e:0213e} \nc{C_1}{L-a_2}(a_1)=\nc{C_1}{}(a_1)\cap(L-a)
\quad\text{and}\quad \nc{C_2}{L-a_1}(a_2)=\nc{C_2}{}(a_2)\cap(L-a).
\end{equation}
The conclusion follows by combining \eqref{e:0213d} and
\eqref{e:0213e}.
\end{proof}

\begin{corollary}[an intersection formula]\label{c:cal2}
Let $A$ and $B$ be nonempty closed convex subsets of $X$, and suppose that
$a\in A\cap B$. Let $L$ be an affine subspace of $X$ containing
$A\cup B$. Then
\begin{equation}
\nc{A}{L}(a)\cap\big(-\nc{B}{L}(a)\big)=\nc{A-B}{L-a}(0).
\end{equation}
\end{corollary}
\begin{proof}
Using \eqref{e:pnALd}, Proposition~\ref{p:elementary}\ref{p:ele-v},
\cite[Exercise~6.44]{Rock98}, and again \eqref{e:pnALd}, we obtain
\begin{subequations}
\begin{align}
\nc{A}{L}(a)\cap\big(-\nc{B}{L}(a)\big)&=N_A(a)\cap(L-a)\cap\big(-N_B(a)\big)\cap(L-a)\\
&=\Big(N_A(a)\cap\big(-N_B(a)\big)\Big)\cap(L-a)\\
&=\big(N_A(a)\cap N_{-B}(-a)\big)\cap(L-a)\\
&=N_{A-B}(0)\cap(L-a)\\
&=\nc{A-B}{L-a}(0),
\end{align}
\end{subequations}
as required.
\end{proof}

Let us now work towards relating the restricted normal cone to the
(relative and classical) interior and to the boundary of a given
set.

\begin{proposition}\label{p:L=aff(A)}
Let $A$ be a nonempty subset of $X$, let $a\in A$, let $L$ be an
affine subspace containing $A$, and suppose that
$\nc{A}{L}(a)=\{0\}$. Then $L=\aff(A)$.
\end{proposition}
\begin{proof}
Using $0\in\nc{A}{\aff(A)}(a)\subseteq\nc{A}{L}(a)=\{0\}$ and
applying \eqref{e:pnALc} and \eqref{e:pnALipc}, we have
\begin{equation}
N_A(a)=0+(L-a)^\perp=0+(\aff(A)-a)^\perp.
\end{equation}
So $L-a=\aff(A)-a$, i.e., $L=\aff(A)$.
\end{proof}

\begin{theorem}
\label{t:ncEquiv} Let $A$ and $B$ be nonempty subsets of $X$, and
let $a\in A$. Then
\begin{equation}
\label{e:0212a} \nc{A}{B}(a)=\{0\} \quad\Leftrightarrow\quad
(\exi\dd>0)\big(\forall x\in A\cap \ball{a}{\dd}\big)\;\; P^{-1}_A(x)\cap
B \subseteq\{x\}.
\end{equation}
Furthermore, if $A$ is closed and $B$ is an affine subspace of $X$
containing $A$, then the following are equivalent:
\begin{enumerate}
\item
\label{t:ncEquivi} $\nc{A}{B}(a)=\{0\}$.
\item
\label{t:ncEquivii} $(\exi\rho>0)$ $\ball{a}{\rho}\cap B\subseteq A$.
\item
\label{t:ncEquiviii} $B=\aff(A)$ and $a\in \reli(A)$.
\end{enumerate}
\end{theorem}
\begin{proof}
Note that $\nc{A}{B}(a)=\{0\}$ $\Leftrightarrow$ $(\exi\dd>0)$
$(\forall x\in A\cap \ball{a}{\dd})$ $\pn{A}{B}(x)=\{0\}$. Hence
\eqref{e:0212a} follows from the definition of $\pn{A}{B}(x)$.

Now suppose that $A$ is closed and $B$ is an affine subspace of $X$
containing $A$.

``\ref{t:ncEquivi}$\Rightarrow$\ref{t:ncEquivii}'':
Let $\delta>0$ be as in \eqref{e:0212a} and set $\rho := \delta/2$.
Let $b\in B(a;\rho)\cap B$, and take $x\in P_Ab$, which is possible
since $A$ is closed. Then $\|b-x\|=d_A(b)\leq\|b-a\|\leq\rho$ and
hence
\begin{equation}
\|x-a\|\leq\|x-b\|+\|b-a\| \leq\rho+\rho=2\rho = \delta.
\end{equation}
Using \eqref{e:0212a}, we deduce that $b\in P_A^{-1}(x)\cap
B\subseteq\{x\}\subseteq A$.

``\ref{t:ncEquivii}$\Rightarrow$\ref{t:ncEquiviii}'': It follows
that $B=\aff(B)\subseteq\aff(A)\subseteq B$; hence, $B=\aff(A)$.
Thus $\ball{a}{\rho}\cap\aff(A)\subseteq A$, which means that
$a\in\reli(A)$.

``\ref{t:ncEquiviii}$\Rightarrow$\ref{t:ncEquivi}'':
Lemma~\ref{l:NsubsetN}\ref{l:NsubsetNri}.
\end{proof}

\begin{corollary}[interior and boundary characterizations]
\label{c:N=0} Let $A$ be a nonempty closed subset of $X$, and let
$a\in A$. Then the following hold:
\begin{enumerate}
\item
\label{c:N=0i} $\nc{A}{\aff(A)}(a)=\{0\}$ $\Leftrightarrow$
$a\in\reli(A)$.
\item
\label{c:N=0ii} $\nc{A}{\aff(A)}(a)\neq\{0\}$ $\Leftrightarrow$
$a\in A\smallsetminus \reli(A)$.
\item
\label{c:N=0iii} $N_A(a)=\{0\}$ $\Leftrightarrow$ $a\in\inte(A)$.
\item
\label{c:N=0iv} $N_A(a)\neq\{0\}$ $\Leftrightarrow$ $a\in
A\smallsetminus\inte(A)$.
\end{enumerate}
\end{corollary}
\begin{proof}
\ref{c:N=0i}: Apply Theorem~\ref{t:ncEquiv} with $B=\aff(A)$.
\ref{c:N=0ii}: Clear from \ref{c:N=0i}. \ref{c:N=0iii}: Apply
Theorem~\ref{t:ncEquiv} with $B=X$, and recall
Corollary~\ref{c:Boris=ncX}. \ref{c:N=0iv}: Clear from
\ref{c:N=0iii}.
\end{proof}


A second look at the proof of
\ref{t:ncEquivi}$\Rightarrow$\ref{t:ncEquivii} in
Theorem~\ref{t:ncEquiv} reveals that this implication does actually
not require the assumption that $B$ be an affine subspace of $X$
containing $A$. The following example illustrates that the converse
implication fails even when $B$ is a superset of $\aff(A)$.

\begin{example}
Suppose that $X=\RR^2$, and set $A := \RR\times\{0\}$, $a=(0,0)$,
and $B=\RR\times\{0,2\}$. Then $A=\aff(A)\subseteq B$ and
$\ball{a}{1}\cap B\subseteq A$; however, $(\forall x\in A)$
$\pn{A}{B}(x)=\{0\}\times\RP$ and therefore
$\nc{A}{B}(a)=\{0\}\times\RP\neq\{(0,0)\}$.
\end{example}


\subsection*{Two convex sets}

It is instructive to interpret the previous results for two convex sets:

\begin{theorem}[two convex sets:
restricted normal cones and relative interiors]
\label{t:compareCQ2}
Let $A$ and $B$ be nonempty convex subsets of $X$.
Then the following are equivalent:
\begin{enumerate}
\item\label{t:CQ2i1} $\reli A\cap\reli B\neq\emptyset$.
\item\label{t:CQ2i} $0\in\reli(B-A)$.
\item\label{t:CQ2i2} $\cone(B-A)=\lspan(B-A)$.
\item\label{t:CQ2ii} $N_A(c)\cap(-N_B(c))\cap\overline{\cone}(B-A)=\{0\}$ for some $c\in A\cap B$.
\item\label{t:CQ2iii} $N_A(c)\cap(-N_B(c))\cap\overline{\cone}(B-A)=\{0\}$ for every $c\in A\cap B$.
\item\label{t:CQ2iv} $N_A(c)\cap(-N_B(c))\cap\lspan(B-A)=\{0\}$ for some $c\in A\cap B$.
\item\label{t:CQ2v} $N_A(c)\cap(-N_B(c))\cap\lspan(B-A)=\{0\}$ for every $c\in A\cap B$.
\item\label{t:CQ2vi} $\nc{A}{\aff(A\cup B)}(c) \cap (-\nc{B}{\aff(A\cup
B)}(c)) = \{0\}$ for some $c\in A\cap B$.
\item\label{t:CQ2vii} $\nc{A}{\aff(A\cup B)}(c) \cap (-\nc{B}{\aff(A\cup B)}(c)) =\{0\}$ for every $c\in A\cap B$.
\item\label{t:CQ2viii} $\nc{A-B}{\lspan(B-A)}(0)=\{0\}$.
\end{enumerate}
\end{theorem}
\begin{proof}
By \cite[Corollary~6.6.2]{Rock70},
\ref{t:CQ2i}
$\Leftrightarrow$
$\reli A\cap \reli B\neq\varnothing$
$\Leftrightarrow$
$0\in \reli A - \reli B$
$\Leftrightarrow$
\ref{t:CQ2i}.

Applying Proposition~\ref{p:onecone} to $B-A$,
and \cite[Proposition~3.1.3]{BBL} to $\ccone(B-A)$,
we obtain
\begin{subequations}
\label{e:0302a}
\begin{align}
\text{\ref{t:CQ2i}}
&\Leftrightarrow
\text{\ref{t:CQ2i2}}
\;\Leftrightarrow\;
\ccone(B-A)=\lspan(B-A)\\
&\Leftrightarrow
\ccone(B-A)\cap\big(\ccone(B-A)\big)^\oplus = \{0\}.
\end{align}
\end{subequations}
Let $c\in A\cap B$. Then Corollary~\ref{c:cal2} (with $L=X$) yields
$N_A(c)\cap\big(-N_B(c)\big)=N_{A-B}(0)=
(A-B)^\ominus=(B-A)^\oplus=(\overline{\cone}(B-A))^\oplus$.
Hence
\begin{equation}
\label{e:0302b}
(\forall c\in C)\quad N_A(c)\cap\big(-N_B(c)\big)
\cap \ccone(B-A) = \big(\ccone(B-A)\big)^\oplus \cap \ccone(B-A)
\end{equation}
and
\begin{equation}
\label{e:0302c}
(\forall c\in C)\quad N_A(c)\cap\big(-N_B(c)\big)
\cap \lspan(B-A) = \big(\ccone(B-A)\big)^\oplus \cap \lspan(B-A).
\end{equation}
Combining \eqref{e:0302a}, \eqref{e:0302b}, and \eqref{e:0302c},
we see that \ref{t:CQ2i}--\ref{t:CQ2v} are equivalent.

Next, Lemma~\ref{l:aff-lspan} and Corollary~\ref{c:cal2} yield
the equivalence of \ref{t:CQ2vi}--\ref{t:CQ2viii}.

Finally,
\ref{t:CQ2viii}$\Leftrightarrow$\ref{t:CQ2i} by
Corollary~\ref{c:N=0}\ref{c:N=0i}.
\end{proof}

\begin{corollary}[two convex sets: normal cones and interiors]
\label{c:compareCQ3}
Let $A$ and $B$ be nonempty convex subsets of $X$.
Then the following are equivalent:
\begin{enumerate}
\item\label{c:CQ3i} $0\in\inte(B-A)$.
\item\label{c:CQ3i1} $\cone(B-A)=X$.
\item\label{c:CQ3ii} $N_A(c) \cap (-N_B(c)) = \{0\}$ for some $c\in A\cap B$.
\item\label{c:CQ3iii} $N_A(c) \cap (-N_B(c)) = \{0\}$ for every $c\in A\cap B$.
\item\label{c:CQ3iv} $N_{A-B}(0)=\{0\}$.
\end{enumerate}
\end{corollary}
\begin{proof}
We start by notating that if $C$ is a convex subset of $X$, then
$0\in\inte C$ $\Leftrightarrow$ $0\in\reli C$ and $\lspan C = X$.
Consequently,
\begin{equation}
\label{e:0302d}
\text{\ref{c:CQ3i}}
\quad\Leftrightarrow\quad
0\in\reli(B-A)\;\text{and}\;\lspan(B-A)=X.
\end{equation}
Assume that \ref{c:CQ3i} holds.
Then \eqref{e:0302d} and Theorem~\ref{t:compareCQ2} imply that
$\cone(B-A)=\ccone(B-A)=\lspan(B-A)=X$. Hence \ref{c:CQ3i1} holds,
and from Theorem~\ref{t:compareCQ2} we obtain that
\ref{c:CQ3i1}$\Rightarrow$\ref{c:CQ3ii}$\Leftrightarrow$\ref{c:CQ3iii}$\Leftrightarrow$\ref{c:CQ3iv}.
Finally, Corollary~\ref{c:N=0}\ref{c:N=0iii} yields the implication
\ref{c:CQ3iv}$\Rightarrow$\ref{c:CQ3i}.
\end{proof}

\section{Further examples}

\label{s:further}

In this section, we provide further examples that illustrate
particularities of restricted normal cones.

As announced in Remark~\ref{r:120212a},
when $a\in A_2\subsetneqq A_1$, it is possible that
the \emph{nonconvex} restricted normal cones
satisfy
$\nc{A_1}{B}(a)\not\subseteq\nc{A_2}{B}(a)$ even when $A_1$ and
$A_2$ are both \emph{convex}. This lack of inclusion is also known
for the Mordukhovich normal cone (see \cite[page~5]{Boris1}, where however
one of the sets is not convex).
Furthermore, the following example also shows that the restricted normal
cone cannot be derived from the Mordukhovich normal cone by
the simple relativization procedure of intersecting with naturally
associated cones and subspaces.

\begin{example}[lack of convexity, inclusion, and relativization]
\label{ex:120212a} Suppose that $X=\RR^2$, and define two nonempty
closed \emph{convex} sets by $A := A_1:=\epi(|\cdot|)$ and
$A_2:=\epi(2|\cdot|)$. Then $a := (0,0) \in A_2\subsetneqq A_1$.
Furthermore, set $B:=\RR\times\{0\}$. Then
\begin{subequations}
\begin{align}
\big(\forall x=(x_1,x_2)\in A_1\big)\quad \pn{A_1}{B}(x)&=
\begin{cases}
\RP(1,-1),& \text{if $x_2=x_1>0$;}\\
\RP(-1,-1),& \text{if $x_2=-x_1>0$;}\\
\{(0,0)\},&\text{otherwise,}
\end{cases}\\
\big(\forall x=(x_1,x_2)\in A_2\big)\quad \pn{A_2}{B}(x)&=
\begin{cases}
\RP(2,-1),& \text{if $x_2=2x_1>0$;}\\
\RP(-2,-1),& \text{if $x_2=-2x_1>0$;}\\
\{(0,0)\},&\text{otherwise.}
\end{cases}
\end{align}
\end{subequations}
Consequently,
\begin{subequations}
\begin{align}
\nc{A_1}{B}(a)&=\cone\big\{(1,-1),(-1,-1)\big\},\\
\nc{A_2}{B}(a)&=\cone\big\{(2,-1),(-2,-1)\big\}.
\end{align}
\end{subequations}
Note that $\nc{A_1}{B}(a)\not\subseteq\nc{A_2}{B}(a)$ and
$\nc{A_2}{B}(a)\not\subseteq\nc{A_1}{B}(a)$; in fact,
$\nc{A_1}{B}(a)\cap \nc{A_2}{B}(a)=\{(0,0)\}$.
Furthermore, neither
$\nc{A_1}{B}(a)$ nor $\nc{A_2}{B}(a)$ is convex even though $A_1$,
$A_2$, and $B$ are.
Finally, observe that
$\cone(B-a)=\lspan(B-a)=B$, that $\cone(B-A)=\RR\times\RM$,
that $\lspan(B-A)= X$, and that $N_A(a)=\cone[(1,-1),(-1,-1)]\neq
\nc{A}{B}(a)$.
Consequently,
$\cone(B-a)\cap N_A(a)=\lspan(B-a)\cap N_A(a)=\{(0,0)\}$,
$\cone(B-A)\cap N_A(a)= N_A(a)=\lspan(B-A)\cap N_A(a)$.
Therefore, $\nc{A}{B}(a)$ \emph{cannot be obtained by intersecting
the Mordukhovich normal cone} with one of the sets
$\cone(B-a)$, $\lspan(B-a)$, $\cone(B-A)$, and $\lspan(B-A)$.
\end{example}


We shall present some further examples. The proof of the following
result is straight-forward and hence omitted.

\begin{proposition}
\label{p:ncK} Let $K$ be a closed cone in $X$, and let $B$ be a
nonempty cone of $X$. Then
\begin{equation}
\nc{K}{B}(0)=\overline{\bigcup_{x\in K}\pn{K}{B}(x)}=
\overline{\bigcup_{x\in \bd K}\pn{K}{B}(x)}= \overline{\bigcup_{x\in
K}\nc{K}{B}(x)}= \overline{\bigcup_{x\in \bd K}\nc{K}{B}(x)}.
\end{equation}
\end{proposition}

\begin{example}
\label{ex:ncK} Let $K$ be a closed convex cone in $X$, suppose that
$u_0\in\inte(K)$ and that $K\subseteq\{u_0\}^\oplus$, and set $B :=
\{u_0\}^\perp$. Then:
\begin{enumerate}
\item
\label{ex:ncKi} $(\forall x\in K\cap B)$ $\pn{K}{B}(x)=\{0\}$.
\item
\label{ex:ncKii} $(\forall x\in K\smallsetminus B)$
$\pn{K}{B}(x)=\nc{K}{B}(x)=N_K(x)=K^\ominus\cap\{x\}^\perp$.
\item
\label{ex:ncKiii} $\nc{K}{B}(0)=\overline{\bigcup_{x\in
K}\pn{K}{B}(x)}= \overline{\bigcup_{x\in K\smallsetminus
B}(K^\ominus\cap \{x\}^\perp)}= \overline{K^\ominus\cap\bigcup_{x\in
K\smallsetminus
B}\{x\}^\perp}$.\\[+2mm]
If one of these unions is closed, then all closures may be omitted.
\end{enumerate}
\end{example}
\begin{proof}
\ref{ex:ncKi}: Let $x\in K\cap B$. It suffices to show that $B\cap
P_K^{-1}(x)=\{x\}$. To this end, take $y\in B\cap P_K^{-1}(x)$. By
definition of $B$, we have $\scal{u_{0}}{x}=0$ and
$\scal{u_{0}}{y}=0$. Hence
\begin{equation}\label{e:ontheplane}
\scal{u_{0}}{y-x}=0.
\end{equation}
Furthermore, $x=P_Ky$ and hence, using e.g.\
\cite[Proposition~6.27]{BC2011}, we have $y-x\in K^\ominus$. Since
$u_{0}\in\inte K$, there exists $\delta>0$ such that
$\ball{u_{0}}{\delta}\subseteq K$. Thus $y-x\in
(\ball{u_{0}}{\delta})^\ominus$. In view of \eqref{e:ontheplane},
$\delta\|y-x\|\leq 0$. Therefore, $y=x$.


\ref{ex:ncKii}: Let $x\in K\smallsetminus B$. Using
Lemma~\ref{l:NsubsetN}\ref{l:NsubsetNii}\&\ref{l:NsubsetNiv},
Corollary~\ref{c:Boris=ncX},
Lemma~\ref{l:NsubsetN}\ref{l:NsubsetNvi}, and
\cite[Example~6.39]{BC2011}, we have
\begin{equation}
\label{e:0212b} \pn{K}{B}(x)\subseteq \pn{K}{X}(x) \subseteq
\nc{K}{X}(x) = N_K(x) = \cnc{K}(x) = K^\ominus\cap \{x\}^\perp.
\end{equation}
Since $x\in K\subseteq\{u_0\}^\oplus$ and $x\notin B$, we have
$\scal{u_0}{x}>0$. Now take $u\in
(K^\ominus\cap\{x\}^\perp)\smallsetminus\{0\}$. Since $u\in
K^\ominus$ and $u_0\in\inte(K)$, we have $\scal{u}{u_0}<0$. Now set
\begin{equation}
b:=x-\frac{\scal{u_0}{x}}{\scal{u_0}{u}}u.
\end{equation}
Then $b\in B$ and $b-x = -\scal{u_0}{x}\scal{u_0}{u}^{-1}u\in\RPP u
\subseteq K^\ominus\cap\{x\}^\perp = \cnc{K}(x)$. By
\cite[Proposition~6.46]{BC2011}, $x=P_Kb$. Hence
$b-x\in\pn{K}{B}(x)$ and thus $u\in\pn{K}{B}(x)$. Therefore,
$K^\ominus\cap\{x\}^\perp \subseteq \pn{K}{B}(x)$. In view of
\eqref{e:0212b}, and since $\pn{K}{B}(x)\subseteq
\nc{K}{B}(x)\subseteq N_K(x)$ by
Lemma~\ref{l:NsubsetN}\ref{l:NsubsetNii}\&\ref{l:NsubsetNiv}, we
have established \ref{ex:ncKii}.

\ref{ex:ncKiii}: Combine \ref{ex:ncKi}, \ref{ex:ncKii}, and
Proposition~\ref{p:ncK}.
\end{proof}

\begin{example}[ice cream cone]
\label{ex:ice} Suppose that $X=\RR^m=\RR^{m-1}\times\RR$, where
$m\in\{2,3,4,\ldots\}$, and let $\beta>0$. Define the corresponding
closed convex \emph{ice cream cone} by
\begin{equation}
K:=\Menge{x\in\RR^m}{ \beta\sqrt{x_1^2+\cdots+x_{m-1}^2}\leq x_m},
\end{equation}
and set $B:=\RR^{m-1}\times\{0\}$.
Then the following hold:
\begin{enumerate}
\item
\label{ex:icei} $\pn{K}{B}(0,0)=\{(0,0)\}$.
\item
\label{ex:icei+}
$N_K(0,0)=\menge{y\in\RR^m}{\beta^{-1}\sqrt{y_1^2+\cdots+y_{m-1}^2}\leq
-y_m} = \bigcup_{z\in\RR^{m-1}\atop \|z\|\leq 1} \RP(\beta z,-1)$.
\item
\label{ex:iceii} $(\forall z\in\RR^{m-1}\smallsetminus\{0\})$
$\pn{K}{B}(z,\beta\|z\|)=\nc{K}{B}(z,\beta\|z\|)
=N_K(z,\beta\|z\|)=\RP(\beta{z},-\|z\|)$.
\item
\label{ex:iceiii} $\nc{K}{B}(0,0)=\bigcup_{z\in\RR^{m-1}\atop
\|z\|=1} \RP(\beta z,-1)$, which is a closed cone that is \emph{not
convex}.
\end{enumerate}
\end{example}
\begin{proof}
Clearly, $K$ is closed and convex. Note that $K$ is the lower level set
of height $0$ of the continuous convex function
\begin{equation}
f\colon \RR^m=\RR^{m-1}\times\RR\to \RR\colon x=(z,x_m)\mapsto
\beta\|z\|-x_m;
\end{equation}
hence , by \cite[Exercise~2.5(b) and its solution on
page~205]{Zalinescu},
\begin{equation}
\label{e:0213a}
\inte(K)=\menge{x=(z,x_m)\in\RR^{m-1}\times\RR}{\beta\|z\|<x_m}.
\end{equation}
Lemma~\ref{l:NsubsetN}\ref{l:NsubsetNii}\&\ref{l:NsubsetNiv},
Corollary~\ref{c:Boris=ncX}, and Corollary~\ref{c:N=0}\ref{c:N=0iii}
imply that
\begin{equation}
\label{e:0213b} \big(\forall x\in\inte(K)\big)\quad \pn{K}{B}(x)
\subseteq \pn{K}{X}(x) \subseteq \nc{K}{X}(x) = N_K(x) = \{0\}.
\end{equation}
Write $x=(z,x_m)\in \RR^{m-1}\times\RR=X$, and assume that $x\in K$.
We thus assume that $x\in\bd(K)$, i.e., $\beta\|z\|=x_m$ by
\eqref{e:0213a}, i.e., $x=(z,\beta\|z\|)$. Combining
\cite[Proposition~16.8]{BC2011} with
\cite[Corollary~2.9.5]{Zalinescu} (or \cite[Lemma~26.17]{BC2011})
applied to $f$, we obtain
\begin{equation}
N_K\big(z,\beta\|z\|\big) = \cone\big(
\beta\partial\|\cdot\|(z)\times\{-1\}\big),
\end{equation}
where $\partial\|\cdot\|$ denotes the subdifferential operator from
convex analysis applied to the Euclidean norm in $\RR^{m-1}$. In
view of \cite[Example~16.25]{BC2011} we thus have
\begin{equation}
\label{e:0213c} N_K\big(z,\beta\|z\|\big) =
\begin{cases}
\cone\big(\beta\|z\|^{-1}z\times\{-1\}\big), &\text{if $z\neq 0$;}\\
\cone\big(\ball{0}{\beta}\times\{-1\}\big), &\text{if $z=0$.}\\
\end{cases}
\end{equation}
The case $z=0$ in \eqref{e:0213c} readily leads to \ref{ex:icei+}.

Now set $u_0 := (0,1)\in\RR^{m-1}\times\RR$. Then $\{u_0\}^\perp =
B$ and $\{u_0\}^\oplus = \RR^{m-1}\times\RP \supseteq K$. Note that
$(0,0)\in K\cap B$ and thus $\pn{K}{B}(0,0)=\{(0,0)\}$ by
Example~\ref{ex:ncK}\ref{ex:ncKi}. We have thus established
\ref{ex:icei}.

Now assume that $z\neq 0$. Then $N_K(z,\beta\|z\|) = \RP(\beta
z,-\|z\|)$. Note that $\beta z\neq 0$ and so $(z,\beta\|z\|)\notin
B$. The formulas announced in \ref{ex:iceii} therefore follow from
Example~\ref{ex:ncK}\ref{ex:ncKii}.

Next, combining \eqref{e:0213a}, \eqref{e:0213b}, and
Example~\ref{ex:ncK}\ref{ex:ncKiii} as well as utilizing the
compactness of the unit sphere in $\RR^{m-1}$, we see that
\begin{equation}
\nc{K}{B}(0,0) =\overline{\bigcup_{z\in\RR^{m-1}\smallsetminus\{0\}}
\RP(\beta z,-\|z\|)} =\overline{\bigcup_{z\in\RR^{m-1}\atop \|z\|=1}
\RP(\beta z,-1)} ={\bigcup_{z\in\RR^{m-1}\atop \|z\|=1} \RP(\beta
z,-1)}.
\end{equation}
This establishes \ref{ex:iceiii}.
\end{proof}

\begin{remark}
Consider Example~\ref{ex:ice}. Note that $\nc{K}{B}(0,0)$ is
actually the boundary of $N_K(0,0)$. Furthermore, since
$N_K(0,0)=\cnc{K}(0,0)$ by Lemma~\ref{l:NsubsetN}\ref{l:NsubsetNvi},
the formulas in \ref{ex:icei+} also describe $K^\ominus$, which is
therefore an ice cream cone as well.
\end{remark}



\section{Cones containing restricted normal cones}

\label{s:nosuper}

In this section, we provide various examples illustrating that
the restricted (proximal) normal cone does not naturally arise
by considering various natural cones containing it.

Let $A$ and $B$ be nonempty subsets of $X$,
and let $a\in A$.
We saw in Lemma~\ref{l:NsubsetN}\ref{l:NsubsetNi+}
that
\begin{equation}\label{e:pn=cone(B-a)}
\pn{A}{B}(a)=\cone\big((B-a)\cap(P^{-1}_Aa-a)\big)\subseteq
\cone(B-a)\cap\pnX{A}(a).
\end{equation}
This raises the question whether or not the inclusion
in \eqref{e:pn=cone(B-a)} is strict.
It turns out and as we shall now illustrate,
both conceivable alternatives (equality and strict inclusion)
do occur. Therefore, $\pn{A}{B}(a)$ is a new construction.

We start with a condition sufficient for equality in \eqref{e:pn=cone(B-a)},
\begin{proposition}\label{p:pn=cone(B-a)}
Let $A$ and $B$ be nonempty subsets of $X$. Let $A$ be closed and $a\in A$.
Assume that one of the following holds:
\begin{enumerate}
\item
\label{p:pn=cone(B-a)i}
$P^{-1}_A(a)-a$ is a cone.
\item
\label{p:pn=cone(B-a)ii}
$A$ is convex.
\end{enumerate}
Then
$\pn{A}{B}(a)= \cone(B-a)\cap\pnX{A}(a)$.
\end{proposition}
\begin{proof}
\ref{p:pn=cone(B-a)i}:
Lemma~\ref{l:cone01}\ref{l:cone01ii}.
\ref{p:pn=cone(B-a)ii}:
Combine
\ref{p:pn=cone(B-a)i} with Lemma~\ref{l:PA&cvex}.
\end{proof}

The next examples illustrates that equality in \eqref{e:pn=cone(B-a)}
can occur even though $P^{-1}_A(a)-a$ is not a cone.
Consequently, the assumption that
$P^{-1}_A(a)-a$ be a cone in Proposition~\ref{p:pn=cone(B-a)}
is sufficient---but not necessary---for equality in \eqref{e:pn=cone(B-a)}.

\begin{example}
Suppose that $X = \RR^2$, and
let $A:= X\smallsetminus\RPP^2$,
$B:=\RP(1,1)$, and $a:=(0,1)$.
Then one verifies that
\begin{subequations}
\begin{align}
P^{-1}_A(a) - a&=[0,1]\times\{0\},\\
\pnX{A}(a)&=\cone(P^{-1}_A a - a)=\RP\times\{0\},\\
\cone(B-a)&=\menge{(t_1,t_2)\in\RR^2}{t_1\geq0,t_2<t_1}\cup\{(0,0)\},\\
\pn{A}{B}(a)&=\RP\times\{0\}.
\end{align}
\end{subequations}
Hence $\pn{A}{B}(a)=\RP\times\{0\} =  \cone(B-a)\cap\pnX{A}(a)$.
\end{example}

We now provide an example where the inclusion
in \eqref{e:pn=cone(B-a)} is strict.
\begin{example}
\label{ex:pn-cone(B-a)}
Suppose that $X=\RR^2$,
let $A := \cone\{(1,0),(0,1)\} =\bd\RP^2$,
$B := \RP(2,1)$, and $a:=(0,1)\in A$.
Then one verifies that
\begin{subequations}
\begin{align}
P^{-1}_A(a) - a&=\left]\minf,1\right]\times\{0\},\\
\pnX{A}(a)&=\cone(P^{-1}_A a - a)=\RR\times\{0\},\\
\cone(B-a)&=\menge{(x_1,x_2)\in\RR^2}{x_1\geq0,2x_2<{x_1}}\cup\{(0,0)\},\\
\pn{A}{B}(a)&=\{(0,0)\}.
\end{align}
\end{subequations}
Hence $\pn{A}{B}(a)=\{(0,0)\}
\subsetneqq \RP\times\{0\} = \cone(B-a)\cap\pnX{A}(a)$,
and therefore the inclusion in \eqref{e:pn=cone(B-a)} is strict.
In accordance with Proposition~\ref{p:pn=cone(B-a)},
neither is $P_A^{-1}(a)-a$ a cone nor is $A$ convex.
\end{example}


Let us now turn to the restricted normal cone $\nc{A}{B}(a)$.
Taking the outer limit in \eqref{e:pn=cone(B-a)}
and recalling \eqref{e:0224b}, we obtain
\begin{subequations}
\label{e:nc=limsup}
\begin{align}
\nc{A}{B}(a)&=\limsup_{x\to a\atop x\in A}\pn{A}{B}(x)\\
&\subseteq\limsup_{x\to a\atop x\in A}\big(\cone(B-x)\cap\pnX{A}(x)\big)\label{e:nc=limsup-b}\\
&\subseteq \big(\limsup_{x\to a\atop x\in A}\cone(B-x)\big)\cap
N_A(a)\label{e:nc=limsup-c}.
\end{align}
\end{subequations}
The inclusions in \eqref{e:nc=limsup} are optimal in the sense
that all possible combinations (strict inclusion and equality)
can occur:
\begin{itemize}
\item For results and examples
illustrating
equality in \eqref{e:nc=limsup-b}
and
equality in \eqref{e:nc=limsup-c},
see Proposition~\ref{p:nc=limsup} and Example~\ref{ex:0225b} below.
\item For an example
illustrating
equality in \eqref{e:nc=limsup-b}
and
strict inequality in \eqref{e:nc=limsup-c},
see Example~\ref{ex:0225c} below.
\item For an example
illustrating
strict inequality in \eqref{e:nc=limsup-b}
and
equality in \eqref{e:nc=limsup-c},
see Example~\ref{ex:0225e} below.
\item For examples
illustrating
strict inequality in \eqref{e:nc=limsup-b}
and
strict inequality in \eqref{e:nc=limsup-c},
see Example~\ref{ex:0225a} and Example~\ref{ex:0225d} below.
\end{itemize}

The remainder of this section is devoted to providing
these examples.


\begin{proposition}
\label{p:nc=limsup1}
Let $A$ and $B$ be nonempty subsets of $X$. Let $A$ be closed $a\in A$.
Assume that one of the following holds:
\begin{enumerate}
\item
\label{p:nc=limsup1i}
$P_A^{-1}(x)-x$ is a cone for every
$x\in A$ sufficiently close to $a$.
\item
\label{p:nc=limsup1ii}
$A$ is convex.
\end{enumerate}
Then \eqref{e:nc=limsup-b} holds with equality, i.e.,
$\nc{A}{B}(a) = \limsup_{x\to a\atop x\in
A}\big(\cone(B-x)\cap\pnX{A}(x)\big)$
\end{proposition}
\begin{proof}
Indeed, if $x\in A$ is sufficiently close to $a$, then
Proposition~\ref{p:pn=cone(B-a)} implies that
$\pn{A}{B}(x)=\cone(B-x)\cap\pnX{A}(x)$.
Now take the outer limit as $x\to a$ in $A$.
\end{proof}

\begin{proposition}
\label{p:nc=limsup}
Let $A$ be a nonempty closed convex subset of $X$,
let $B$ be a nonempty subset of $X$, and let $a\in A$.
Assume that $x\mapsto \cone(B-x)$ is outer semicontinuous at $a$ relative
to $A$, i.e.,
\begin{equation}
\label{e:nc=limsupii}
\limsup_{x\to a\atop x\in A}\cone(B-x)=\cone(B-a),
\end{equation}
Then \eqref{e:nc=limsup} holds with equalities, i.e.,
\begin{equation}
\label{e:0224c}
\nc{A}{B}(a)=
\limsup_{x\to a\atop x\in A}\big(\cone(B-x)\cap\pnX{A}(x)\big)
=\big(\limsup_{x\to a\atop x\in A}\cone(B-x)\big)\cap N_A(a).
\end{equation}
\end{proposition}
\begin{proof}
The convexity of $A$ and Lemma~\ref{l:NsubsetN}\ref{l:NsubsetNvi} yield
\begin{equation}
\cone(B-a)\cap N_A(a)=\cone(B-a)\cap\pnX{A}(a).
\end{equation}
On the other hand,
Proposition~\ref{p:pn=cone(B-a)}\ref{p:pn=cone(B-a)ii}
and Lemma~\ref{l:NsubsetN}\ref{l:NsubsetNiv} imply
\begin{equation}
\cone(B-a)\cap\pnX{A}(a)=\pn{A}{B}(a) \subseteq\nc{A}{B}(a).
\end{equation}
Altogether,
$\cone(B-a)\cap N_A(a)\subseteq \nc{A}{B}(a)$.
In view of \eqref{e:nc=limsupii},
\begin{equation}
\big(\limsup_{x\to a\atop x\in A}\cone(B-x)\big)
\cap N_A(a)\subseteq \nc{A}{B}(a).
\end{equation}
Recalling \eqref{e:nc=limsup}, we therefore obtain \eqref{e:0224c}.
\end{proof}

\begin{example}
\label{ex:0225b}
Let $A$ be a linear subspace of $X$, set $B := A$,
and $a:=(0,0)$.
Then $\nc{A}{B}(a) =  \{0\}$ by \eqref{e:pnALd},
$N_A(a)=A^\perp$, and
$\cone(B-x)=A$, for every $x\in A$.
Hence $(\limsup_{x\to a\atop x\in A}\cone(B-x))\cap N_A(a)=\{0\}$
and \eqref{e:nc=limsup} holds with equalities.
\end{example}


In Proposition~\ref{p:nc=limsup},
the convexity and the outer semicontinuity assumptions are
both \emph{essential} in the sense that absence of either assumption
may make the inclusion \eqref{e:nc=limsup-c} strict;
we shall illustrate this in the next three examples.


\begin{example}
\label{ex:0225c}
Suppose that $X=\RR^2$,
and let $A:=\epi(|\cdot|)$, $B := \RR\times\{0\}$, 
and $a:=(0,0)$.
If $x=(x_1,x_2)\in A\smallsetminus\{a\}$,
then $x_2>0$, $B-x=\RR\times\{-x_2\}$, and so
$\cone(B-x) =\RR\times\RMM\cup\{(0,0)\}$.
Hence
\begin{equation}
\limsup_{x\to a\atop x\in A}
\cone(B-x)=\RR\times\RM\neq\RR\times\{0\}=\cone(B-a),
\end{equation}
i.e., \eqref{e:nc=limsupii} fails.
Since $A$ is closed and convex,
Lemma~\ref{l:NsubsetN}\ref{l:NsubsetNvi} implies that
$N_A(a)=\cnc{A}(a)=-A$.
Thus
\begin{equation}
\big(\limsup_{x\to a\atop x\in A} \cone(B-x)\big)
\cap N_A(a) = -A.
\end{equation}
Proposition~\ref{p:nc=limsup1}\ref{p:nc=limsup1ii} yields
equality in \eqref{e:nc=limsup-b}, i.e.,
\begin{equation}
\nc{A}{B}(a)=\limsup_{x\to a\atop x\in A}
\big(\cone(B-x)\cap\pnX{A}(x)\big).
\end{equation}
Already in Example~\ref{ex:120212a} did we observe that
\begin{equation}
\nc{A}{B}(a)=\cone\{(1,-1),(-1,-1)\}.
\end{equation}
Therefore we have
\begin{equation}
\nc{A}{B}(a)=\limsup_{x\to a\atop x\in
A}\big(\cone(B-x)\cap\pnX{A}(x)\big)\subsetneqq
\big(\limsup_{x\to a\atop x\in A}\cone(B-x)\big)\cap N_A(a),
\end{equation}
i.e., the inclusion \eqref{e:nc=limsup-c} is strict.
\end{example}


\begin{example}
\label{ex:0225a}
Suppose that $X=\RR^2$, and let
$A:=\cone\{(1,0),(0,1)\} = \bd\RP^2$,
$B:=\RR\times\{1\}\cup\{(1,0),(-1,0)\}$,
and $a:=(0,0)$.
Clearly, $A$ is not convex.
If $x=(x_1,x_2)\in A$ is sufficiently close to $a$, we have
\begin{equation}
\label{e:0224d}
\cone(B-x)=
\begin{cases}
\RR\times\RP, & \text{if $x_1\geq 0$;}\\
\RR\times\RPP\cup\cone\{(1,-x_2),(-1,-x_2)\}, &\text{if $x_2>0$.}
\end{cases}
\end{equation}
This yields
\begin{equation}
\label{e:0224f}
\limsup_{x\to a\atop x\in A}\cone(B-x)= \RR\times\RP = \cone(B-a),
\end{equation}
i.e., \eqref{e:nc=limsupii} holds.
Next, if $x=(x_1,x_2)\in A$, then
\begin{equation}
P^{-1}_A(x)=
\begin{cases}
\{x_1\}\times\left]\minf,x_1\right], &\text{if $x_1>0$ and $x_2=0$;}\\
\left]\minf,x_2\right]\times\{x_2\},&\text{if $x_1=0$ and $x_2>0$;}\\
\RM^2, &\text{if $x_1=x_2=0$,}
\end{cases}
\end{equation}
and so
\begin{equation}
\label{e:0224e}
\pnX{A}(x)=\cone\big(P^{-1}_A(x)-x\big)=
\begin{cases}
\{0\}\times\RR,&\text{if $x_1>0$ and $x_2=0$;}\\
\RR\times\{0\}, &\text{if $x_1=0$ and $x_2>0$;}\\
\RM^2, &\text{if $x_1=x_2=0$.}
\end{cases}
\end{equation}
It follows that
\begin{equation}
\label{e:0224i}
N_A(a)=\limsup_{x\to a\atop x\in A}\pnX{A}(x)=
\RM^2\cup \big(\{0\}\times\RR\big) \cup \big(\RR\times\{0\}\big).
\end{equation}
If $x\in A$ is sufficiently close $a$, then
\begin{equation}
\pn{A}{B}(x)=
\begin{cases}
\{(0,0)\},& \text{if $x\neq a$;}\\
\RM\times\{0\}, &\text{if $x=a$.}
\end{cases}
\end{equation}
It follows that
\begin{equation}
\label{e:0224g}
\nc{A}{B}(a)=\RM\times\{0\}.
\end{equation}
Combining \eqref{e:0224d} and \eqref{e:0224e}, we obtain
for every $x=(x_1,x_2)\in A$ sufficiently close to $a$ that
\begin{equation}
\cone(B-x)\cap \pnX{A}(x)=
\begin{cases}
\{0\}\times\RP,&\text{if $x_1>0$ and $x_2=0$;}\\
\{(0,0)\}, &\text{if $x_1=0$ and $x_2>0$;}\\
\RM\times\{0\}, &\text{if $x_1=x_2=0$.}
\end{cases}
\end{equation}
Thus
\begin{equation}
\label{e:0224h}
\limsup_{x\to a\atop x\in A}
\big(\cone(B-x)\cap\pnX{A}(x)\big)=
\big(\{0\}\times\RP\big) \cup \big(\RM\times\{0\}\big).
\end{equation}
Using \eqref{e:0224g}, \eqref{e:0224h},
\eqref{e:0224f}, and \eqref{e:0224i},
we conclude that
\begin{subequations}
\begin{align}
\nc{A}{B}(a)&= \RM\times\{0\}\\
&\subsetneqq \big(\{0\}\times\RP\big)\cup\big(\RM\times\{0\}\big)
= \limsup_{a'\to a\atop a'\in A}\big(\cone(B-x)\cap\pnX{A}(x)\big)\\
&\subsetneqq \big(\{0\}\times\RP\big)\cup\big(\RR\times\{0\}\big) =
\Big(\limsup_{x\to a\atop x\in A}\cone(B-x)\Big)\cap N_A(a).
\end{align}
\end{subequations}
Therefore, both inclusions in \eqref{e:nc=limsup} are strict;
however, $A$ is not convex while \eqref{e:nc=limsupii} does hold.
\end{example}

\begin{example}
\label{ex:0225d}
Suppose that $X=\RR^2$,  let $A:=\cone\{(1,0),(0,1)\}=\bd\RP^2$,
$B:=\RP(2,1)$ and $a:=(0,0)$.
Let $x=(x_1,x_2)\in A$. Then (see Example~\ref{ex:0225a})
\begin{equation}
P^{-1}_A(x)-x=
\begin{cases}
\{0\}\times\left]\minf,x_1\right], &\text{if $x_1>0$ and $x_2=0$;}\\
\left]\minf,x_2\right]\times\{0\},&\text{if $x_1=0$ and $x_2>0$;}\\
\RM^2, &\text{if $x_1=x_2=0$,}
\end{cases}
\end{equation}
\begin{equation}
\label{e:0225a}
\pnX{A}(x)=
\begin{cases}
\{0\}\times\RR,&\text{if $x_1>0$ and $x_2=0$;}\\
\RR\times\{0\}, &\text{if $x_1=0$ and $x_2>0$;}\\
\RM^2, &\text{if $x_1=x_2=0$,}
\end{cases}
\end{equation}
and
\begin{equation}
\label{e:0225g}
N_A(a)=\limsup_{x\to a\atop x\in A}\pnX{A}(x)=
\RM^2\cup \big(\{0\}\times\RR\big) \cup \big(\RR\times\{0\}\big).
\end{equation}
Thus
\begin{equation}
\pn{A}{B}(x)=\cone\big((P^{-1}_A(x)-x)\cap (B-x)\big)=
\begin{cases}
\{0\}\times\RP,& \text{if $x_1>0$ and $x_2=0$;}\\
\{(0,0)\},& \text{if $x_1=0$ and $x_2\geq 0$.}
\end{cases}
\end{equation}
Hence
\begin{equation}
\label{e:0225d}
\nc{A}{B}(a)=\limsup_{x\to a\atop x\in A}\pn{A}{B}(x)= \{0\}\times\RP.
\end{equation}
On the other hand,
\begin{equation}
\label{e:0225b}
\cone(B-x)=
\begin{cases}
\menge{(y_1,y_2)}{y_2\geq0,\,y_1<2y_2}\cup\{(0,0)\},
&\text{if $x_1>0$ and $x_2=0$;}\\
\menge{(y_1,y_2)}{y_1\geq0,\,2y_2<y_1}\cup\{(0,0)\},
&\text{if $x_1=0$ and $x_2>0$;}\\
B, &\text{if $x_1=x_2=0$.}
\end{cases}
\end{equation}
Combining \eqref{e:0225a} and \eqref{e:0225b}, we deduce that
\begin{equation}
\label{e:0225c}
\cone(B-x) \cap \pnX{A}(x) =
\begin{cases}
\{0\}\times\RP,
&\text{if $x_1>0$ and $x_2=0$;}\\
\RP\times\{0\},
&\text{if $x_1=0$ and $x_2>0$;}\\
\{(0,0)\}, &\text{if $x_1=x_2=0$.}
\end{cases}
\end{equation}
Using \eqref{e:0225b} and \eqref{e:0225c}, we compute
\begin{equation}
\label{e:0225f}
\limsup_{x\to a\atop x\in A}\cone(B-x)
=\menge{(y_1,y_2)}{y_1\geq0\text{\, or \,}y_2\geq0} =
X\smallsetminus\RMM^2 \neq B = \cone(B-a)
\end{equation}
and
\begin{equation}
\label{e:0225e}
\limsup_{x\to a\atop x\in A}
\big(\cone(B-x)\cap\pnX{A}(x)\big)=\big(\{0\}\times\RP\big) \cup \big(\RP\times\{0\}\big)
= \cone\{(0,1),(1,0)\}.
\end{equation}
Using \eqref{e:0225d},
\eqref{e:0225e}, \eqref{e:0225f}, and \eqref{e:0225g},
we conclude that
\begin{subequations}
\begin{align}
\nc{A}{B}(a)&= \{0\}\times\RP \\
&\subsetneqq \big(\{0\}\times\RP\big)\cup\big(\RP\times\{0\}\big)
= \limsup_{x\to a\atop x\in A}\big(\cone(B-x)\cap\pnX{A}(x)\big)\\
&\subsetneqq \big(\{0\}\times\RR\big)\cup\big(\RR\times\{0\}\big)
= \big(\limsup_{x\to a\atop x\in A}\cone(B-x)\big)\cap N_A(a).
\end{align}
\end{subequations}
Therefore, both inclusions in \eqref{e:nc=limsup} are strict;
however, $A$ is not convex and \eqref{e:nc=limsupii} does not hold
(see \eqref{e:0225f}).
\end{example}

Finally, we provide an example where
the inclusion \eqref{e:nc=limsup-b} is strict while
the inclusion \eqref{e:nc=limsup-c} is an equality.

\begin{example}
\label{ex:0225e}
Suppose that $X=\RR^2$,
let $A:=\cone\{(1,0),(0,1)\}$,
$B:=\menge{(y_1,y_2)}{y_1+y_2=1}$,
and $a:=(0,0)$.
Let $x=(x_1,x_2)\in A$ be sufficiently close to $a$.
We compute
\begin{subequations}
\begin{align}
\cone(B-x)&=\menge{(y_1,y_2)}{y_1+y_2>0}\cup\{(0,0)\},\label{e:0225h}\\[+2mm]
\pnX{A}(x)&=
\begin{cases}
\{0\}\times \RR, &\text{if $x_1>0$ and $x_2=0$;}\\
\RR\times\{0\},&\text{if $x_1=0$ and $x_2>0$;}\\
\RM^2, &\text{if $x_1=x_2=0$,}
\end{cases}\\[+2mm]
\pn{A}{B}(x) &= \{(0,0)\}.
\end{align}
\end{subequations}
Furthermore, Example~\ref{ex:0225a}
(see \eqref{e:0224i}) implies that
$N_A(a) = \RM^2 \cup\big(\{0\}\times\RR\big)\cup \big(\RR\times\{0\}\big)$.
We thus deduce that
\begin{subequations}
\begin{align}
\nc{A}{B}(a)&= \{(0,0)\}\\
&\subsetneqq \big(\{0\}\times\RP\big)\cup\big(\RP\times\{0\}\big)
= \limsup_{x\to a\atop x\in A}\big(\cone(B-x)\cap\pnX{A}(x)\big)\\
&= \big(\{0\}\times\RP\big)\cup\big(\RP\times\{0\}\big)
= \big(\limsup_{x\to a\atop x\in A}\cone(B-x)\big)\cap N_A(a).
\end{align}
\end{subequations}
Therefore,
the inclusion \eqref{e:nc=limsup-b} is strict while
the inclusion \eqref{e:nc=limsup-c} is an equality.
\end{example}


\section{Constraint qualification conditions and numbers}

\label{s:CQ1}

\label{s:CQnumber}
Utilizing restricted normal cones, we introduce
in this section
the notions of \emph{CQ-number}, \emph{joint-CQ-number},
\emph{CQ condition}, and \emph{joint-CQ condition},
where CQ stands for ``constraint qualification''.

\subsection*{CQ and joint-CQ numbers}

\begin{definition}[CQ-number]
\label{d:CQn}
Let $A$, $\wt{A}$, $B$, $\wt{B}$, be nonempty subsets of $X$,
let $c\in X$, and let $\dd\in\RPP$.
The \emph{CQ-number} at $c$ associated with
$(A,\wt{A},B,\wt{B})$ and $\dd$ is
\begin{equation}
\label{e:CQn}
\theta_\dd:=\theta_\dd\big(A,\wt{A},B,\wt{B}\big)
:=\sup\mmenge{\scal{u}{v}}
{\begin{aligned}
&u\in\pn{A}{\wt{B}}(a),v\in-\pn{B}{\wt{A}}(b),\|u\|\leq 1, \|v\|\leq 1,\\
&\|a-c\|\leq\dd,\|b-c\|\leq\dd.
\end{aligned}}.
\end{equation}
The \emph{limiting CQ-number} at $c$ associated with
$(A,\wt{A},B,\wt{B})$ is
\begin{equation}\label{e:lCQn}
\overline\theta:=\overline\theta\big(A,\wt{A},B,\wt{B}\big)
:=\lim_{\dd\dn0}\theta_\dd\big(A,\wt{A},B,\wt{B}\big).
\end{equation}
\end{definition}
Clearly,
\begin{equation}
\label{e:120405a}
\theta_\dd\big(A,\wt{A},B,\wt{B}\big)=\theta_\dd\big(B,\wt{B},A,\wt{A}\big)
\quad\text{and}\quad
\overline\theta\big(A,\wt{A},B,\wt{B}\big)=\overline\theta\big(B,\wt{B},A,\wt{A}\big).
\end{equation}
Note that, $\delta\mapsto \theta_\delta$ is increasing;
this makes $\overline\theta$ well defined.
Furthermore, since $0$ belongs to nonempty
$B$-restricted proximal normal cones
and because of the Cauchy-Schwarz inequality, we have
\begin{equation}
\label{e:120406a}
c\in \overline{A}\cap \overline{B}\text{~and~}0<\dd_1<\dd_2
\quad\Rightarrow\quad
0\leq\overline{\theta}\leq \theta_{\dd_1}\leq\theta_{\dd_2} \leq 1,
\end{equation}
while $\theta_\dd$, and hence $\overline{\theta}$,
is equal to $\minf$ if $c\notin \overline{A}\cap \overline{B}$
and $\delta$ is sufficiently
small (using the fact that $\sup\varnothing=\minf$).
Using Proposition~\ref{p:elementary}\ref{p:ele-ii}\&\ref{p:ele-vi}, we see that
\begin{equation}
\wt{A}\subseteq A'
\;\text{and}\;
\wt{B}\subseteq B'
\quad\Rightarrow\quad
\theta_\dd(A,\wt{A},B,\wt{B}) \leq \theta_\dd(A,A',B,B')
\end{equation}
and, for every $x\in X$,
\begin{equation}
\label{e:120405b}
\theta_\dd\big(A,\wt{A},B,\wt{B}\big)\text{~at $c$}
\quad = \quad
\theta_\dd\big(A-x,\wt{A}-x,B-x,\wt{B}-x\big)\text{~at $c-x$.}
\end{equation}

To deal with unions,
it is convenient to extend this notion as follows.

\begin{definition}[joint-CQ-number]
\label{d:jCQn}
Let $\mathcal{A} := (A_i)_{i\in I}$, $\wt{\mcA}:=(\wt{A}_i)_{i\in
I}$, $\mathcal{B} := (B_j)_{j\in J}$, $\wt{\mcB}:=(\wt{B}_j)_{j\in
J}$ be nontrivial collections\footnote{The collection $(A_i)_{i\in
I}$ is said to be \emph{nontrivial} if $I\neq\varnothing$.} of
nonempty subsets of $X$, let $c\in X$, and let $\dd\in\RPP$.
The \emph{joint-CQ-number} at
$c$ associated with $(\mathcal{A},\wt{\mcA},\mathcal{B},\wt{\mcB})$
and $\dd$ is
\begin{equation}
\label{e:jCQn}
\theta_\dd=\theta_\dd\big(\mathcal{A},
\wt{\mcA},\mathcal{B},\wt{\mcB}\big):=\sup_{(i,j)\in I\times
J}\theta_\dd\big(A_i,\wt{A}_i,B_j,\wt{B}_j\big),
\end{equation}
and the limiting joint-CQ-number at $c$ associated with
$(\mathcal{A},\wt{\mcA},\mathcal{B},\wt{\mcB})$ is
\begin{equation}
\label{e:ljCQn}
\overline\theta=\overline\theta\big(\mathcal{A}, \wt{\mcA},
\mathcal{B},\wt{\mcB}\big)
:=\lim_{\dd\dn0}\theta_\dd\big(\mathcal{A},\wt{\mcA},\mathcal{B},\wt{\mcB}\big).
\end{equation}
\end{definition}

For convenience, we will simply write $\theta_\dd$,
$\overline\theta$ and omit the possible arguments
$(A,\wt{A},B,\wt{B})$ and
$(\mathcal{A},\wt{\mcA},\mathcal{B},\wt{\mcB})$ when there is no
cause for confusion. If $I$ and $J$ are singletons, then the notions
of CQ-number and joint-CQ-number coincide. Also observe that
\begin{equation}
c\in \bigcup_{i\in I}A_i \cap \bigcup_{j\in J}B_j
\quad\Rightarrow\quad
(\forall \delta\in\RPP)\;\; 0\leq \overline{\theta}\leq\theta_\delta\leq 1
\end{equation}
while $\overline{\theta}=\theta_\dd=\minf$ when
$\dd>0$ is sufficiently small and $c$ does not belong to both
$\overline{\bigcup_{i\in I}A_i}$ and
$\overline{\bigcup_{j\in J}B_j}$.
Furthermore,
the joint-CQ-number (and hence the limiting joint-CQ-number as well)
really depends only on those sets $A_i$ and $B_j$ for which
$c\in \overline{A_i}\cap \overline{B_j}$.



To illustrate this notion, let us compute
the CQ-number of two lines.
The formula provided is the cosine of the angle
between the two lines --- as we shall see in Theorem~\ref{t:CQn=c} below,
this happens actually for all linear subspaces although
then the angle must be defined appropriately and
the proof is more involved.

\begin{proposition}[CQ-number of two distinct lines through the origin]
\label{p:CQn2l}
Suppose that $w_a$ and $w_b$ are two vectors in $X$ such that
$\|w_a\|=\|w_b\|=1$.
Let $A :=\RR w_a$, $B:= \RR w_b$, and $\dd\in\RPP$.
Assume that $A\cap B = \{0\}$.
Then the CQ-number at $0$ is
\begin{equation}
\theta_\dd(A,A,B,B)=|\scal{w_a}{w_b}|.
\end{equation}
\end{proposition}
\begin{proof}
Set $s := \scal{w_a}{w_b}$.

Assume first that $s\neq 0$.
Let $a = \alpha w_a\in A$
and $b=\beta w_b\in B$.
Then $P_A^{-1}(a)-a = N_A(a) = \{w_a\}^\perp$;
considering  $(B-a)\cap \{w_a\}^\perp$ leads to $\beta s = \alpha$.
Hence  $(P_A^{-1}(a)-a)\cap (B-a)=\beta w_b - \alpha w_a$ and
\begin{equation}
\pn{A}{B}(a) = \cone\big(\alpha s^{-1}w_b-\alpha w_a).
\end{equation}
Similarly,
\begin{equation}
-\pn{B}{A}(b) = \cone\big(\beta w_b-\beta s^{-1}w_a).
\end{equation}
Now set $u := \alpha s^{-1}w_b-\alpha w_a\in \pn{A}{B}(a)$
and $v := \beta w_b-\beta s^{-1}w_a \in -\pn{B}{A}(b)$.
One computes
\begin{equation}
\|u\|=\frac{|\alpha|\sqrt{1-s^2}}{|s|},
\;\;
\|v\|=\frac{|\beta|\sqrt{1-s^2}}{|s|},
\;\;\text{and}\;\;
\scal{u}{v} = \frac{\alpha\beta(1-s^2)}{s}.
\end{equation}
Hence
\begin{equation}
\frac{\scal{u}{v}}{\|u\|\cdot\|v\|} = \sgn(\alpha)\sgn(\beta) s.
\end{equation}
Choosing $\alpha$ and $\beta$ in $\{-1,1\}$ appropriately,
we arrange for $\scal{u}{v}/(\|u\|\cdot\|v\|) = |s|$, as claimed.

Now assume that $s=0$. Arguing similarly,
we see that
\begin{equation}
(\forall a\in A)\quad
\pn{A}{B}(a) = \begin{cases} \{0\},&\text{if $a\neq 0$;}\\
B, &\text{if $a=0$,} \end{cases}
\quad\text{and}\quad
(\forall b\in B)\quad
\pn{B}{A}(b) = \begin{cases} \{0\},&\text{if $b\neq 0$;}\\
A, &\text{if $b=0$.} \end{cases}
\end{equation}
This leads to $\theta_\dd(A,A,B,B)=0=|s|$, again as claimed.
\end{proof}

Let $\mcA := (A_i)_{i\in I}$, $\wt{\mcA} :=
(\wt{A}_i)_{i\in I}$, $\mcB := (B_j)_{j\in J}$ and $\wt{\mcB} :=
(\wt{B}_j)_{j\in J}$ be nontrivial collections of nonempty closed
subsets of $X$  and let $\dd\in\RPP$. Set $A := \bigcup_{i\in I}
A_i$, $\wt{A} := \bigcup_{i\in I} \wt{A}_i$, $B := \bigcup_{j\in J}
B_j$, $\wt{B} := \bigcup_{j\in J} \wt{B}_j$, and suppose that
$c\in A\cap B$.
It is interesting to compare the joint-CQ-number of
collections, i.e., $\theta_\delta\big(\mcA,\wt{\mcA},
\mcB,\wt{\mcB}\big)$, to the CQ-number of the unions, i.e.,
$\theta_\delta\big(A,\wt{A},B,\wt{B}\big)$.
We shall see in the following two examples that \emph{neither
of them is smaller than the other}; in fact,
one of them can be equal to 1 while the other
is {strictly} less than 1.


\begin{example}[joint-CQ-number $<$ CQ-number of the unions]
\label{ex:jCQn<CQn}
Suppose that $X=\RR^3$,
let $I:=J:=\{1,2\}$,
$A_1:=\RR(0,1,0)$,
$A_2:=\RR(2,0,-1)$, $B_1:=\RR(0,1,1)$, $B_2:=\RR(1,0,0)$,
$c:=(0,0,0)$, and let $\dd>0$.
Furthermore, set
$\mcA:=(A_i)_{i\in I}$,
$\mcB:=(B_j)_{j\in J}$, $A:= A_1\cup A_2$,
and $B:= B_1\cup B_2$. Then
\begin{equation}
\theta_\delta\big(\mcA,\mcA,\mcB,\mcB\big)
=\tfrac{2}{\sqrt{5}} < 1 =
\theta_\delta\big(A,A,B,B\big).
\end{equation}
\end{example}
\begin{proof}
Using Proposition~\ref{p:CQn2l}, we compute,
for the reference point $c$,
\begin{subequations}
\begin{align}
\theta_\dd(A_1,A_1,B_1,B_1)&=
\big|\bscal{(0,1,0)}{\tfrac{1}{\sqrt{2}}(0,1,1)}\big|=\tfrac{1}{\sqrt{2}},\\
\theta_\dd(A_1,A_1,B_2,B_2)&=|\scal{(0,1,0)}{(1,0,0)}|=0,\\
\theta_\dd(A_2,A_2,B_1,B_1)&=
\big|\bscal{\tfrac{1}{\sqrt{5}}(2,0,-1)}{\tfrac{1}{\sqrt{2}}(0,1,1)}\big|
=\tfrac{1}{\sqrt{10}},\\
\theta_\dd(A_2,A_2,B_2,B_2)&=
\big|\bscal{\tfrac{1}{\sqrt{5}}(2,0,-1)}{(1,0,0}\big|
=\tfrac{2}{\sqrt{5}}.
\end{align}
\end{subequations}
Hence $\theta_\dd(\mcA,\mcA,\mcB,\mcB)
=\max_{(i,j)\in I\times J}\theta_\dd(A_i,A_i,B_j,B_j)=\tfrac{2}{\sqrt{5}}<1$.

To estimate the CQ-number of the union, set
\begin{equation}
a:=(0,\dd,0)\in A_1\subseteq A \text{~and~}
b:=(\dd,0,0)\in B_2\subseteq B.
\end{equation}
Note that $\|a-c\|=\|a\|=\dd$ and $\|b-c\|=\|b\|=\dd$.
Now define
\begin{equation}
\wt{a}:=(\dd,0,-\dd/2)\in A_2\subseteq {A}
\text{~~and~~}
\wt{b}:=(0,\dd,\dd)\in B_1\subseteq {B}.
\end{equation}
Since $\|\wt{a}-P_{B_2}\wt{a}\|<\|\wt{a}-P_{B_1}\wt{a}\|$
and $P_{B_2}\wt{a}=b$, we have
$b=P_B\wt{a}$.
Since $\|\wt{b}-P_{A_1}\wt{b}\|<\|\wt{b}-P_{A_2}\wt{b}\|$
and $P_{A_1}\wt{b}=a$, we have
$a=P_A\wt{b}$.
Therefore,
$\wt{b}\in B\cap P^{-1}_A(a)$ and
$\wt{a}\in A\cap P^{-1}_B(b)$.
It follows that
\begin{subequations}
\begin{align}
u&:=\tfrac{1}{\dd}(\wt{b}-a)=(0,0,1)\in\pn{A}{B}(a),\\
v&:=\tfrac{2}{\dd}(b-\wt{a})=(0,0,1)\in -\pn{B}{A}(b).
\end{align}
\end{subequations}
Since $\|u\|=\|v\|=1$, we obtain
$1=\scal{u}{v}\leq\theta_\dd(A,A,B,B)\leq 1$.
\end{proof}

\begin{example}[CQ-number of the unions $<$ joint-CQ-number]
Suppose that $X=\RR$, let $I := J:= \{1,2\}$,
$A_1:=B_1:=\RM$, $A_2:=B_2:=\RP$,
$c:=0$, and $\dd>0$.
Furthermore, set
$\mcA := (A_i)_{i\in I}$, $\mcB := (B_j)_{j\in I}$,
$A:= A_1\cup A_2=\RR$, and $B:= B_1\cup B_2=\RR$.
Then
\begin{equation}
\theta_\delta\big(A,A,B,B\big)=0 < 1 =
\theta_\delta\big(\mcA,\mcA,\mcB,\mcB\big).
\end{equation}
\end{example}
\begin{proof}
Lemma~\ref{l:NsubsetN}\ref{l:NsubsetNri} implies that
$(\forall x\in \RR)$ $\pn{\RR}{\RR}(x)=\{0\}$.
Hence $\theta_\delta(\RR,\RR,\RR,\RR)=0$ as claimed.
On the other hand, $\pn{\RP}{\RM}(0)=\RM$
and $\pn{\RM}{\RP}(0)=\RP$.
Hence $\theta_\delta(\RM,\RM,\RP,\RP)=1$
and therefore
$\theta_\delta\big(\mcA,\mcA,\mcB,\mcB\big)=1$ as well.
\end{proof}

The two preceding examples illustrated the independence of the
two types of CQ-numbers (for the collection and for the union).
In some cases, such as Example~\ref{ex:jCQn<CQn}, it is beneficial to
work with a suitable partition to obtain a CQ-number that is less than one,
which in turn is very desirable in applications (see
Section~\ref{s:application}).


\subsection*{CQ and joint-CQ conditions}

\begin{definition}[CQ and joint-CQ conditions]
\label{d:tildeCQ}
Let $c\in X$.
\begin{enumerate}
\item Let $A$, $\wt{A}$, $B$ and $\wt{B}$ be nonempty subsets of $X$.
Then the \emph{$(A,\wt{A},B,\wt{B})$-CQ condition} holds at $c$ if
\begin{equation}\label{e:CQ}
\nc{A}{\wt{B}}(c) \cap\big(-\nc{B}{\wt{A}}(c)\big)\subseteq\{0\}.
\end{equation}
\item
Let $\mathcal{A} := (A_i)_{i\in I}$, $\wt{\mcA} := (\wt{A}_i)_{i\in
I}$, $\mathcal{B} := (B_j)_{j\in J}$ and $\wt{\mcB} :=
(\wt{B}_j)_{j\in J}$ be nontrivial collections of nonempty subsets of
$X$.
Then the \emph{$(\mcA,\wt{\mcA},\mcB,\wt{\mcB})$-joint-CQ condition}
holds at $c$ if for every $(i,j)\in I\times J$, the
$(A_i,\wt{A}_i,B_j,\wt{B}_j)$-CQ condition holds at $c$, i.e.,
\begin{equation}\label{e:jCQ}
\big(\forall (i,j)\in I\times J\big)\quad \nc{A_i}{\wt{B}_j}(c)
\cap\big(-\nc{B_j}{\wt{A}_i}(c)\big)\subseteq\{0\}.
\end{equation}
\end{enumerate}
\end{definition}

In view of the definitions, the key case to consider is when $c\in
A\cap B$ (or when $c\in A_i\cap B_j$ in the joint-CQ case). The
CQ-number is based on the behavior of the restricted proximal normal
cone in a neighborhood of the point under consideration --- a related
notion is that of the exact CQ-number, where we consider the
restricted normal cone at the point instead of
nearby restricted proximal normal cones.

\begin{definition}[exact CQ-number and exact joint-CQ-number]
\label{d:exactCQn}
Let $c\in X$.
\begin{enumerate}
\item
Let $A$, $\wt{A}$, $B$ and $\wt{B}$ be nonempty subsets of $X$. The
\emph{exact CQ-number} at $c$ associated with $(A,\wt{A},B,\wt{B})$
is \footnote{Note that if $c\notin A\cap B$, then $\overline{\alpha} =
\sup\varnothing=\minf$.}
\begin{equation}
\label{e:0217a}
\overline{\alpha} :=
\overline{\alpha}\big(A,\wt{A},B,\wt{B}\big) :=
\sup\mmenge{\scal{u}{v}}{u\in\nc{A}{\wt{B}}(c),v\in-\nc{B}{\wt{A}}(c),\|u\|\leq
1, \|v\|\leq 1}.
\end{equation}
\item
Let $\mathcal{A} := (A_i)_{i\in I}$, $\wt{\mcA} := (\wt{A}_i)_{i\in
I}$, $\mathcal{B} := (B_j)_{j\in J}$ and $\wt{\mcB} :=
(\wt{B}_j)_{j\in J}$ be nontrivial collections of nonempty subsets
of $X$. The \emph{exact joint-CQ-number} at $c$ associated with
$(\mathcal{A},\mathcal{B},\wt{\mcA},\wt{\mcB})$ is
\begin{equation}
\label{e:0217b}
\overline{\alpha} := \overline{\alpha}(\mcA,\wt{\mcA},\mcB,\wt{\mcB}) :=
\sup_{(i,j)\in I\times
J}\overline{\alpha}(A_i,\wt{A}_i,B_j,\wt{B}_j).
\end{equation}
\end{enumerate}
\end{definition}

The next result relates the various condition numbers defined above.

\begin{theorem}\label{t:CQ1}
Let $\mathcal{A} := (A_i)_{i\in I}$, $\wt{\mcA} := (\wt{A}_i)_{i\in
I}$, $\mathcal{B} := (B_j)_{j\in J}$ and $\wt{\mcB} :=
(\wt{B}_j)_{j\in J}$ be nontrivial collections of nonempty subsets
of $X$. Set $A := \bigcup_{i\in I} A_i$ and $B := \bigcup_{j\in
J}B_j$, and suppose that $c\in A\cap B$. Denote the exact
joint-CQ-number at $c$ associated with
$(\mcA,\wt{\mcA},\mcB,\wt{\mcB})$ by $\overline{\alpha}$ (see
\eqref{e:0217b}), the joint-CQ-number at $c$ associated with
$(\mcA,\wt{\mcA},\mcB,\wt{\mcB})$ and $\delta>0$ by $\theta_\dd$
(see \eqref{e:jCQn}), and the limiting joint-CQ-number at $c$
associated with $(\mcA,\wt{\mcA},\mcB,\wt{\mcB})$ by
$\overline{\theta}$ (see \eqref{e:ljCQn}). Then the following hold:
\begin{enumerate}
\item
\label{t:CQ1i} If $\overline\alpha<1$, then the
$(\mcA,\wt{\mcA},\mcB,\wt{\mcB})$-CQ condition holds at $c$.
\item
\label{t:CQ1ii}
$\overline{\alpha}\leq\theta_\delta$.
\item
\label{t:CQ1iii}
$\overline{\alpha}\leq\overline{\theta}$.
\end{enumerate}
Now assume in addition that
$I$ and $J$ are finite.
Then the following hold:
\begin{enumerate}[resume]
\item
\label{t:CQ1iv}
$\overline{\alpha}=\overline{\theta}$.
\item
\label{t:CQ1v}
The $(\mcA,\wt{\mcA},\mcB,\wt{\mcB})$-joint-CQ condition holds at $c$
if and only if
$\overline{\alpha}=\overline{\theta}<1$.
\end{enumerate}
\end{theorem}
\begin{proof}
\ref{t:CQ1i}: Suppose that $\overline\alpha<1$. The condition for
equality in the Cauchy-Schwarz inequality implies that for all
$(i,j)\in I\times J$, the intersection $\nc{A_i}{\wt{B}_j}(c)\cap
(-\nc{B_j}{\wt{A}_i}(c))$ is either empty or $\{0\}$. In view of
Definition~\ref{d:tildeCQ}, we see that the
$(\mcA,\wt{\mcA},\mcB,\wt{\mcB})$-joint-CQ holds at $c$.

\ref{t:CQ1ii}:
Let $(i,j)\in I\times J$. Take
$u\in\nc{A_i}{\wt{B}_j}(c)$ and $v\in-\nc{B_j}{\wt{A}_i}(c)$ such
that $\|u\|\leq 1$ and $\|v\|\leq 1$. Then, by definition of the
restricted normal cone, there exist sequences
$(a_n)_\nnn$ in $A_i$,
$(b_n)_\nnn$ in $B_j$,
$(u_n)_\nnn$ and $(v_n)_\nnn$ in $X$ such that
$a_n\to c$, $b_n\to c$, $u_n\to u$, $v_n\to v$, and
$(\forall\nnn)$ $u_n\in\pn{A_i}{\wt{B}_j}(a_n)$ and
$v_n\in-\pn{B_j}{\wt{A}_i}(b_n)$.
Note that since $\delta>0$,
eventually $a_n$ and $b_n$ lie in $\ball{c}{\dd}$; consequently,
$\scal{u_n}{v_n}\leq \theta_\dd(A_i,\wt{A}_i,B_j,\wt{B}_j)$.
Taking the limit as $n\to\pinf$, we obtain
$\scal{u}{v}\leq\theta_\dd(A_i,\wt{A}_i,B_j,\wt{B}_j)\leq\theta_\dd$.
Now taking the supremum over suitable $u$ and $v$, followed by
taking the supremum over $(i,j)$, we conclude that
$\overline{\alpha}\leq \theta_\dd$.

\ref{t:CQ1iii}: This is clear from \ref{t:CQ1ii} and \eqref{e:ljCQn}.

\ref{t:CQ1iv}:
Let $(\dd_n)_\nnn$ be a sequence in $\RPP$ such that $\delta_n\to 0$.
Then for every $\nnn$, there exist
\begin{equation}
i_n\in I,\ j_n\in J,\
a_n\in A_{i_n},\ b_n\in B_{j_n},\
u_n\in\pn{A_{i_n}}{\wt{B}_{j_n}}(a_n),\
v_n\in-\pn{B_{j_n}}{\wt{A}_{i_n}}(b_n)
\end{equation}
such that
\begin{equation}
\|a_n-c\|\leq \dd_n,\ \|b_n-c\|\leq \dd_n,\
\|u_n\|\leq1,\ \|v_n\|\leq1,\
\text{~and~}\scal{u_n}{v_n}>\theta_{\dd_n}-\dd_n.
\end{equation}
Since $I$ and $J$
are finite, and after passing to a subsequence and relabeling if
necessary, we can and do assume that
there exists $(i,j)\in I\times J$ such that
$u_n\to u\in \nc{A_i}{\wt{B}_j}(c)$
and $v_n\to v\in -\nc{B_j}{\wt{A}_i}(c)$.
Hence
$\overline{\theta} \leftarrow \theta_{\delta_n}-\delta_n
<\scal{u_n}{v_n}\to\scal{u}{v}\leq\overline{\alpha}$.
Hence $\overline{\theta}\leq\overline{\alpha}$.
On the other hand, $\overline{\alpha}\leq\overline{\theta}$
by \ref{t:CQ1iii}. Altogether,
$\overline{\alpha}=\overline{\theta}$.

\ref{t:CQ1v}:
``$\Rightarrow$'':
Let $(i,j)\in I\times J$.
If $c\not\in {A_i}\cap {B_j}$,
then $\overline\alpha(A_i,\wt{A}_i,B_j,\wt{B}_j)=\minf$.
Now assume that $c\in {A_i}\cap {B_j}$.
Since the $(\mcA,\wt{\mcA},\mcB,\wt{\mcB})$-joint-CQ condition holds,
we have
$\nc{A_i}{\wt{B}_j}(c)\cap -\nc{B_j}{\wt{A}_i}(c)=\{0\}$.
By Cauchy-Schwarz,
\begin{equation}
\overline\alpha(A_i,\wt{A}_i,B_j,\wt{B}_j)=
\sup\mmenge{\scal{u}{v}}{u\in\nc{A_i}{\wt{B}_j}(c),v\in-\nc{B_j}{\wt{A}_i}(c),\|u\|\leq
1, \|v\|\leq 1}<1.
\end{equation}
Since $I$ and $J$ are finite and because of \ref{t:CQ1iv},
we deduce that $\overline{\theta}=\overline\alpha<1$.\\
``$\Leftarrow$'': Combine \ref{t:CQ1i} with \ref{t:CQ1iv}.
\end{proof}


\section{CQ conditions and CQ numbers: examples}

\label{s:CQ2}


In this section, we provide further results and examples
illustrating CQ conditions and CQ numbers.

First, let us note that the assumption that the sets of indices be finite in
Theorem~\ref{t:CQ1}\ref{t:CQ1iv} is essential:

\begin{example}[$\overline{\alpha}<\overline{\theta}$]
Suppose that $X=\RR^2$,
let $\Gamma\subseteq\RPP$ be such that $\sup\Gamma=\pinf$,
set $(\forall\gamma\in\Gamma)$ $A_\gamma :=
\epi(\thalb\gamma|\cdot|^2)$,
$B := \RP\times\RR$,
$\mcA := (A_\gamma)_{\gamma\in\Gamma}$,
$\wt{\mcA} := (X)_{\gamma\in\Gamma}$,
$\mcB := (B)$,
$\wt{\mcB} := (X)$,
and $c:=(0,0)$.
Denote the exact
joint-CQ-number at $c$ associated with
$(\mcA,\wt{\mcA},\mcB,\wt{\mcB})$ by $\overline{\alpha}$ (see
\eqref{e:0217b}), the joint-CQ-number at $c$ associated with
$(\mcA,\wt{\mcA},\mcB,\wt{\mcB})$ and $\delta>0$ by $\theta_\dd$
(see \eqref{e:jCQn}), and the limiting joint-CQ-number at $c$
associated with $(\mcA,\wt{\mcA},\mcB,\wt{\mcB})$ by
$\overline{\theta}$ (see \eqref{e:ljCQn}). Then
\begin{equation}
\overline{\alpha}=0<1 = \theta_\dd = \overline{\theta}.
\end{equation}
\end{example}
\begin{proof}
Let $\gamma\in\Gamma$ and pick $x>0$ such that
$a := (x,\thalb\gamma x^2)\in A_\gamma$ satisfies
$\|a\|=\|a-c\|=\delta$, i.e.,
$x>0$ and
\begin{equation}
\gamma^2x^2 = 2\Big(\sqrt{1+\gamma^2\delta^2}-1\Big)
\to\pinf \quad\text{as $\gamma\to \pinf$ in $\Gamma$.}
\end{equation}
Hence
\begin{equation}
\label{e:0303b}
\gamma x\to \pinf,
\quad\text{as $\gamma\to \pinf$ in $\Gamma$.}
\end{equation}
Since $A_\gamma$ is closed and convex, it follows from
Lemma~\ref{l:NsubsetN}\ref{l:NsubsetNvi} that
\begin{equation}
u := \frac{(\gamma x,-1)}{\sqrt{\gamma^2 x^2+1}}
\in \RP(\gamma x,-1) = \cnc{A_\gamma}(a) = \pn{A_\gamma}{X}(a) =
\nc{A_\gamma}{X}(a) = N_{A_\gamma}(a).
\end{equation}
Furthermore, $v := (1,0)\in -(\RM\times\{0\})=
-\pn{B}{X}(c) = -\nc{B}{X}(c)= - N_B(c)$,
$\|u\|=\|v\|=1$,
and, in view of \eqref{e:0303b},
\begin{subequations}
\begin{align}
1 &\geq \theta_\dd \geq \theta_\dd(A_\gamma,X,B,X)
\geq \scal{u}{v} = \frac{\gamma x}{\sqrt{\gamma^2 x^2+1}}\\
& \to 1 \quad\text{as $\gamma\to \pinf$ in $\Gamma$.}
\end{align}
\end{subequations}
Thus $\theta_\dd=1$, which implies that $\overline{\theta}=1$.
Finally, $N_{A_\gamma}(c) = (\{0\}\times\RM)
\perp (\RP\times\{0\}) = -N_B(c)$, which shows that
$\overline{\alpha}=0$.
\end{proof}

For the eventual application of these results to the method of alternating
projections, the condition $\overline{\alpha}=\overline{\theta}<1$ is
critical to ensure linear convergence.

The following example illustrates that
the CQ-number can be interpreted as a quantification of the CQ condition.

\begin{example}[CQ-number quantifies CQ condition]
\label{ex:compareCQ1}
Let $A$ and $B$ be subsets of $X$,
and suppose that $c\in A\cap B$.
Let $L$ be an affine subspace of $X$ containing $A\cup B$.
Then the following are equivalent:
\begin{enumerate}
\item
\label{t:cCQ1-i} $\nc{A}{L}(c) \cap (-\nc{B}{L}(c))= \{0\}$,
i.e., the $(A,L,B,L)$-CQ condition holds at $c$ (see \eqref{e:CQ}).
\item
\label{t:cCQ1-ii}
$N_A(c) \cap (-N_B(c))\cap (L-c) = \{0\}$.
\item
\label{t:cCQ1-iii}
$\overline{\theta}<1$,
where $\overline{\theta}$ is the limiting CQ-number at $c$
associated with $(A,L,B,L)$ (see \eqref{e:lCQn}).
\end{enumerate}
\end{example}
\begin{proof}
The identity \eqref{e:pnALd} of Theorem~\ref{p:pnA(L)} yields
$\nc{A}{L}(c)=N_A(c)\cap(L-c)$ and
$\nc{B}{L}(c)=N_B(c)\cap(L-c)$.
Hence
\begin{equation}
\nc{A}{L}(c) \cap \big(-\nc{B}{L}(c)\big)=
N_A(c)\cap \big(-N_B(c)\big)\cap (L-c),
\end{equation}
and the equivalence of
\ref{t:cCQ1-i} and \ref{t:cCQ1-ii} is now clear.
Finally, Theorem~\ref{t:CQ1}\ref{t:CQ1iv}\&\ref{t:CQ1v} yields
the equivalence of \ref{t:cCQ1-i} and \ref{t:cCQ1-iii}.
\end{proof}

Depending on the choice of the restricting sets $\wt{A}$ and $\wt{B}$,
the $(A,\wt{A},B,\wt{B})$-CQ condition may either hold or fail:

\begin{example}[CQ condition depends on restricting sets]
\label{ex:CQdif(AB)}
Suppose that $X=\RR^2$, and set
$A:=\epi(|\cdot|)$, $B:=\RR\times\{0\}$, and $c:=(0,0)$.
Then we readily verify that
$N_A(c) = \nc{A}{X}(c) = -A$,
$\nc{A}{B}(c) = -\bd A$,
$N_B(c)=\nc{B}{X}(c) =\{0\}\times\RR$,
and $\nc{B}{A}(c) = \{0\}\times\RP$.
Consequently,
\begin{equation}
\nc{A}{X}(c)\cap\big(-\nc{B}{X}(c)\big) = \{0\}\times\RM
\text{~~while~~}
\nc{A}{B}(c)\cap\big(-\nc{B}{A}(c)\big) = \{(0,0)\}.
\end{equation}
Therefore, the $(A,A,B,B)$-CQ condition holds, yet the
$(A,X,B,X)$-CQ condition fails.
\end{example}


For two spheres, it is possible to
quantify the convergence of $\theta_\dd$ to
$\overline{\dd}=\overline{\alpha}$:

\begin{proposition}[CQ-numbers of two spheres]
\label{p:sphere2}
Let $z_1$ and $z_2$ be in $X$,
let $\rho_1$ and $\rho_2$ be in $\RPP$,
set $S_1 := \sphere{z_1}{\rho_1}$ and $S_2 := \sphere{z_2}{\rho_2}$
and assume that $c\in S_1\cap S_2$.
Denote
the limiting CQ-number at $c$ associated with $(S_1,X,S_2,X)$
by $\overline{\theta}$ (see Definition~\ref{d:CQn}), and
the exact CQ-number at $c$ associated with
$(S_1,X,S_2,X)$ by $\overline{\alpha}$ (see Definition~\ref{d:exactCQn}).
Then the following hold:
\begin{enumerate}
\item
\label{p:sphere2i}
$\displaystyle \overline{\theta} = \overline{\alpha} = \frac{|\scal{z_1-c}{z_2-c}|}{\rho_1\rho_2}$.
\item
\label{p:sphere2ii}
$\overline{\alpha}<1$ unless the spheres are identical or intersect only at
$c$.
\end{enumerate}
Now assume that $\overline{\alpha}<1$, let $\ve\in\RPP$, and set
$\dd := (\sqrt{(\rho_1+\rho_2)^2+4\rho_1\rho_2\ve}-(\rho_1+\rho_2))/2>0$.
Then
\begin{equation}
\label{e:0308d}
\overline\alpha\leq \theta_\dd\leq\overline\alpha+\ve,
\end{equation}
where
$\theta_\dd$ is the CQ-number at $c$ associated with $(S_1,X,S_2,X)$
(see Definition~\ref{d:CQn}).

\end{proposition}
\begin{proof}
\ref{p:sphere2i}:
This follows from Theorem~\ref{t:CQ1}\ref{t:CQ1iv} and
Example~\ref{ex:ncS}.

\ref{p:sphere2ii}: Combine \ref{p:sphere2i} with
the characterization of equality in the Cauchy-Schwarz inequality.

Let us now establish \eqref{e:0308d}.
By Theorem~\ref{t:CQ1}\ref{t:CQ1ii}, 
we have $\overline{\alpha}\leq\theta_\dd$.
Let $s_1\in S_1$ be such that $\|s_1-c\|\leq\dd$,
let $u_1\in\pn{S_1}{X}(s_1)$ be such that  $\|u_1\|=1$,
let $s_2\in S_2$ be such that $\|s_2-c\|\leq\dd$,
and let $u_2\in\pn{S_2}{X}(s_2)$ be such that  $\|u_2\|=1$.
By Example~\ref{ex:ncS},
\begin{equation}
u_1=\pm\frac{s_1-z_1}{\|s_1-z_1\|}=\pm\frac{s_1-z_1}{\rho_1}
\quad\text{and}\quad
u_2=\pm\frac{s_2-z_2}{\|s_2-z_2\|}=\pm\frac{s_2-z_2}{\rho_2}.
\end{equation}
Hence
\begin{subequations}
\begin{align}
\rho_1\rho_2\scal{u_1}{u_2}
&\leq |\scal{s_1-z_1}{s_2-z_2}|\\
&=|\scal{(s_1-c)+(c-z_1)}{(s_2-c)+(c-z_2)}|\\
&\leq |\scal{s_1-c}{s_2-c}| + |\scal{s_1-c}{c-z_2}|\\
&\qquad + |\scal{c-z_1}{s_2-c}| + |\scal{c-z_1}{c-z_2}|\\
&\leq \dd^2 + \dd(\rho_1+\rho_2) + \rho_1\rho_2\overline{\alpha}
\end{align}
\end{subequations}
and thus, using the definition of $\dd$,
\begin{equation}
\scal{u_1}{u_2}\leq \overline{\alpha}
+ \frac{\dd^2+\dd(\rho_1+\rho_2)}{\rho_1\rho_2}
= \overline{\alpha}+\ve.
\end{equation}
Therefore, by the definition of $\theta_\dd$, we have
$\theta_\dd\leq\overline{\alpha}+\ve$.
\end{proof}


\subsection*{Two convex sets}

Let us turn to the classical convex setting.
We start by noting
that well known constraint qualifications are conveniently
characterized using our CQ conditions.

\begin{proposition}
\label{p:0301a}
Let $A$ and $B$ be nonempty convex subsets of $X$ such that $A\cap
B\neq\varnothing$, and set $L=\aff(A\cup B)$.
Then the following are equivalent:
\begin{enumerate}
\item
\label{p:0301ai}
$\reli A\cap \reli B\neq \varnothing$.
\item
\label{p:0301aii}
The $(A,L,B,L)$-CQ condition holds at some point in $A\cap B$.
\item
\label{p:0301aiii}
The $(A,L,B,L)$-CQ condition holds at every point in $A\cap B$.
\end{enumerate}
\end{proposition}
\begin{proof}
This is clear from Theorem~\ref{t:compareCQ2}.
\end{proof}

\begin{proposition}
\label{p:0301b}
Let $A$ and $B$ be nonempty convex subsets of $X$ such that $A\cap
B\neq\varnothing$.
Then the following are equivalent:
\begin{enumerate}
\item
\label{p:0301bi}
$0\in\inte(B-A)$.
\item
\label{p:0301bii}
The $(A,X,B,X)$-CQ condition holds at some point in $A\cap B$.
\item
\label{p:0301biii}
The $(A,X,B,X)$-CQ condition holds at every point in $A\cap B$.
\end{enumerate}
\end{proposition}
\begin{proof}
This is clear from Corollary~\ref{c:compareCQ3}.
\end{proof}

In stark contrast to Proposition~\ref{p:0301a} and \ref{p:0301b},
if the restricting sets are not both equal to $L$ or to $X$,
then the CQ-condition may actually depend on the reference point
as we shall illustrate now:

\begin{example}[CQ condition depends on the reference point]
Suppose that $X=\RR^2$, and let $f\colon\RR\to\RR\colon
x\mapsto (\max\{0,x\})^2$, which is a continuous convex function.
Set $A := \epi f$ and $B:=\RR\times\{0\}$, which are closed convex subsets
of $X$.
Consider first the point $c:=(-1,0)\in A\cap B$.
Then
$\nc{A}{B}(c)=\{(0,0)\}$ and $\nc{B}{{A}}(c)=\{0\}\times\RP$;
hence,
\begin{equation}
\nc{A}{{B}}(c)\cap\big(-\nc{B}{A}(c)\big)=\{(0,0)\},
\end{equation}
i.e., the $(A,A,B,B)$-CQ condition holds at $c$.
On the other hand, consider now $d:=(0,0)\in A\cap B$.
Then $\nc{A}{B}(d) = \{0\}\times\RM$ and
$\nc{B}{A}(d)=\{0\}\times\RP$;
thus,
\begin{equation}
\nc{A}{{B}}(d)\cap\big(-\nc{B}{A}(d)\big)=\{0\}\times\RM,
\end{equation}
i.e., the $(A,A,B,B)$-CQ condition fails at $d$.
\end{example}


\subsection*{Two linear (or intersecting affine) subspaces}

We specialize further to two linear subspaces of $X$.
A pleasing connection
between CQ-number and the angle between two linear subspaces will
be revealed. But first we provide some auxiliary results.

\begin{proposition}
\label{p:nc-subsp} Let $A$ and $B$ be linear subspaces of $X$,
and let $\dd\in\RPP$.
Then
\begin{equation}
\label{e:0301a}
\bigcup_{a\in A\cap(B+A^\perp)\cap \ball{0}{\delta}}\pn{A}{B}(a)=
\bigcup_{a\in A\cap \ball{0}{\delta}}\pn{A}{B}(a)=
\bigcup_{a\in A}\pn{A}{B}(a)
=A^\bot\cap(A+B).
\end{equation}
\end{proposition}
\begin{proof}
Let $a\in A$. Then $P_A^{-1}(a) = a+A^\perp$ and
hence $P_A^{-1}(a)-a=A^\perp$.
If $B\cap(a+A^\perp)=\varnothing$, then $\pn{A}{B}(a)=\{0\}$.
Thus we assume that
$B\cap(a+A^\perp)\neq\varnothing$, which is equivalent to
$a\in A\cap(B+A^\perp)$.
Next, by Lemma~\ref{l:NsubsetN}\ref{l:NsubsetNi+},
$\pn{A}{B}(a) = A^\perp\cap \cone(B-a)$.
This implies
$(\forall\lambda\in\RPP)$
$\cone(B-\lambda a)=\cone(\lambda(B-a))=\cone(B-a)$.
Thus,
\begin{equation}
(\forall\lambda\in\RPP)\quad
\pn{A}{B}(\lambda a)= A^\perp \cap \cone(B-\lambda a)=
A^\perp\cap\cone(B-a)= \pn{A}{B}(a).
\end{equation}
This establishes not only the first two equalities in \eqref{e:0301a}
but also the third because
\begin{subequations}
\begin{align}
\bigcup_{a\in A}\pn{A}{B}(a)
&=\bigcup_{a\in A}\big(A^\perp\cap\cone(B-a)\big)
= A^\perp\cap \bigcup_{a\in A}\cone(B-a)\\
&= A^\perp\cap\cone\Big(\bigcup_{a\in A}(B-a)\Big)
=A^\perp\cap \cone(B-A) = A^\perp\cap(B-A)\\
&=A^\perp\cap (B+A).
\end{align}
\end{subequations}
The proof is complete.
\end{proof}


We now introduce two notions of angles between subspaces;
for further information, we highly recommend \cite{Deut94} and
\cite{Deutsch}.

\begin{definition}
Let $A$ and $B$ be linear subspaces of $X$.
\begin{enumerate}
\item\label{d:aDix}
{\rm \textbf{(Dixmier angle)} \cite{Dix49}}
The \emph{Dixmier angle} between $A$ and $B$ is the number
in $[0,\frac{\pi}{2}]$ whose cosine is given by
\begin{equation}
c_0(A,B):=\sup\menge{|\scal{a}{b}|}{a\in A,b\in B, \|a\|\leq 1,\|b\|\leq 1}.
\end{equation}
\item\label{d:aFri}
{\rm \textbf{(Friedrichs angle)} \cite{Frie37}}
The \emph{Friedrichs angle} (or simply the \emph{angle}) between $A$ and $B$ is the
number in $[0,\frac{\pi}{2}]$
whose cosine is given by
\begin{subequations}
\begin{align}
c(A,B)&:=c_0(A\cap(A\cap B)^\perp,B\cap(A\cap B)^\perp)\\
&= \sup\mmenge{|\scal{a}{b}|}
{\begin{aligned}
&a\in A\cap(A\cap B)^\bot,\|a\|\leq 1,\\
&b\in B\cap(A\cap B)^\bot,\|b\|\leq 1
\end{aligned}}.
\end{align}
\end{subequations}
\end{enumerate}
\end{definition}

Let us gather some properties of angles.

\begin{fact}
\label{f:c0&c}
Let $A$ and $B$ be linear subspaces of $X$.
Then the following hold:
\begin{enumerate}
\item
\label{f:c0&c-ii}
If $A\cap B=\{0\}$, then $c(A,B)=c_0(A,B)$.
\item
\label{f:c0&c-ii+}
If $A\cap B\neq \{0\}$, then $c_0(A,B)=1$.
\item
\label{f:c0&c-ii++}
$c(A,B)<1$.
\item\label{f:c0&c-ii3}
$c(A,B)=c_0(A,B\cap(A\cap B)^\bot)=c_0(A\cap(A\cap B)^\bot,B)$.
\item
\label{f:c0&c-iii}
{\rm \textbf{(Solmon)}}
$c(A,B)=c(A^\bot,B^\bot)$.
\end{enumerate}
\end{fact}
\begin{proof}
\ref{f:c0&c-ii}--\ref{f:c0&c-ii++}: Clear from the definitions.
\ref{f:c0&c-ii3}: See, e.g., \cite[Lemma~2.10(1)]{Deut94}
or \cite[Lemma~9.5]{Deutsch}.
\ref{f:c0&c-iii}: See, e.g., \cite[Theorem~2.16]{Deut94}.
\end{proof}


\begin{proposition}[CQ-number of two linear subspaces and Dixmier angle]
\label{p:CQn=c0}
Let $A$ and $B$ be linear subspaces of $X$, and let $\delta>0$.
Then
\begin{subequations}
\begin{align}
\theta_\dd(A,A,B,B)&= c_0\big(A^\bot\cap(A+B),B^\bot\cap(A+B)\big),\\
\theta_\dd(A,X,B,B)&= c_0\big(A^\bot\cap(A+B),B^\bot\big),\\
\theta_\dd(A,A,B,X)&= c_0\big(A^\bot,B^\bot\cap(A+B)\big),
\end{align}
\end{subequations}
where the CQ-numbers at 0 are defined as in \eqref{e:CQn}.
\end{proposition}
\begin{proof}
This follows from Proposition~\ref{p:nc-subsp}.
\end{proof}


We are now in a position to derive a striking connection
between the CQ-number and the Friedrichs angle,
which underlines a possible interpretation of the CQ-number as a generalized
Friedrichs angle between two sets.

\begin{theorem}[CQ-number of two linear subspaces and Friedrichs angle]
\label{t:CQn=c}
Let $A$ and $B$ be linear subspaces of $X$, and let $\dd>0$.
Then
\begin{equation}
\theta_\dd(A,A,B,B)=\theta_\dd(A,X,B,B)=\theta_\dd(A,A,B,X)=c(A,B)<1,
\end{equation}
where the CQ-number at $0$ is defined as in \eqref{e:CQn}.
\end{theorem}
\begin{proof}
On the one hand, using Fact~\ref{f:c0&c}\ref{f:c0&c-iii}, we have
\begin{subequations}
\begin{align}
c(A,B)&=c(A^\bot,B^\bot)\\
&=c_0\big(A^\bot\cap(A^\bot\cap B^\bot)^\bot,B^\bot\cap(A^\bot\cap B^\bot)^\bot\big)\\
&=c_0\big(A^\bot\cap(A+B),B^\bot\cap(A+B)\big).
\end{align}
\end{subequations}
On the other hand, Fact~\ref{f:c0&c}\ref{f:c0&c-ii3} yields
\begin{subequations}
\begin{align}
c_0\big(A^\bot\cap(A+B),B^\bot\big) &= c_0\big(A^\bot\cap(A^\bot\cap
B^\bot)^\bot,B^\bot\big)\\
&= c(A^\bot,B^\bot)\\
&= c_0\big(A^\bot,B^\bot\cap(A^\bot\cap B^\bot)^\bot\big)\\
&=c_0\big(A^\bot,B^\bot\cap(A+B)\big).
\end{align}
\end{subequations}
Altogether, recalling Proposition~\ref{p:CQn=c0}, we obtain the
result.
\end{proof}

The results in this subsection have a simple generalization to intersecting
affine subspaces.
Indeed, if $A$ and $B$ are \emph{intersecting} affine subspaces, then
the corresponding Friedrichs angle is
\begin{equation}
c(A,B):=c(\pa A,\pa B).
\end{equation}
Combining \eqref{e:120405b} with Theorem~\ref{t:CQn=c}, we immediately
obtain the following result.
\begin{corollary}[CQ-number of two intersecting affine subspaces and Friedrichs angle]
\label{c:CQn=c}
Let $A$ and $B$ be affine subspaces of $X$, suppose that
$c\in A\cap B$, and let $\dd>0$.
Then
\begin{equation}
\theta_\dd(A,A,B,B)=\theta_\dd(A,X,B,B)=\theta_\dd(A,A,B,X)=c(A,B)<1,
\end{equation}
where the CQ-number at $c$ is defined as in \eqref{e:CQn}.
\end{corollary}


\section{Regularities}

\label{s:superregularity}

In this section, we study a notion of set regularity that
is based on restricted normal cones.

\begin{definition}[regularity and superregularity]\label{d:reg}
Let $A$ and $B$ be nonempty subsets of $X$, and let $c\in X$.
\begin{enumerate}
\item
We say that $B$ is \emph{$(A,\ve,\dd)$-regular} at $c\in X$ if
$\ve\geq 0$, $\dd>0$, and
\begin{equation}
\label{e:dreg} \left.
\begin{array}{c}
(y,b)\in B\times B,\\
\|y-c\| \leq \dd,\|b-c\|\leq \dd,\\
u\in \pn{B}{A}(b)
\end{array}
\right\}\quad\Rightarrow\quad \scal{u}{y-b}\leq
\ve\|u\|\cdot\|y-b\|.
\end{equation}
If $B$ is $(X,\ve,\dd)$-regular at $c$, then
we also simply speak of $(\ve,\dd)$-regularity.
\item
The set $B$ is called $A$-\emph{superregular} at $c\in X$ if for
every $\ve>0$ there exists $\dd>0$ such that $B$ is
$(A,\ve,\dd)$-regular at $c$.
Again, if $B$ is $X$-superregular at $c$, then we also say
that $B$ is superregular at $c$.
\end{enumerate}
\end{definition}

\begin{remark}
\label{r:0303a}
Several comments on Definition~\ref{d:reg} are in order.
\begin{enumerate}
\item Superregularity with $A=X$ was introduced
by Lewis, Luke and Malick in \cite[Section~4]{LLM}.
Among other things, they point out that
amenability and prox regularity are sufficient conditions for
superregularity, while Clarke regularity is a necessary condition.
\item
The reference point $c$ does not have to belong to $B$.
If $c\not\in \overline{B}$, then for every $\dd\in\left]0,d_B(c)\right[$,
$B$ is $(0,\dd)$-regular at $c$; consequently, $B$ is superregular at $c$.
\item
If $\ve_1>\ve_2$ and $B$ is $(A,\ve_2,\dd)$-regular at $c$ then
$B$ is also $(A,\ve_1,\dd)$-regular at $c$.
\item
If $\varepsilon \in \left[1,\pinf\right[$, then Cauchy-Schwarz
implies that $B$ is $(\varepsilon,\pinf)$-regular at every point in $X$.
\item
\label{r:0303av}
It follows from Proposition~\ref{p:elementary}\ref{p:ele-ii} that
$B$ is $(A_1\cup A_2,\ve,\dd)$-regular at $c$ if and only if
$B$ is both $(A_1,\ve,\dd)$-regular and $(A_2,\ve,\dd)$-regular at $c$.
\item
\label{r:0303avi}
If $B$ is convex, then it follows with
Lemma~\ref{l:NsubsetN}\ref{l:NsubsetNvi} that
$B$ is $(A,0,\pinf)$-regular at $c$; consequently,
$B$ is superregular.
\item
\label{r:0303avi+}
Similarly, if $B$ is locally convex at $c$, i.e.,
there exists $\rho\in\RPP$ such that $B\cap\ball{c}{\rho}$ is convex,
then $B$ is superregular at $c$.
\item
\label{r:0303avii}
If $B$ is $(A,0,\dd)$-regular at $c$, then $B$ is $A$-superregular at $c$;
the converse, however, is not true in general
(see Example~\ref{ex:sphere1} below).
\end{enumerate}
\end{remark}

As a first example, let us consider the sphere.

\begin{example}[sphere]
\label{ex:sphere1}
Let $z\in X$ and $\rho\in\RPP$.
Set $S:= \sphere{z}{\rho}$, suppose that $s\in S$,
let $\ve\in\RPP$, and let $\dd\in\RPP$.
Then $S$ is $(\ve,\rho\ve)$-regular at $s$;
consequently, $S$ is superregular at $s$ (see Definition~\ref{d:reg}).
However, $S$ is not $(0,\dd)$-regular at $s$.
\end{example}
\begin{proof}
Let $b\in S$ and  $y\in S$.
Then
$\rho^2 = \|z-y\|^2=\|z-b\|^2+\|y-b\|^2-2\scal{z-b}{y-b}
=\rho^2+\|y-b\|^2-2\scal{z-b}{y-b}$,
which implies
\begin{equation}
\label{e:0308c}
2\scal{z-b}{y-b}=\|y-b\|^2.
\end{equation}
On the other hand, by Example~\ref{ex:ncS}, we have
\begin{equation}
\label{e:0308b}
\pn{S}{X}(b) \cap \sphere{0}{1} =
\bigg\{\pm \frac{z-b}{\|z-b\|}\bigg\} = \bigg\{\pm \frac{z-b}{\rho}\bigg\}.
\end{equation}
Suppose that $u\in\pn{S}{X}(b) \cap \sphere{0}{1}$.
Combining \eqref{e:0308c} and \eqref{e:0308b}, we obtain
\begin{equation}\label{e:0308c1}
\scal{\pn{S}{X}(b)\cap\sphere{0}{1}}{y-b} =
\bigg\{\pm\frac{1}{2\rho}\|y-b\|^2\bigg\}.
\end{equation}
Thus if $\|y-s\|\leq\rho\ve$, $\|b-s\|\leq\rho\ve$,
and $u\in\pn{S}{X}(b) \cap \sphere{0}{1}$, then
\begin{align}
\scal{u}{y-b}&\leq \frac{1}{2\rho}\|y-b\|^2
\leq \frac{1}{2\rho}\big(\|y-s\|+\|s-b\|\big)\|y-b\|
\leq \frac{\rho\ve+\rho\ve}{2\rho}\|y-b\|\\
&=\ve\|u\|\cdot \|y-b\|,
\end{align}
which verifies the $(\ve,\rho\ve)$-regularity of $S$ at $s$.
Finally,
by \eqref{e:0308c1},
\begin{equation}\max\big\{ \tscal{\pn{S}{X}(b)\cap\sphere{0}{1}}{y-b}  \big\} =
\frac{1}{2\rho}\|y-b\|^2 >0
\end{equation}
and therefore $S$ is not $(0,\dd)$-regular at $s$.
\end{proof}

We now characterizes $A$-superregularity using restricted normal cones.

\begin{theorem}[characterization of $A$-superregularity]
\label{t:Asreg}
Let $A$ and $B$ be nonempty subsets of $X$, and let $c\in X$.
Then $B$ is $A$-superregular at $c$ if and only if
for every $\varepsilon\in\RPP$, there exists $\dd\in\RPP$ such that
\begin{equation}
\label{e:Asreg} \left.
\begin{array}{c}
(y,b)\in B\times B\\
\|y-c\| \leq \dd,\|b-c\|\leq \dd\\
u\in \nc{B}{A}(b)
\end{array}
\right\}\quad\Rightarrow\quad \scal{u}{y-b}\leq \ve\|u\|\cdot\|y-b\|.
\end{equation}
\end{theorem}
\begin{proof}
``$\Leftarrow$'': Clear from Lemma~\ref{l:NsubsetN}\ref{l:NsubsetNiv}.
``$\Rightarrow$'':
We argue by contradiction; thus, we assume
there exists $\ve\in\RPP$
and sequences $(y_n,b_n,u_n)_\nnn$ in $B\times B\times X$
such that $(y_n,b_n)\to(c,c)$ and for every $\nnn$,
\begin{equation}
u_n\in\nc{B}{A}(b_n)
\quad\text{and}\quad
\scal{u_n}{y_n-b_n}>\ve\|u_n\|\cdot\|y_n-b_n\|.
\end{equation}
By the definition of the restricted normal cone,
for every $\nnn$, there exists a sequence $(b_{n,k},u_{n,k})_{k\in\NN}$
in $B\times X$ such that $\lim_{k\in\NN}b_{n,k}=b_n$,
$\lim_{k\in\NN}u_{n,k}=u_n$, and
$(\forall k\in\NN)$ $u_{n,k}\in \pn{B}{A}(b_{n,k})$.
Hence there exists a subsequence $(k_n)_\nnn$ of $(n)_\nnn$ such that
$b_{n,k_n}\to c$ and
\begin{equation}
(\forall\nnn)\quad
\scal{u_{n,k_n}}{y_n-b_{n,k_n}}>\frac{\ve}{2}\|u_{n,k_n}\|\cdot
\|y_n-b_{n,k_n}\|.
\end{equation}
However, this contradicts the $A$-superregularity of $B$ at $c$.
\end{proof}


When $B=X$, then Theorem~\ref{t:Asreg} turns into
\cite[Proposition~4.4]{LLM}:

\begin{corollary}[Lewis-Luke-Malick]
\label{c:LLMsreg}
Let $B$ be a nonempty subset of $X$ and let $c\in B$.
Then $B$ is superregular at $c$
if and only if for every $\ve\in\RPP$
there exists $\dd\in\RPP$ such
that
\begin{equation}
\label{e:sreg} \left.
\begin{array}{c}
(y,b)\in B\times B\\
\|y-c\| \leq \dd,\|b-c\|\leq \dd\\
u\in N_B(b)
\end{array}
\right\}\quad\Rightarrow\quad \scal{u}{y-b}\leq \ve\|u\|\cdot\|y-b\|.
\end{equation}
\end{corollary}


We now introduce  the notion of joint-regularity, which is tailored
for collections of sets and which turns into Definition~\ref{d:reg}
when the index set is a singleton.

\begin{definition}[joint-regularity]
\label{d:jreg}
Let $A$ be a nonempty subset of $X$, let
$\mcB := (B_j)_{j\in J}$ be a nontrivial collection
of nonempty subsets of $X$, and let $c\in X$.
\begin{enumerate}
\item We say that $\mcB$ is $(A,\ve,\dd)$-joint-regular at $c$
if $\ve\geq 0$, $\dd>0$, and for every $j\in J$,
$B_j$ is $(A,\ve,\dd)$-regular at $c$.
\item
The collection $\mcB$ is $A$-joint-superregular at $c$
if for every $j\in J$, $B_j$ is $A$-superregular at $c$.
\end{enumerate}
As in Definition~\ref{d:reg}, we may omit the prefix $A$ if $A=X$.
\end{definition}

Here are some verifiable conditions that guarantee
joint-(super)regularity.

\begin{proposition}
\label{p:jsreg}
Let $\mcA := (A_j)_{j\in J}$ and $\mcB := (B_j)_{j\in J}$
be nontrivial collections of nonempty subsets of $X$,
let $c\in X$, let $(\ve_j)_{j\in J}$ be a collection in $\RP$,
and let $(\dd_j)_{j\in J}$ be a collection in $\left]0,\pinf\right]$.
Set $A := \bigcap_{j\in J}A_j$,
$\ve := \sup_{j\in J}\ve_j$,
and $\dd := \inf_{j\in J}\dd_j$.
Then the following hold:
\begin{enumerate}
\item
\label{p:jsreg-i}
If $\dd>0$ and $(\forall j\in J)$ $B_j$ is $(A_j,\ve_j,\dd_j)$-regular
at $c$, then $\mcB$ is $(A,\ve,\dd)$-joint-regular at $c$.
\item
\label{p:jsreg-ii}
If $J$ is finite and
$(\forall j\in J)$ $B_j$ is $(A_j,\ve_j,\dd_j)$-regular at $c$,
then $\mcB$ is $(A,\ve,\dd)$-joint-regular at $c$.
\item
\label{p:jsreg-iii}
If $J$ is finite and
$(\forall j\in J)$ $B_j$ is $A_j$-superregular at $c$,
then $\mcB$ is $A$-joint-superregular at $c$.
\end{enumerate}
\end{proposition}
\begin{proof}
\ref{p:jsreg-i}:
Indeed, by Remark~\ref{r:0303a}\ref{r:0303av},
$B_j$ is $(A,\ve,\dd)$-regular at $c$ for every $j\in J$.

\ref{p:jsreg-ii}:
Since $J$ is finite,
we have $\dd>0$ and so the conclusion follows from \ref{p:jsreg-i}.

\ref{p:jsreg-iii}:
This follows from \ref{p:jsreg-ii} and the definitions.
\end{proof}


\begin{corollary}[convexity and regularity]
\label{c:jsreg}
Let $\mcB := (B_j)_{j\in J}$ be a nontrivial collection of nonempty
convex subsets of $X$, let ${A}\subseteq X$, and let $c\in X$.
Then $\mcB$ is $(0,\pinf)$-joint-regular,
$(A,0,\pinf)$-joint-regular,
joint-superregular,
and $A$-joint-superregular at $c$.
\end{corollary}
\begin{proof}
By Remark~\ref{r:0303a}\ref{r:0303avi},
$B_j$ is $(0,\pinf)$-regular, superregular, and $A$-superregular at $c$,
for every $j\in J$.
Now apply Proposition~\ref{p:jsreg}\ref{p:jsreg-i}\&\ref{p:jsreg-iii}.
\end{proof}


The following example illustrates the flexibility gained through the
notion of joint-regularity.



\begin{example}[two lines: joint-superregularity $\not\Rightarrow$
superregularity of the union]
\label{ex:badlines}
Suppose that
$d_1$ and $d_2$ are in $\sphere{0}{1}$.
Set $B_1 := \RR d_1$, $B_2 := \RR d_2$, and $B := B_1\cup B_2$,
and assume that $B_1\cap B_2=\{0\}$.
By Corollary~\ref{c:jsreg},
$(B_1,B_2)$ is joint-superregular at $0$.
Let $\delta\in\RPP$,
and set $b := \delta d_1$ and $y := \delta d_2$.
Then $\|y-0\|=\delta$, $\|b-0\|=\delta$, and $0<\|y-b\| = \delta\|d_2-d_1\|$.
Using Proposition~\ref{p:Nequi}\ref{p:Ne-iii}, we see
that $N_{B}(b)=\{d_1\}^\perp$.
Note that there exists $v\in \{d_1\}^\perp$ such that
$\scal{v}{d_2}\neq 0$
(for otherwise $\{d_1\}^\perp \subseteq \{d_2\}^\perp$
$\Rightarrow$ $B_2\subseteq B_1$, which is absurd).
Hence there exists $u\in \{d_1\}^\perp= \{b\}^\perp = N_B(b)$ such that
$\|u\|=1$ and $\scal{u}{d_2}>0$.
It follows that $\scal{u}{y-b} = \scal{u}{y}
=\delta\scal{u}{d_2}=\scal{u}{d_2}\|u\|\|y-b\|/\|d_2-d_1\|$.
Therefore, $B$ is not superregular at $0$.
\end{example}


Let us provide an example of an $A$-superregular set that is not
superregular. To do so, we require the following elementary result.

\begin{lemma}
\label{l:0303a}
Consider in $\RR^2$ the sets
$C := [(0,1),(m,1+m^2)] = \menge{(x,1+mx)}{x\in[0,m]}$
and $D := [(m,1),(m,1+m^2)]$, where $m\in\RPP$.
Let $z\in\RR$.
Then
\begin{equation}
\label{e:0303a}
P_{C\cup D}(z,0) =
\begin{cases}
(0,1), &\text{if $z<m/2$;}\\
\{(0,1),(m,1)\}, &\text{if $z=m/2$;}\\
(m,1), &\text{if $z>m/2$.}
\end{cases}
\end{equation}
\end{lemma}
\begin{proof}
It is clear that $P_D(z,0)=(m,1)$.
We assume that $0<z<m$ for otherwise \eqref{e:0303a} is clearly true.
We claim that $P_C(z,0)=(0,1)$.
Indeed,
$f\colon x\mapsto \|(x,1+mx)-(z,0)\|^2$ is a convex quadratic with
minimizer $x_z := (z-m)/(1+m^2)$. The requirement $x_z\geq 0$ from the
definition of $C$ forces $z\geq m$, which is a contradiction.
Hence $P_C(z,0)$ is a subset of the relative boundary of $C$, i.e.,
of $\{(0,1),(m,1+m^2)\}$. Clearly, $(0,1)$ is the closer to $(z,0)$
than $(m,1+m^2)$. This verifies the claim.
Since $P_{C\cup D}(z,0)$ is the subset of points in $P_C(z,0)\cup
P_D(z,0)$ closest to $(z,0)$, the result follows.
\end{proof}

\begin{example}[$A$-superregularity $\not\Rightarrow$ superregularity]
\label{ex:saw}
Suppose that $X=\RR^2$.
As in \cite[Example~4.6]{LLM}, we consider $c:=(0,0)\in X$ and
$B :=\epi f$, where
\begin{equation}
f\colon\RR\to\RX\colon x\mapsto \begin{cases}
2^k(x-2^k),&\text{if $2^k\leq x < 2^{k+1}$ and $k\in \ZZ$;}\\
0,&\text{if $x=0$;}\\
\pinf,&\text{if $x<0$.}
\end{cases}
\end{equation}
Then $B$ is not superregular at $c$;
however, $B$ is $A$-superregular at $c$, where
$A := \RR\times\{-1\}$.
\end{example}
\begin{proof}
It is stated in \cite[Example~4.6]{LLM}
that $B$ is not superregular at $c$ (and that $B$ is
Clarke regular at $c$).

To tackle $A$-superregularity, let us determine $P_B(A)$.
Let us consider the point $a = (\alpha,-1)$,
where $\alpha\in\left[2^{-1},1\right[$.
Then
Lemma~\ref{l:0303a}
(see also the picture below) implies that
\begin{equation}
P_B(\alpha,-1)=\begin{cases}
(\thalb,0), &\text{if $\thalb\leq\alpha<\tfrac{3}{4}$;}\\
\big\{(\thalb,0),(1,0)\big\}, &\text{if $\alpha=\tfrac{3}{4}$;}\\
(1,0), &\text{if $\tfrac{3}{4}<\alpha<1$;}
\end{cases}
\end{equation}
\setlength{\unitlength}{0.03in}
\begin{picture}(150,145)(-25,-120)

\put(10,10){$B=\epi f$}
\put(100,0){\line(1,1){25}} \put(100,-6){$1$}
\put(50,0){\line(2,1){50}}  \put(45,-6){$\frac{1}{2}$}
\put(25,0){\line(4,1){25}}  \put(22,-6){$\frac{1}{4}$}
\put(12.5,0){\line(6,1){12.5}} \put(10,-6){$\frac{1}{8}$}
\put(6.25,0){\line(6,1){6.25}} \put(3,-6){$\frac{1}{16}$}
\put(3.125,0){\line(6,1){3.125}}
\put(-4,-6){$0$}

\put(0,0){\line(1,0){4}}

\put(0,0){\vector(0,1){25}}
\put(120,0){\vector(1,0){10}}

\put(100,0){\line(0,1){25}}
\put(50,0){\line(0,1){6.25}}
\put(25,0){\line(0,1){2}}
\put(12.5,0){\line(0,1){1}}
\put(6.25,0){\line(0,1){0.5}}

\multiput(-25,0)(5,0){30}{\line(1,0){2}}
\multiput(0,-2)(0,-5){22}{\line(0,-1){2}}

\multiput(75,-100)(3,12){9}{\line(1,4){1.2}}
\multiput(75,-100)(-3,12){9}{\line(-1,4){1.2}} \put(75,-106){$\frac{3}{4}$}

\multiput(37.5,-100)(3,24){5}{\line(1,6){1}}
\multiput(37.5,-100)(-3,24){5}{\line(-1,6){1}} \put(37.5,-106){$\frac{3}{8}$}

\multiput(25,-100)(0,12){9}{\line(0,4){4}}
\put(25,-100){\circle*{2}}
\put(25,-106){$\frac{1}{4}$}

\multiput(50,-100)(0,12){9}{\line(0,4){4}}
\put(50,-100){\circle*{2}}
\put(50,-106){$\frac{1}{2}$}

\multiput(100,-100)(0,12){9}{\line(0,4){4}}
\put(100,-100){\circle*{2}}
\put(100,-106){$1$}

\linethickness{0.5mm}
\put(-25,-100){\line(1,0){150}}
\put(-23,-106){$A=\RR\times\{-1\}$}
\put(-10,-97){$-1$}

\end{picture}

and more generally,
\begin{equation}
2^k\leq \alpha < 2^{k+1}
\;\Rightarrow\;
P_B(\alpha,-1)=\begin{cases}
(2^k,0), &\text{if $2^k\leq\alpha<2^k+2^{k-1}$;}\\
\big\{(2^k,0),(2^{k+1},0)\big\}, &\text{if $\alpha=2^k+2^{k-1}$;}\\
(2^{k+1},0), &\text{if $2^k+2^{k-1}<\alpha<2^{k+1}$.}
\end{cases}
\end{equation}
Clearly, if $a\in\RM\times\{-1\}$, then $P_B(a)=(0,0)$.
Let $b\in B$.
Then
\begin{equation}
A \cap P^{-1}_B(b)=\begin{cases}
\big[2^{k-2}+2^{k-1},2^{k-1}+2^k\big]\times\{-1\},
&\text{if $b=(2^k,0)$ and $k\in\ZZ$;}\\
\RM\times\{-1\}, &\text{if $b=(0,0)$;}\\
\varnothing,&\text{otherwise.}
\end{cases}
\end{equation}
Thus
\begin{equation}
\pn{B}{A}(b)=\begin{cases}
\cone\Big(\big[-2^{k-2},2^{k-1}\big]\times\{-1\}\Big),
&\text{if $b=(2^k,0)$ and $k\in\ZZ$;}\\
\{(0,0)\}\cup \big(\RM\times\RMM\big), &\text{if $b=(0,0)$;}\\
\{(0,0)\},&\text{otherwise.}
\end{cases}
\end{equation}
Let $\ve\in\RPP$.
Let $K\in\ZZ$ be such that $2^{K-1}\leq\ve$,
and let $\delta\in \left]0,2^K\right]$.
Furthermore, let $y=(y_1,y_2)\in B$,
let $b=(b_1,b_2)\in B$,
let $u\in\pn{B}{A}(b)$,
and assume that
$\|y-c\|\leq\dd$ and that $\|b-c\|\leq\dd$.
We consider three cases.

\emph{Case~1:} $b=(0,0)$. Then $u\in\RM^2$ and $y\in\RP^2$;
consequently, $\scal{u}{y-b} = \scal{u}{y}\leq 0
\leq\ve\|u\|\cdot\|y-b\|$.

\emph{Case~2:} $b\notin (\{0\}\cup 2^\ZZ)\times\{0\}$.
Then $\pn{B}{A}(b)=\{(0,0\}$; hence $u=0$ and so
$\scal{u}{y-b}=0\leq \ve\|u\|\cdot\|y-b\|$.

\emph{Case~3:} $b\in 2^\ZZ\times\{0\}$, say $b=(2^k,0)$, where $k\in \ZZ$.
Since $2^k=\|b-0\|=\|b-c\|\leq\delta\leq 2^K$,
we have $k\leq K$. Furthermore, $y_2\geq 0$,
$\max\{|y_1-b_1|,|y_2-b_2|\}\leq\|y-b\|$,
and $u=\lambda(t,-1)=(\lambda t,-\lambda)$
where $t\in[-2^{k-2},2^{k-1}]$ and $\lambda\geq0$.
Hence $\lambda\leq\|u\|$ and
\begin{subequations}
\begin{align}
\scal{u}{y-b}&=\lambda t(y_1-b_1)-\lambda (y_2-b_2)
=\lambda t(y_1-b_1)-\lambda (y_2-0) \\
&\leq \lambda t(y_1-b_1) \leq \lambda|t|\cdot|y_1-b|\\
&\leq \|u\|\cdot 2^{k-1}\cdot\|y-b\| \leq 2^{K-1}\|u\|\cdot\|y-b\| \leq
\ve\cdot\|u\|\cdot\|y-b\|.
\end{align}
\end{subequations}

Therefore, in all three cases, we have shown that
$\scal{u}{y-b} \leq\ve\|u\|\cdot\|y-b\|$.
\end{proof}

We now use Example~\ref{ex:saw}
to construct an example complementary to
Example~\ref{ex:badlines}.

\begin{example}[superregularity of the union $\not\Rightarrow$
joint-superregularity]
Suppose that $X=\RR^2$,
set
$B_1 := \epi f$, where $f$ is as in Example~\ref{ex:saw},
$B_2 := X\smallsetminus B_1$,
and $c := (0,0)$.
Since $B_1\cup B_2= X$ is convex, it is clear
from Remark~\ref{r:0303a}\ref{r:0303avi} that
$B_1\cup B_2$ is superregular at $c$.
On the other hand, since $B_1$ is not superregular at $c$
(see Example~\ref{ex:saw}), it is obvious that
$(B_1,B_2)$ is not joint-superregular at $c$.
\end{example}



\section{The method of alternating projections (MAP)}

\label{s:application}

We now apply the machinery of restricted normal cones and associated
results to derive linear convergence results.

\subsection*{On the composition of two projection operators}

The method of alternating projections iterates
projection operators. Thus, in the next few results,
we focus on the outcome of a single iteration of the composition.

\begin{lemma}
\label{l:0305a}
Let $A$ and $B$ be nonempty closed subsets of $X$.
Then the following hold\footnote{We denote by
$\bd_{\aff A\cup B}(S)$ the boundary of $S\subseteq X$
with respect to $\aff(A\cup B)$.}:
\begin{enumerate}
\item
\label{l:0305ai}
$P_A(B\smallsetminus A)\subseteq \bd_{\aff A\cup B}A \subseteq \bd A$.
\item
\label{l:0305aii}
$P_B(A\smallsetminus B)\subseteq \bd_{\aff A\cup B}(B)\subseteq \bd B$.
\item
\label{l:0305aiii}
If $b\in B$ and $a\in P_A b$, then:
\begin{equation}
a\in (\bd A)\smallsetminus B
\;\Leftrightarrow\;
a\in A\smallsetminus B
\;\Rightarrow\;
b\in B\smallsetminus A
\;\Rightarrow\; a\in\bd A.
\end{equation}
\item
\label{l:0305aiv}
If $a\in A$ and $b\in P_B a$, then:
\begin{equation}
b\in (\bd B)\smallsetminus A
\;\Leftrightarrow\;
b\in B\smallsetminus A
\;\Rightarrow\;
a\in A\smallsetminus B
\;\Rightarrow\;
b\in \bd B.
\end{equation}
\end{enumerate}
\end{lemma}
\begin{proof}
\ref{l:0305ai}:
Take $b\in B\smallsetminus A$ and $a\in P_Ab$.
Assume to the contrary that there exists $\dd\in\RPP$
such that $\aff(A\cup B)\cap \ball{a}{\dd}\subseteq A$.
Hence $\wt{a} := a+\dd(b-a)/\|b-a\|\in A$ and
thus $d_A(b)\leq d(\wt{a},b)<d(a,b)=d_A(b)$, which is absurd.

\ref{l:0305aii}:
Interchange the roles of $A$ and $B$ in \ref{l:0305ai}.

\ref{l:0305aiii}:
If $a\in (\bd A)\smallsetminus B$, then clearly $a\in A\smallsetminus B$.
Now assume that $a\in A\smallsetminus B$.
If $b\in A$, then $a\in P_Ab=\{b\}\subseteq B$, which is absurd.
Hence $b\in B\smallsetminus A$ and thus \ref{l:0305ai}
implies that $a\in P_A(B\smallsetminus A)\subseteq\bd A$.

\ref{l:0305aiv}:
Interchange the roles of $A$ and $B$ in \ref{l:0305aiii}.
\end{proof}


\begin{lemma}
\label{l:Ed}
Let $A$ and $B$ be nonempty closed subsets of $X$,
let $c\in X$, let $y\in B$, let $a\in P_Ay$, let $b\in P_Ba$, and
let $\dd\in\RP$.
Assume that $d_A(y)\leq\dd$ and that $d(y,c)\leq\dd$.
Then the following hold:
\begin{enumerate}
\item
\label{l:Edi}
$d(a,c)\leq 2\dd$.
\item
\label{l:Edii}
 $d(b,y)\leq 2d(a,y)\leq 2\dd$.
\item
\label{l:Ediii}
$d(b,c)\leq3\dd$.
\end{enumerate}
\end{lemma}
\begin{proof}
Since $y\in B$, we have
\begin{equation}
\label{e:1026a} d(a,b)=d_B(a)\leq d(a,y)=d_A(y)\leq\dd.
\end{equation}
Thus,
\begin{equation}
\label{e:1026b}
d(a,c)\leq d(a,y)+d(y,c) \leq \dd + \dd = 2\dd,
\end{equation}
which establishes \ref{l:Edi}.
Using \eqref{e:1026a}, we also conclude that
$d(b,y)\leq d(b,a) + d(a,y) \leq 2d(a,y)\leq 2\dd$;
hence, \ref{l:Edii} holds.
Finally, combining \eqref{e:1026a} and \eqref{e:1026b},
we obtain \ref{l:Ediii} via
$d(b,c) \leq d(b,a) + d(a,c) \leq \dd + 2\dd = 3\dd$.
\end{proof}

\begin{corollary}\label{c:0330a}
Let $A$ and $B$ be nonempty closed subsets of $X$,
let $\rho\in\RPP$, and suppose that $c\in A\cap B$.
Then
\begin{equation}
  P_AP_BP_A\ball{c}{\rho}\subseteq\ball{c}{6\rho}.
\end{equation}
\end{corollary}
\begin{proof}
Let $b_{-1}\in\ball{c}{\rho}$, $a_0\in P_A b_{-1}$,
$b_0\in P_B a_0$, and $a_1\in P_A b_0$.
We have $d(a_0,b_{-1})=d_A(b_{-1})\leq d(b_{-1},c)\leq\rho$,
so $d_B(a_0)\leq d(a_0,c)\leq d(a_0,b_{-1})+d(b_{-1},c)\leq 2\rho$.
Applying Lemma~\ref{l:Ed}\ref{l:Ediii} to the
sets $B$ and $A$, the points $a_0,b_0,a_1$, and $\delta = 2\rho$,
we deduce that $d(a_1,c)\leq 3(2\rho)=6\rho$.
\end{proof}

The next two results are essential
to guarantee a local
contractive property of the composition.

\begin{proposition}[regularity and contractivity]
\label{p:effect1}
Let $A$ and $B$ be nonempty closed subsets of $X$,
let $\wt{A}$ and $\wt{B}$ be nonempty subsets of $X$,
let $c\in X$,
let $\ve\geq 0$, and let $\dd>0$.
Assume that $B$ is $(\wt{A},\ve,3\dd)$-regular at $c$
(see Definition~\ref{d:reg}).
Furthermore,
assume that
$y\in B\cap\wt{B}$,
that $a\in P_A(y)\cap\wt{A}$,
that $b\in P_B(a)$, that
$\|y-c\|\leq\dd$, and
that $d_A(y)\leq\dd$.
Then
\begin{equation}
\|a-b\|\leq(\theta_{3\dd}+2\ve)\|a-y\|,
\end{equation}
where $\theta_{3\dd}$ the CQ-number at $c$ associated with
$(A,\wt{A},B,\wt{B})$ (see \eqref{e:CQn}).
\end{proposition}
\begin{proof}
Lemma~\ref{l:Ed}\ref{l:Edi}\&\ref{l:Ediii} yields
$\|a-c\|\leq 2\dd$ and $\|b-c\|\leq 3\dd$.
On the other hand, $y-a\in\pn{A}{\wt{B}}(a)$ and
$b-a\in-\pn{B}{\wt{A}}(b)$.
Therefore,
\begin{equation}
\label{e:1026d}
\scal{b-a}{y-a}\leq\theta_{3\dd}\|b-a\|\cdot\|y-a\|.
\end{equation}
Since $a-b\in\pn{B}{\wt{A}}(b)$,
$\|y-c\|\leq \dd$, and $\|b-c\|\leq3\dd$,
we obtain, using the $(\wt{A},\ve,3\dd)$-regularity of $B$,
that $\scal{a-b}{y-b}\leq\ve\|a-b\|\cdot\|y-b\|$.
Moreover, Lemma~\ref{l:Ed}\ref{l:Edii} states that
$\|y-b\|\leq 2\|a-y\|$.
It follows that
\begin{equation}
\label{e:1026e} \scal{a-b}{y-b}\leq 2\ve\|a-b\|\cdot\|a-y\|.
\end{equation}
Adding \eqref{e:1026d} and \eqref{e:1026e} yields
$\|a-b\|^2\leq(\theta_{3\dd}+2\ve)\|a-b\|\cdot\|a-y\|$.
The result follows.
\end{proof}

We now provide a result for collections of sets similar to---and
relying upon---Proposition~\ref{p:effect1}.



\begin{proposition}[joint-regularity and contractivity]
\label{p:effect}
Let $\mcA:=(A_i)_{i\in I}$ and $\mcB:=(B_j)_{j\in J}$
be nontrivial collections of closed subsets of $X$,
Assume that $A:=\bigcup_{i\in I}A_i$ and $B:=\bigcup_{j\in J}B_j$
are closed, and that $c\in A\cap B$.
Let $\wt{\mcA}:=(\wt{A}_i)_{i\in I}$ and
$\wt{\mcB}:=(\wt{B}_j)_{j\in J}$ be nontrivial collections
of nonempty subsets of $X$ such that
$(\forall i\in I)$ $P_{A_i}((\bd B)\smallsetminus A)\subseteq\wt{A}_i$ and
$(\forall j\in J)$ $P_{B_j}((\bd A)\smallsetminus B)\subseteq\wt{B}_j$.
Set $\wt{A}:=\bigcup_{i\in I}\wt{A}_i$ and
$\wt{B}:=\bigcup_{j\in J}\wt{B}_j$,
let $\ve\geq 0$ and let $\dd> 0$.
\begin{enumerate}
\item\label{p:effect-i}
If $b\in (\bd B)\smallsetminus A$ and $a\in P_A(b)$,
then $(\exi i\in I)$ $a\in P_{A_i}(b)\subseteq A_i\cap\wt{A}_i$.
\item\label{p:effect-i1}
If $a\in (\bd A)\smallsetminus B$ and $b\in P_B(a)$,
then $(\exi j\in J)$ $b\in P_{B_j}(a)\subseteq B_j\cap\wt{B}_j$.
\item\label{p:0329a-iii}
If $y\in B$, $a\in P_A(y)$ and $b\in P_B(a)$, then:
    \begin{equation}
      b\in\big((\bd B)\smallsetminus A\big)\cap\bigcup_{j\in J}(B_j\cap\wt{B}_j)\ \Leftrightarrow\ b\in B\smallsetminus A\ \Rightarrow\
      a\in A\smallsetminus B.
    \end{equation}
\item\label{p:0329a-iv}
If $x\in A$, $b\in P_B(x)$, and $a\in P_A(b)$, then:
    \begin{equation}
      a\in\big((\bd A)\smallsetminus B\big)\cap\bigcup_{i\in I}(A_i\cap\wt{A}_i)\ \Leftrightarrow\ a\in A\smallsetminus B
      \ \Rightarrow\ b\in B\smallsetminus A.
    \end{equation}
\item \label{p:effect-ii}
Suppose that $\mcB$ is $(\wt{A},\ve,3\dd)$-joint-regular at $c$
(see Definition~\ref{d:jreg}),
that $y\in ((\bd B)\smallsetminus A)\cap
\bigcup_{j\in J}(B_j\cap\wt{B}_j)$,
that $a\in P_A(y)$, that $b\in P_B(a)$, and that $\|y-c\|\leq\dd$.
Then
\begin{equation}
\|b-a\|\leq(\theta_{3\dd}+2\ve)\|a-y\|,
\end{equation}
where $\theta_{3\dd}$ is the joint-CQ-number at $c$ associated with
$(\mcA,\wt{\mcA},\mcB,\wt{\mcB})$
(see \eqref{e:jCQn}).
\item \label{p:effect-iv}
Suppose that $\mcA$ is $(\wt{B},\ve,3\dd)$-joint-regular at $c$
(see Definition~\ref{d:jreg}),
that $x\in ((\bd A)\smallsetminus B)\cap
\bigcup_{i\in I}(A_i\cap\wt{A}_i)$,
that $b\in P_B(x)$, that $a\in P_A(b)$, and that $\|x-c\|\leq\dd$.
Then
\begin{equation}
\|a-b\|\leq(\theta_{3\dd}+2\ve)\|b-x\|,
\end{equation}
where $\theta_{3\dd}$ is the joint-CQ-number at $c$ associated with
$(\mcA,\wt{\mcA},\mcB,\wt{\mcB})$
(see \eqref{e:jCQn}).
\end{enumerate}
\end{proposition}
\begin{proof}
\ref{p:effect-i}\&\ref{p:effect-i1}:
Clear from Lemma~\ref{l:unionproj} and the assumptions.


\ref{p:0329a-iii}:
Note that Lemma~\ref{l:0305a}\ref{l:0305aiv}\&\ref{l:0305aiii}
and \ref{p:effect-i1}
yield the implications
\begin{equation}
b\in B\smallsetminus A
\;\Leftrightarrow\;
b\in(\bd B)\smallsetminus A
\;\Rightarrow\;
a\in A\smallsetminus B
\;\Leftrightarrow\;
a\in (\bd A)\smallsetminus B
\;\Rightarrow\;
b\in\bigcup_{j\in J}(B_j\cap\wt{B}_j),
\end{equation}
which give the conclusion.

\ref{p:0329a-iv}:
Interchange the roles of $A$ and $B$ in \ref{p:0329a-iii}.

\ref{p:effect-ii}:
There exists $j\in J$ such that
$y\in B_j\cap\wt{B}_j\cap ((\bd B)\smallsetminus A)$.
Let $b'\in P_{B_j}a$.
Then
\begin{equation}
\label{e:0304a}
\|a-b\|=d_B(a)\leq d_{B_j}(a)=\|a-b'\|.
\end{equation}
Since $\mcB$ is $(\wt{A},\ve,3\dd)$-joint-regular at $c$,
it is clear that $B_j$ is  $(\wt{A},\ve,3\dd)$-regular at $c$.
Since $y\in (\bd B)\smallsetminus A$  and
because of \ref{p:effect-i}, there exists $i\in I$ such that
$a\in P_{A_i}y\subseteq \wt{A}_i$.
Since $\wt{A}_i\subseteq \wt{A}$, it follows that
(see also Remark~\ref{r:0303a}\ref{r:0303av})
$B_j$ is $(\wt{A}_i,\varepsilon,3\dd)$-regular at $c$.
Since $y\in B_j\cap\wt{B}_j$,
$a\in P_{A_i}y\cap\wt{A}_i$,
$b'\in P_{B_j}a$,
and $d_{A_i}(y)=d_A(y)=\|y-a\|\leq\|y-c\|\leq\dd$, we obtain
from Proposition~\ref{p:effect1} that
\begin{equation}
\|a-b'\|\leq \big(\theta_{3\dd}(A_i,\wt{A}_i,B_j,\wt{B}_j)
+2\ve\big)\|a-y\|.
\end{equation}
Combining with \eqref{e:0304a}, we deduce that
$\|a-b\|\leq\|a-b'\|\leq(\theta_{3\dd}+2\ve)\|a-y\|$.

\ref{p:effect-iv}: This follows from \ref{p:effect-ii} and
\eqref{e:120405a}.
\end{proof}

\subsection*{An abstract linear convergence result}

Let us now focus on algorithmic results
(which are actually true even in complete metric spaces).

\begin{definition}[linear convergence]
Let $(x_n)_\nnn$ be a sequence in $X$, let $\bar{x}\in X$,
and let $\gamma\in\left[0,1\right[$.
Then $(x_n)_\nnn$ \emph{converges linearly} to $\bar{x}$
with \emph{rate} $\gamma$ if there exists $\mu\in\RP$ such that
\begin{equation}
(\forall\nnn)
\quad
d(x_n,\bar{x})\leq \mu\gamma^n.
\end{equation}
\end{definition}

\begin{remark}[rate of convergence depends only on the tail of the
sequence]
\label{r:0307a}
Let $(x_n)_\nnn$ be a sequence in $X$,
let $\bar{x}\in X$, and let $\gamma\in\zeroun$.
Assume that there exists $n_0\in\NN$ and $\mu_0\in\RP$
such that
\begin{equation}
\big(\forall n\in\{n_0,n_0+1,\ldots\}\big)
\quad
d(x_n,\bar{x})\leq \mu_0\gamma^n.
\end{equation}
Set $\mu_1 :=
\max\menge{d(x_m,\bar{x})/\gamma^m}{m\in\{0,1,\ldots,n_0-1\}}$.
Then
\begin{equation}
(\forall\nnn)
\quad
d(x_n,\bar{x})\leq \max\{\mu_0,\mu_1\}\gamma^n,
\end{equation}
and therefore $(x_n)_\nnn$ converges linearly to $\bar{x}$ with rate
$\gamma$.
\end{remark}

\begin{proposition}[abstract linear convergence]
\label{p:geo}
Let $A$ and $B$ be nonempty closed subsets of $X$,
let $(a_n)_\nnn$ be a sequence in $A$,
and let $(b_n)_\nnn$ be a sequence in $B$.
Assume that there exist constants $\alpha\in\RP$ and
$\beta\in\RP$ such that
\begin{subequations}
\begin{equation}
\gamma := \alpha\beta < 1
\end{equation}
and
\begin{equation}
(\forall\nnn)
\quad
d(a_{n+1},b_n)\leq\alpha d(a_n,b_n)
\;\text{and}\;
d(a_{n+1},b_{n+1})\leq\beta d(a_{n+1},b_{n}).
\end{equation}
\end{subequations}
Then $(\forall\nnn)$ $d(a_{n+1},b_{n+1})\leq\gamma d(a_n,b_n)$
and there exists $c\in A\cap B$ such that
\begin{equation}
\label{e:0305c}
(\forall\nnn)\quad
\max\big\{d(a_n,c),d(b_n,c)\big\} \leq
\frac{1+\alpha}{1-\gamma}d(a_0,b_0)\cdot \gamma^n;
\end{equation}
consequently, $(a_n)_\nnn$ and $(b_n)_\nnn$ converge linearly to $c$ with
rate $\gamma$.
\end{proposition}
\begin{proof}
Set $\delta := d(a_0,b_0)$. Then
for every $\nnn$,
\begin{equation}
\label{e:1027d'}
d(a_n,b_n) \leq \beta d(a_n,b_{n-1})
\leq\alpha\beta d(a_{n-1},b_{n-1})
=\gamma d(a_{n-1},b_{n-1})
\leq \cdots \leq \gamma^n\delta;
\end{equation}
hence,
\begin{subequations}
\label{e:0305a}
\begin{align}
\label{e:1027c'}
d(b_n,b_{n+1})&\leq d(b_n,a_{n+1})+d(a_{n+1},b_{n+1})
\leq \alpha d(b_n,a_n)+\gamma d(a_n,b_n)\\
&= (\alpha+\gamma)d(a_n,b_n) \leq (\alpha+\gamma)\delta\gamma^n.
\end{align}
\end{subequations}
Thus $(b_n)_\nnn$ is a Cauchy sequence, so there exists
$c\in B$ such that $b_n\to c$.
On the other hand, by \eqref{e:1027d'}, $d(a_n,b_n)\to 0$ and
$(a_n)_\nnn$ lies in $A$.
Hence, $a_n\to c$ and $c\in A$.
Thus, $c\in A\cap B$.
Fix $n\in\NN$ and let $m\geq n$.
Using \eqref{e:0305a},
\begin{equation}
\label{e:1027e'}
d(b_n,b_m)\leq \sum_{k=n}^{m-1} d(b_k,b_{k+1})
\leq \sum_{k\geq n} d(b_k,b_{k+1})
\leq \sum_{k\geq n} (\alpha+\gamma)\delta\gamma^k
= \frac{(\alpha+\gamma)\delta\gamma^n}{1-\gamma}.
\end{equation}
Hence, using \eqref{e:1027d'} and \eqref{e:1027e'}, we estimate that
\begin{equation}
\label{e:0305b}
d(a_n,b_m) \leq d(a_n,b_n)+d(b_n,b_m)
\leq \delta\gamma^n + \frac{(\alpha+\gamma)\delta\gamma^n}{1-\gamma}
=\frac{(1+\alpha)\delta\gamma^n}{1-\gamma}.
\end{equation}
Letting $m\to\pinf$ in \eqref{e:1027e'} and \eqref{e:0305b},
we obtain \eqref{e:0305c}.
\end{proof}

\subsection*{The sequence generated by the MAP}

We start with the following definition,
which is well defined by Proposition~\ref{p:0224a}.

\begin{definition}[MAP]
Let $A$ and $B$ be nonempty closed subsets of $X$,
let $b_{-1}\in X$,
and let
\begin{equation}
(\forall\nnn)\quad a_{n}\in P_A(b_{n-1})
\;\text{and}\;b_n\in P_B(a_n).
\end{equation}
Then we say that the sequences $(a_n)_\nnn$ and $(b_n)_\nnn$
are \emph{generated by the method of alternating projections}
(with respect to the pair $(A,B)$) with starting point $b_{-1}$.
\end{definition}

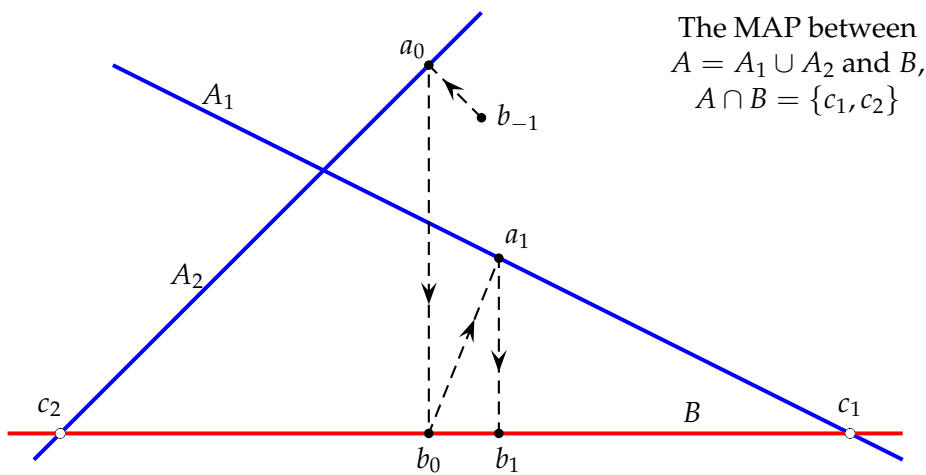
\begin{center}
\psset{xunit=0.7cm, yunit=0.7cm}
\begin{pspicture}
(-6,-6)(11,3)

\psline[linewidth=1.5pt,linecolor=blue]{-}(-5.5,-5.5)(3,3)
\psline[linewidth=1.5pt,linecolor=blue]{-}(-4,2)(11,-5.5)
\rput(-2.6,-2){$A_2$}
\rput(-2,1.4){$A_1$}

\psline[linewidth=1.5pt,linecolor=red]{-}(-6,-5)(11,-5)
\psdot[dotstyle=o,dotsize=4pt](10,-5) \rput(10,-4.5){$c_1$}
\psdot[dotstyle=o,dotsize=4pt](-5,-5) \rput(-5.2,-4.5){$c_2$}
\rput(7,-4.6){$B$}

\psline[linestyle=dashed,arrowsize=6pt,ArrowInside=->,ArrowInsidePos=0.65,showpoints=true]
(3,1)(2,2)(2,-5)(3.33,-1.67)(3.33,-5)

\rput(3.7,1){$b_{-1}$}
\rput(1.7,2.3){$a_0$}
\rput(2,-5.5){$b_0$}
\rput(3.7,-1.3){$a_1$}
\rput(3.5,-5.5){$b_1$}

\rput(9,2){\begin{tabular}{c} The MAP between\\ $A=A_1\cup A_2$ and $B$, \\ $A\cap B=\{c_1,c_2\}$ \end{tabular}}

\end{pspicture}
\end{center}

Our aim is to provide sufficient conditions for linear convergence
of the sequences generated by the method of alternating projections.
The following two results are simple yet useful.

\begin{proposition}
\label{p:easymap}
Let $A$ and $B$ be nonempty closed subsets of $X$,
and let $(a_n)$ and $(b_n)$ be sequences generated
by the method of alternating projections.
Then the following hold:
\begin{enumerate}
\item
\label{p:easymap1}
The sequences $(a_n)_{\nnn}$ and $(b_n)_\nnn$ lie in $A$ and $B$,
respectively.
\item
\label{p:easymap1+}
$(\forall\nnn)$
$\|a_{n+1}-b_{n+1}\|\leq\|a_{n+1}-b_n\|\leq\|a_n-b_n\|$.
\item
\label{p:easymap2}
If $\{a_n\}_\nnn \cap B\neq\varnothing$,
or $\{b_n\}_\nnn \cap A\neq\varnothing$,
then there exists $c\in A\cap B$ such that
for all $n$ sufficiently large, $a_n=b_n=c$.
\end{enumerate}
\end{proposition}
\begin{proof}
\ref{p:easymap1}: This is clear from the definition.

\ref{p:easymap1+}: Indeed, for every $\nnn$,
$\|a_{n+1}-b_{n+1}\| = d_B(a_{n+1}) \leq \|a_{n+1}-b_n\|
=d_A(b_n)\leq\|b_n-a_n\|$ using \ref{p:easymap1}.

\ref{p:easymap2}:
Suppose, say that $a_n\in B$.
Then $b_n = P_Ba_n = a_n =:c \in A\cap B$ and all
subsequent terms of the sequences are equal to $c$ as well.
\end{proof}


\subsection*{New convergence results for the MAP}

We are now in a position to state and derive
new linear convergence results.
In this section, we shall often assume the following:

\boxedeqn{
\label{e:MAPsettings}
\left\{
\begin{aligned}
&\text{$\mcA := (A_i)_{i\in I}$ and $\mcB := (B_j)_{j\in J}$ are
nontrivial collections}\\
&\quad\text{of nonempty closed subsets of $X$;}\\
&A :=\bigcup_{i\in I} A_i \text{~and~}
B:= \bigcup_{j\in J} B_j\text{~are closed;}\\
&c\in A\cap B; \\
&\text{$\wt{\mcA} := (\wt{A}_i)_{i\in I}$ and
$\wt{\mcB} := (\wt{B}_j)_{j\in J}$ are collections}\\
&\quad\text{of nonempty subsets of $X$ such that }\\
&\qquad (\forall i\in I)\;\;P_{A_i}\big((\bd B)\smallsetminus
A\big)\subseteq\wt{A}_i,\\
&\qquad (\forall j\in J)\;\;P_{B_j}\big((\bd A)\smallsetminus
B\big)\subseteq\wt{B}_j;\\
&\wt{A} :=\bigcup_{i\in I} \wt{A}_i \text{~and~}
\wt{B}:= \bigcup_{j\in J} \wt{B}_j.
\end{aligned}
\right.
}

\begin{lemma}[backtracking MAP]\label{l:back}
Assume that \eqref{e:MAPsettings} holds.
Let $(a_n)_\nnn$ and $(b_n)_\nnn$ be generated by the MAP
with starting point $b_{-1}$.
Let $n\in\{1,2,3,\ldots\}$.
Then the following hold:
\begin{enumerate}
\item\label{l:back-i}
If $b_n\notin A$, then
 $a_n\in((\bd A)\smallsetminus B)\cap\bigcup_{i\in I}(A_i\cap \wt{A}_i)$
 and $b_n\in((\bd B)\smallsetminus A)\cap\bigcup_{j\in J}(B_j\cap \wt{B}_j)$.
\item\label{l:back-ii}
If $a_n\notin B$, then $a_n\in((\bd A)\smallsetminus B)\cap\bigcup_{i\in I}(A_i\cap \wt{A}_i)$.
\item\label{l:back-iii}
If $a_n\notin B$ and $n\geq 2$, then
$b_{n-1}\in((\bd B)\smallsetminus A)\cap\bigcup_{j\in J}(B_j\cap \wt{B}_j)$.
\end{enumerate}
\end{lemma}
\begin{proof}
\ref{l:back-i}:
Applying Proposition~\ref{p:effect}\ref{p:0329a-iii}
to $b_{n-1}\in B$, $a_n\in P_A b_{n-1}$, $b_n\in P_B a_n$, we obtain
\begin{equation}
  b_n\in B\smallsetminus A\ \Leftrightarrow\
  b_n\in\big((\bd B)\smallsetminus A\big)\cap\bigcup_{j\in J}(B_j\cap \wt{B}_j)
  \ \Rightarrow\ a_n\in A\smallsetminus B.
\end{equation}
On the other hand, applying Proposition~\ref{p:effect}\ref{p:0329a-iv}
to $a_{n-1}\in A$, $b_{n-1}\in P_B a_{n-1}$, $a_n\in P_A b_{n-1}$,
we see that
\begin{equation}
  a_n\in A\smallsetminus B\ \Leftrightarrow\
  a_n\in\big((\bd A)\smallsetminus B\big)\cap\bigcup_{i\in I}(A_i\cap \wt{A}_i).
\end{equation}
Altogether, \ref{l:back-i} is established.

\ref{l:back-ii}\&\ref{l:back-iii}:
The proofs are analogous to that of \ref{l:back-i}.
\end{proof}

Let us now state and prove a key technical result.


\begin{proposition}
\label{p:joint}
Assume that \eqref{e:MAPsettings} holds.
Suppose that there exist $\ve\geq0$ and $\dd>0$ such that the
following hold:
\begin{enumerate}
\item\label{p:joint1}
$\mcA$ is $(\wt{B},\ve,3\dd)$-joint-regular at $c$
(see Definition~\ref{d:jreg}) and set
\begin{equation}
\k := \begin{cases}
1, &\text{if $\mcB$ is not known to be $(\wt{A},\ve,3\dd)$-joint-regular at
$c$;}\\
2, &\text{if $\mcB$ is also $(\wt{A},\ve,3\dd)$-joint-regular at $c$.}
\end{cases}
\end{equation}
\item
\label{p:joint2}
$\theta_{3\dd} < 1-2\ve$, where
$\theta_{3\dd}$ is
the joint-CQ-number at $c$ associated with
$(\mcA,\wt{\mcA},\mcB,\wt{\mcB})$
(see Definition~\ref{d:jCQn}).
\end{enumerate}
Set $\theta := \theta_{3\dd}+2\ve \in\zeroun$.
Let $(a_n)_\nnn$ and $(b_n)_\nnn$ be sequences generated by the MAP with
starting point $b_{-1}$ satisfying
\begin{equation}
\label{e:0306a}
\|b_{-1}-c\|\leq \frac{(1-\theta^\k)\dd}{6(2+\theta-\theta^\k)}.
\end{equation}
Then $(a_n)_\nnn$ and $(b_n)_\nnn$ converge linearly to some point
$\bar{c}\in A\cap B$ with rate $\theta^\k$; in fact,
\begin{equation}
\label{e:0306e}
\|\bar{c}-c\|\leq\dd
\quad\text{and}\quad
(\forall n\geq1)\;\;
\max\big\{\|a_n-\bar{c}\|,\|b_n-\bar{c}\|\big\}\leq
\frac{\delta(1+\theta)}{2+\theta-\theta^\k}\theta^{\k (n-1)}.
\end{equation}
\end{proposition}
\begin{proof}
In view of $a_1\in P_AP_BP_Ab_{-1}$ and \eqref{e:0306a},
Corollary~\ref{c:0330a} yields
\begin{equation}
\label{e:120406b}
\beta:=\|a_1-c\|\leq \frac{(1-\theta^\k)\dd}{(2+\theta-\theta^\k)}\leq\frac{\dd}{2}.
\end{equation}
Since $c\in A\cap B$, we have $\theta_{3\dd}\geq 0$ by
\eqref{e:120406a} and hence $\theta>0$.
Using \eqref{e:120406b}, we estimate
\begin{subequations}
\label{e:0306c}
\begin{align}
(\forall n\geq1)\quad
\beta\theta^{\k (n-1)}+\beta+\beta(1+\theta)\sum_{k=0}^{n-2}\theta^{\k k}
&\leq \beta+\beta(1+\theta)\sum_{k=0}^{n-1}\theta^{\k k} \\
&=\beta+\beta(1+\theta)\frac{1-\theta^{\k n}}{1-\theta^\k}\\
&\leq \beta+\beta\frac{1+\theta}{1-\theta^\k}\\
&=\beta\Big(\frac{2+\theta-\theta^\k}{1-\theta^\k}\Big)\\
&\leq \dd.
\end{align}
\end{subequations}

We now claim that if
\begin{equation}\label{e:0327a}
n\geq 1,\quad \|a_n-b_n\|\leq\beta\theta^{\k (n-1)}\quad\mbox{and}\quad
\|a_n-c\|\leq \beta+\beta(1+\theta)\sum_{k=0}^{n-2}\theta^{\k k},
\end{equation}
then
\begin{subequations}
\label{e:0306b}
\begin{align}
\|a_{n+1}-b_{n+1}\|&\leq\theta^{\k-1}\|a_{n+1}-b_{n}\|\leq\theta^{\k}\|a_{n}-b_{n}\|
\leq\beta\theta^{\k n},\label{e:Jb} \\
\|a_{n+1}-c\|&\leq \beta+\beta(1+\theta)\sum_{k=0}^{n-1}\theta^{\k k}.\label{e:Jc}
\end{align}
\end{subequations}
To prove this claim, assume that \eqref{e:0327a} holds.
Using \eqref{e:0327a} and \eqref{e:0306c},
we first observe that
\begin{subequations}
\begin{align}
\label{e:Jd}
\max\big\{\|a_n-c\|, \|b_n-c\|\big\} &\leq \|b_n-a_n\|+\|a_n-c\|\\
&\leq \beta\theta^{\k (n-1)}+\beta+\beta(1+\theta)\sum_{k=0}^{n-2}\theta^{\k k}\leq\dd.
\end{align}
\end{subequations}
We now consider two cases:

\emph{Case~1}: $b_n\in A\cap B$. Then $b_n=a_{n+1}=b_{n+1}$ and
thus \eqref{e:Jb} holds.
Moreover,
$\|a_{n+1}-c\|=\|b_n-c\|$ and
\eqref{e:Jc} follows from \eqref{e:Jd}.

\emph{Case~2}: $b_n\not\in A\cap B$.
Then $b_n\in B\smallsetminus A$.
Lemma~\ref{l:back}\ref{l:back-i} implies $a_n\in((\bd A)\smallsetminus
B)\cap\bigcup_{i\in I}(A_i\cap \wt{A}_i)$ and $b_n\in((\bd B)\smallsetminus
A)\cap\bigcup_{j\in J}(B_j\cap \wt{B}_j)$.
Note that $\|a_n-c\|\leq\dd$ by \eqref{e:Jd}, and recall that
$\mcA$ is $(\wt{B},\ve,3\dd)$-joint-regular at $c$ by \ref{p:joint1}.
It thus follows from
Proposition~\ref{p:effect}\ref{p:effect-iv}
(applied to $a_n,b_n,a_{n+1}$) that
\begin{equation}
\label{e:0325a}
\|a_{n+1}-b_n\|\leq\theta\|a_n-b_n\|.
\end{equation}
On the one hand, if $\k=1$,
then Proposition~\ref{p:easymap}\ref{p:easymap1+} yields
$\|a_{n+1}-b_{n+1}\|\leq\|a_{n+1}-b_n\|=\theta^{\k-1}\|a_{n+1}-b_n\|$.
On the other hand, if $\k=2$,
then $\mcB$ is $(\wt{A},\ve,3\dd)$-joint-regular at
$c$ by \ref{p:joint1}; hence,
Proposition~\ref{p:effect}\ref{p:effect-ii}
(applied to $b_n,a_{n+1},b_{n+1}$) yields
$\|a_{n+1}-b_{n+1}\|\leq\theta\|a_{n+1}-b_n\|=\theta^{\k-1}\|a_{n+1}-b_n\|$.
Altogether, in either case,
\begin{equation}\label{e:0325b}
\|a_{n+1}-b_{n+1}\|\leq\theta^{\k-1}\|a_{n+1}-b_n\|.
\end{equation}
Combining \eqref{e:0325b} with \eqref{e:0325a} and \eqref{e:0327a} gives
\begin{equation}
\label{e:120406c}
\|a_{n+1}-b_{n+1}\|\leq\theta^{\k-1}\|a_{n+1}-b_n\|
\leq\theta^\k\|a_n-b_n\|\leq\beta\theta^{\k n},
\end{equation}
which is \eqref{e:Jb}.
Furthermore,
\eqref{e:0325a}, \eqref{e:0327a} and \eqref{e:Jd} yield
\begin{subequations}
\begin{align}
\|a_{n+1}-c\| &\leq\|a_{n+1}-b_n\|+\|b_n-c\|\\
&\leq \theta\|a_n-b_n\| + \|b_n-c\|\\
&\leq \theta\beta\theta^{\k(n-1)}+\beta\theta^{\k(n-1)}
+\beta+\beta(1+\theta)\sum_{k=0}^{n-2}\theta^{\k k}\\
&=\beta+\beta(1+\theta)\sum_{k=0}^{n-1}\theta^{\k k},
\end{align}
\end{subequations}
which establishes \eqref{e:Jc}. Therefore,
in all cases, \eqref{e:0306b} holds.

Since $\|a_1-b_{1}\|=d_B(a_{1})\leq\|a_{1}-c\|=\beta$, we see that
\eqref{e:0327a} holds for $n=1$.
Thus, the above claim and the principle of mathematical
induction principle imply that \eqref{e:0306b}
holds for every $n\geq 1$.

Next, \eqref{e:Jb} implies
\begin{equation}
\label{e:0307e}
(\forall n\geq1)\quad
\|a_{n+1}-b_n\|\leq \theta\|a_n-b_n\|\quad\text{and}\quad
\|a_{n+1}-b_{n+1}\|\leq\theta^{\k-1}\|a_{n+1}-b_{n}\|.
\end{equation}
In view of \eqref{e:0307e} and $\|a_1-b_{1}\|\leq\beta$,
Proposition~\ref{p:geo} yields $\bar{c}\in A\cap B$ such that
\begin{align}
(\forall n\geq1)\quad
\max\big\{\|a_n-\bar{c}\|,\|b_n-\bar{c}\|\big\}
&\leq \frac{1+\theta}{1-\theta^\k}\|a_1-b_1\|\cdot\theta^{\k (n-1)}\\
&\leq \frac{1+\theta}{1-\theta^\k}\beta\cdot\theta^{\k (n-1)}\\
&\leq\frac{\delta(1+\theta)}{2+\theta-\theta^\k}\theta^{\k (n-1)}.
\end{align}
On the other hand,
\eqref{e:Jc} and \eqref{e:0306c} imply
$(\forall n\geq 1)$
$\|a_{n+1}-c\|\leq\dd$;
thus, letting $n\to\pinf$, we obtain $\|\bar{c}-c\|\leq\dd$.
This completes the proof of \eqref{e:0306e}.
\end{proof}

\begin{remark}
\label{r:joint}
In view of Lemma~\ref{l:0305a}\ref{l:0305ai}\&\ref{l:0305aii},
an aggressive choice
for use in \eqref{e:MAPsettings} is
$(\forall i\in I)$ $\wt{A}_i = \bd A_i$ and $(\forall j\in J)$
$\wt{B}_j = \bd B_j$.
\end{remark}


Our main convergence result on the linear convergence of the MAP is
the following:

\begin{theorem}[linear convergence of the MAP and superregularity]
\label{t:jsuper}
Assume that \eqref{e:MAPsettings} holds and that
$\mcA$ is $\wt{B}$-joint-superregular at $c$
(see Definition~\ref{d:jreg}).
Denote the limiting joint-CQ-number at $c$
associated with $(\mcA,\wt{\mcA},\mcB,\wt{\mcB})$
(see Definition~\ref{d:jCQn}) by $\overline{\theta}$,
and the the exact joint-CQ-number at $c$
associated with $(\mcA,\wt{\mcA},\mcB,\wt{\mcB})$
(see Definition~\ref{d:exactCQn}) by $\overline{\alpha}$.
Assume further that one of the following holds:
\begin{enumerate}
\item
\label{t:jsuperi}
$\overline{\theta}<1$.
\item
\label{t:jsuperii}
$I$ and $J$ are finite, and
$\overline{\alpha}<1$.
\end{enumerate}
Let $\theta \in \left]\overline{\theta},1\right[$
and set $\ve := (\theta-\overline{\theta})/3 >0$.
Then there exists $\dd>0$ such that
the following hold:
\begin{enumerate}[resume]
\item
\label{t:jsuper1}
$\mcA$ is $(\wt{B},\ve,3\dd)$-joint-regular at $c$
(see Definition~\ref{d:jreg}).
\item
\label{t:jsuper2}
$\theta_{3\dd} \leq \overline{\theta}+\ve<  1-2\ve$, where
$\theta_{3\dd}$ is
the joint-CQ-number at $c$ associated with
$(\mcA,\wt{\mcA},\mcB,\wt{\mcB})$
(see Definition~\ref{d:jCQn}).
\end{enumerate}
Consequently, suppose the starting point of the MAP $b_{-1}$ satisfies
$\|b_{-1}-c\|\leq (1-\theta)\dd/12$.
Then $(a_n)_\nnn$ and $(b_n)_\nnn$ converge linearly
to some point in $\bar{c}\in A\cap B$ with $\|\bar{c}-c\|\leq\dd$ and rate $\theta$:
\begin{equation}
(\forall n\geq1)\ \max\{\|a_n-\bar{c}\|,\|b_n-\bar{c}\|\}\leq\frac{\dd(1+\theta)}{2}\theta^{n-1}.
\end{equation}
\end{theorem}
\begin{proof}
Observe that \ref{t:jsuperii} implies \ref{t:jsuperi} by
Theorem~\ref{t:CQ1}\ref{t:CQ1iv}.
The definitions of $\wt{B}$-joint-superregularity and of
$\overline\theta$ allow us to find $\dd>0$ sufficiently small such that
both \ref{t:jsuper1} and \ref{t:jsuper2} hold.
The result thus follows from Proposition~\ref{p:joint} with $\k=1$.
\end{proof}

\begin{corollary}
\label{c:jsuper}
Assume that \eqref{e:MAPsettings} holds and
that, for every $i\in I$, $A_i$ is convex.
Denote the limiting joint-CQ-number at $c$
associated with $(\mcA,\wt{\mcA},\mcB,\wt{\mcB})$
(see Definition~\ref{d:jCQn}) by $\overline{\theta}$,
and assume that $\overline{\theta}<1$.
Let $\theta \in \left]\overline{\theta},1\right[$,
and let $b_{-1}$, the starting point of the MAP,
be sufficiently close to $c$.
Then $(a_n)_\nnn$ and $(b_n)_\nnn$ converge linearly
to some point in $A\cap B$ with rate $\theta$.
\end{corollary}
\begin{proof}
Combine Theorem~\ref{t:jsuper} with Corollary~\ref{c:jsreg}.
\end{proof}

\begin{example}[working with collections and joint notions is useful]
Consider the setting of Example~\ref{ex:jCQn<CQn},
and suppose that $\wt{\mcA}=\mcA$ and $\wt{\mcB}=\mcB$.
Note that $A_i$ is convex, for every $i\in I$.
Then $\theta_\dd(\mcA,\wt{\mcA},\mcB,\wt{\mcB})<1=\theta_\dd(A,A,B,B)
=\overline{\theta}(A,X,B,X)$.
Hence Corollary~\ref{c:jsuper} guarantees linear convergence of the
MAP while it is not possible to work directly with the unions $A$ and $B$ due
to their condition number being equal to $1$ \emph{and} because
neither $A$ nor $B$ is superregular by Example~\ref{ex:badlines}!
This illustrates that
the main result of Lewis-Luke-Malick
(see Corollary~\ref{c:LLM} below) is not applicable because two of its
hypotheses fail.
\end{example}


The following result features an improved rate of
convergence $\theta^2$ due to the additional presence of superregularity.

\begin{theorem}[linear convergence of the MAP and double superregularity]
\label{t:jointdoubly}
Assume that \eqref{e:MAPsettings} holds,
that $\mcA$ is $\wt{B}$-joint-superregular at $c$
and that $\mcB$ is $\wt{A}$-joint-superregular at $c$
(see Definition~\ref{d:jreg}).
Denote the limiting joint-CQ-number at $c$
associated with $(\mcA,\wt{\mcA},\mcB,\wt{\mcB})$
(see Definition~\ref{d:jCQn}) by $\overline{\theta}$,
and the the exact joint-CQ-number at $c$
associated with $(\mcA,\wt{\mcA},\mcB,\wt{\mcB})$
(see Definition~\ref{d:exactCQn}) by $\overline{\alpha}$.
Assume further that
{\rm (a)} $\overline{\theta}<1$, or (more restrictively) that
{\rm (b)} $I$ and $J$ are finite, and
$\overline{\alpha}<1$ (and hence
$\overline{\theta}=\overline{\alpha}<1$).
Let $\theta \in \left]\overline{\theta},1\right[$ and
$\ve:=\frac{\theta-\overline{\theta}}{3}$.
Then there exists $\dd>0$ such that
\begin{enumerate}
\item\label{t:jdb-i} $\mcA$ is $(\wt{B},\ve,3\dd)$-joint-regular at $c$;
\item\label{t:jdb-ii} $\mcB$ is $(\wt{A},\ve,3\dd)$-joint-regular at
$c$; and
\item\label{t:jdb-iii} $\theta_{3\dd}<\overline{\theta}+\ve=\theta-2\ve<1-2\ve$, where
$\theta_{3\dd}$ is
the joint-CQ-number at $c$ associated with
$(\mcA,\wt{\mcA},\mcB,\wt{\mcB})$
(see Definition~\ref{d:jCQn}).
\end{enumerate}
Consequently, suppose the starting point of MAP $b_{-1}$ satisfies
$\|b_{-1}-c\|\leq\frac{(1-\theta)\dd}{6(2-\theta)}$.
Then $(a_n)_\nnn$ and $(b_n)_\nnn$ converge linearly
to some point in $\bar{c}\in A\cap B$ with $\|\bar{c}-c\|\leq\dd$
and rate $\theta^2$; in fact,
\begin{equation}
(\forall n\geq1)\quad
\max\big\{\|a_n-\bar{c}\|,\|b_n-\bar{c}\|\big\}
\leq\frac{\dd}{2-\theta}\big(\theta^2\big)^{n-1}.
\end{equation}
\end{theorem}
\begin{proof}
The existence of $\dd>0$ such that \ref{t:jdb-i}--\ref{t:jdb-iii} hold
is clear.
Then apply Proposition~\ref{p:joint} with $\k=2$.
\end{proof}


In passing, let us point out
a sharper rate of convergence under sufficient conditions
stronger than superregularity.

\begin{corollary}[refined convergence rate]\label{c:jdb}
Assume that \eqref{e:MAPsettings} holds and
that there exists $\dd>0$ such that
\begin{enumerate}
\item $\mcA$ is $(\wt{B},0,3\dd)$-joint-regular at $c$;
\item $\mcB$ is $(\wt{A},0,3\dd)$-joint-regular at $c$; and
\item $\theta<1$, where
$\theta:=\theta_{3\dd}$ is
the joint-CQ-number at $c$ associated with
$(\mcA,\wt{\mcA},\mcB,\wt{\mcB})$
(see Definition~\ref{d:jCQn}).
\end{enumerate}
Suppose also that the starting point of the MAP $b_{-1}$ satisfies
$\|b_{-1}-c\|\leq\frac{(1-\theta)\dd}{6(2-\theta)}$.
Then $(a_n)_\nnn$ and $(b_n)_\nnn$ converge linearly
to some point in $\bar{c}\in A\cap B$ with $\|\bar{c}-c\|\leq\dd$
and rate $\theta^2$; in fact,
\begin{equation}
(\forall n\geq1)\quad
\max\big\{\|a_n-\bar{c}\|,\|b_n-\bar{c}\|\big\}\leq
\frac{\dd}{2-\theta}\big(\theta^2\big)^{n-1}.
\end{equation}
\end{corollary}
\begin{proof}
Apply Proposition~\ref{p:joint} with $\k=2$.
\end{proof}

Let us illustrate a situation where it is possible to make $\dd$ in
Theorem~\ref{t:jointdoubly} precise.

\begin{example}[the MAP for two spheres]
Let $z_1$ and $z_2$ be in $X$,
let $\rho_1$ and $\rho_2$ be in $\RR$,
set $A:= \sphere{z_1}{\rho_1}$ and $B := \sphere{z_2}{\rho_2}$,
and assume that $\{c\}\subsetneqq A\cap B\subsetneqq A\cup B$.
Then $\overline{\alpha} := |\scal{z_1-c}{z_2-c}|/(\rho_1\rho_2)<1$.
Let $\theta\in\left]\overline{\alpha},1\right[$.
Then the conclusion of Theorem~\ref{t:jointdoubly} holds with
\begin{equation}
\dd := \min
\Bigg\{
\frac{\sqrt{(\rho_1+\rho_2)^2+\rho_1\rho_2(\theta-\overline{\alpha})}-(\rho_1+\rho_2)}{6},\frac{\ve\rho_1}{3},\frac{\ve\rho_2}{3}\Bigg\}
\end{equation}
\end{example}
\begin{proof}
Combine Example~\ref{ex:sphere1} (applied with
$\ve=(\theta-\overline{\alpha})/4$ there),
Proposition~\ref{p:sphere2}, and Theorem~\ref{t:jointdoubly}.
\end{proof}


Here is a useful special case of Theorem~\ref{t:jointdoubly}:

\begin{theorem}\label{t:dregAff}
Assume that $A$ and $B$ are $L$-superregular, and that
\begin{equation}
N_A(c)\cap\big(-N_B(c)\big)\cap\big(L-c\big)=\{0\},
\end{equation}
where $L:=\aff(A\cup B)$.
Then the sequences generated by the MAP converge linearly to
a point in $A\cap B$ provided that the starting point is sufficiently
close to $c$.
\end{theorem}
\begin{proof}
Combine
Example~\ref{ex:compareCQ1} with Theorem~\ref{t:jointdoubly}
(applied with $I$ and $J$ being singletons, and with
$\wt{A}=\wt{B}=L$).
\end{proof}


We now obtain a well known global linear convergence result for the convex case,
which does not require the starting point to be sufficiently close to $A\cap B$:

\begin{theorem}[two convex sets]\label{t:globalCvex}
Assume that $A$ and $B$ are convex, and $A\cap B\neq\varnothing$.
Then for every starting point $b_{-1}\in X$, the sequences $(a_n)_\nnn$
and $(b_n)_\nnn$ generated by the MAP converge to some
point in $A\cap B$.
The convergence of these sequences is linear provided that
$\reli A \cap \reli B \neq\varnothing$.
\end{theorem}
\begin{proof}
By Fact~\ref{f:convproj}\ref{f:convproj4}, we have
\begin{equation}\label{e:cvex1}
(\forall c\in A\cap B)\quad
\|a_0-c\|\geq\|b_0-c\|\geq\|a_1-c\|\geq\|b_1-c\|\geq\cdots
\end{equation}
After passing to subsequences if needed, we assume that
$a_{k_n}\to a\in A$ and $b_{k_n}\to b\in B$.
We show that $a=b$ by contradiction, so we assume
that $\ve := \|a-b\|/3>0$.
We have eventually
$\max\{\|a_{k_n}-a\|,\|b_{k_n}-b\|\}<\ve$;
hence $\|a_{k_n}-b_{k_n}\|\geq \ve$ eventually.
By Fact~\ref{f:convproj}\ref{f:convproj3}, we have
\begin{equation}
\|a_{k_n}-c\|^2\geq \|a_{k_n}-b_{k_n}\|^2+\|b_{k_n}-c\|^2\geq \ve^2+\|a_{k_n+1}-c\|^2 \geq \ve^2+\|a_{k_{n+1}}-c\|^2
\end{equation}
eventually.
But this would imply that for all $n$ sufficiently large,
and for every $m\in\NN$, we have
$\|a_{k_n}-c\|^2\geq m\ve^2+\|a_{k_{n+m}}-c\|^2\geq m\ve^2$,
which is absurd.
Hence $\bar{c} := a=b\in A\cap B$ and now \eqref{e:cvex1} (with
$c=\bar{c}$) implies that $a_n\to\bar{c}$ and $b_n\to\bar{c}$.

Next, assume that $\reli A \cap\reli B\neq\varnothing$, and
set $L := \aff(A\cup B)$.
By Proposition~\ref{p:0301a},
the $(A,L,B,L)$-CQ conditions holds at $\bar{c}$.
Thus, by Example~\ref{ex:compareCQ1},
$N_A(\bar{c})\cap(-N_B(\bar{c}))\cap(L-\bar{c})=\{0\}$.
Furthermore, Corollary~\ref{c:jsreg} and
Remark~\ref{r:0303a}\ref{r:0303avi}\&\ref{r:0303avii} imply
that $A$ and $B$ are $L$-superregular at $\bar{c}$.
The conclusion now follows from Theorem~\ref{t:dregAff},
applied to suitably chosen tails of the sequences $(a_n)_\nnn$ and
$(b_\nnn)$.
\end{proof}


\begin{example}[the MAP for two linear subspaces]
\label{ex:vN}
Assume that $A$ and $B$ are linear subspaces of $X$.
Since $0\in A\cap B = \reli A\cap \reli B$,
Theorem~\ref{t:globalCvex} guarantees the linear convergence
of the MAP to some point in $A\cap B$, where
$b_{-1}\in X$ is the arbitrary starting point.
On the other hand, $A$ and $B$ are $(0,+\infty)$-regular
(see Remark~\ref{r:0303a}\ref{r:0303avi}).
Since $(\forall\dd\in\RPP)$ $\theta_{\dd}(A,A,B,B)=c(A,B)<1$,
where $c(A,B)$ is the cosine of the Friedrichs angle between $A$ and
$B$ (see Theorem~\ref{t:CQn=c}),
we obtain from Corollary~\ref{c:jdb} that
the rate of convergence is $c^2(A,B)$.
In fact, it is well known that this is the optimal rate,
and also that $\lim_{n} a_n=\lim_{n}
b_n = P_{A\cap B}(b_{-1})$; see \cite[Section~3]{Deut94} and
\cite[Chapter~9]{Deutsch}.
\end{example}

\begin{remark}
For further linear convergence results for the MAP in
the convex setting we refer the reader to
\cite{BB93},
\cite{bb96},
\cite{BBL},
\cite{DHp1},
\cite{DHp2},
\cite{DHp3},
and the references therein.
See also \cite{Luke08}
and \cite{Luke12} for recent related work for the nonconvex case.
\end{remark}


\subsection*{Comparison to Lewis-Luke-Malick results and further
examples}


The main result of Lewis, Luke, and Malick arises as a special case of
Theorem~\ref{t:jsuper}:

\begin{corollary}[Lewis-Luke-Malick]\label{c:LLM}
{\rm (See \cite[Theorem~5.16]{LLM}.)}
Suppose that $N_A(c)\cap (-N_B(c))=\{0\}$ and that $A$ is superregular at
$c\in A\cap B$. If the starting point of MAP is sufficiently close to
$c$, then the sequences generated by the MAP converge linearly to a
point in $A\cap B$.
\end{corollary}
\begin{proof}
Since $N_A(c)\cap (-N_B(c))=\{0\}$, we have
$\overline{\theta}<1$. Now apply
Theorem~\ref{t:jsuper}(i) with $\wt{\mcA}:=\wt{\mcB}:=(X)$,
$\mcA:=(A)$ and $\mcB:=(B)$.
\end{proof}

However, even in simple situations,
Corollary~\ref{c:LLM} is not powerful enough to recover known convergence
results.

\begin{example}[Lewis-Luke-Malick CQ may fail even for two subspaces]
\label{ex:LLMsubsp}
Suppose that $A$ and $B$ are two linear subspaces of $X$, and set
$L:=\aff(A\cup B)=A+B$.
For $c\in A\cap B$, we have
\begin{equation}
N_{A}(c)\cap(-N_{B}(c))=A^\perp \cap B^\perp
=(A+B)^\perp = L^\perp.
\end{equation}
Therefore, the Lewis-Luke-Malick CQ (see \cite[Theorem~5.16]{LLM} and
also  Corollary~\ref{c:LLM}) holds
for $(A,B)$ at $c$ if and only if
\begin{equation}
N_{A}(c)\cap(-N_{B}(c))=\{0\}\ \Leftrightarrow\ A+B=X.
\end{equation}
On the other hand, the CQ provided in Theorem~\ref{t:dregAff}
(see also Example~\ref{ex:vN})
\emph{always holds} and we obtain linear convergence of the MAP.
However, even for two lines in $\RR^3$, the Lewis-Luke-Malick CQ
(see Corollary~\ref{c:LLM}) is unable to achieve this.
(It was this example that originally motivated us to pursue the
present work.)
\end{example}

\begin{example}[Lewis-Luke-Malick CQ is too strong even for convex sets]
\label{ex:LLMcv}
Assume that $A$ and $B$ are convex (and hence superregular).
Then the Lewis-Luke-Malick CQ condition is $0\in\inte(B-A)$
(see Corollary~\ref{c:CQ3i}) while the $(A,\aff(A\cup B),B,\aff(A\cup
B))$-CQ is equivalent to the much less restrictive condition
$\reli A\cap \reli B\neq\varnothing$ (see Theorem~\ref{t:compareCQ2}).
\end{example}

\subsection*{The flexibility of choosing $(\wt{A},\wt{B})$}

Often, $L=\aff(A\cup B)$ is a convenient choice which
yields linear convergence of the MAP as in Theorem~\ref{t:dregAff}.
However, there are situations when this choice for
$\wt{A}$ and $\wt{B}$ is not helpful but when a different,
more aggressive, choice does guarantee linear convergence:

\begin{example}[$(\wt{A},\wt{B})=(A,B)$]
\label{ex:AABB}
Let $A$, $B$, and $c$ be as in Example~\ref{ex:CQdif(AB)},
and let $L := \aff(A\cup B)$.
Since $A$ and $B$ are \emph{convex} and hence \emph{superregular},
the $(A,L,B,L)$-CQ condition is equivalent to $\reli A\cap \reli
B\neq\varnothing$ (see Proposition~\ref{p:0301a}), which fails in this
case. However, the $(A,A,B,B)$-CQ condition does hold;
hence, the corresponding limiting CQ-number is less than 1
by Theorem~\ref{t:CQ1}\ref{t:CQ1v}.
Thus linear convergence of the MAP is guaranteed by
Theorem~\ref{t:jointdoubly}.
\end{example}

The next example illustrates a situation
where the choice
$(\wt{A},\wt{B})=(A,B)$ fails while
the even tighter choice
$(\wt{A},\wt{B})=(\bd A,\bd B)$ results in success:

\begin{example}[${(\wt{A},\wt{B})=(\bd A,\bd B)}$]
\label{ex:diff-choice}
Suppose that $X=\RR^2$, that $A=\epi(|\cdot|/2)$,
that $B=-\epi(|\cdot|/3)$, and that $c=(0,0)$.
Note that $\aff(A\cup B)=X$ and $\reli A\cap \reli B=\varnothing$.
Then
\begin{subequations}
\begin{align}
&\nc{A}{{B}}(c)=\nc{A}{X}(c)=N_A(c)=
\menge{(u_1,u_2)\in\RR^2}{u_2+2|u_1|\leq0},\\
&\nc{B}{{A}}(c)=\nc{B}{X}(c)=N_B(c)=
\menge{(u_1,u_2)\in\RR^2}{-u_2+3|u_1|\leq0},
\end{align}
\end{subequations}
and so the $(A,{A},B,{B})$-CQ condition fails because
\begin{equation}
\nc{A}{{B}}(c)\cap(-\nc{B}{{A}}(c))=
\menge{(u_1,u_2)\in\RR^2}{u_2+3|u_1|\leq0}\neq\{0\}.
\end{equation}
Consequently, for either
$(\wt{A},\wt{B})=(A,B)$ or $(\wt{A},\wt{B})=(X,X)$,
Theorem~\ref{t:jointdoubly}  is not applicable because
$\overline\alpha=\overline{\theta}=1$:
indeed,
$u=(0,-1)\in N_A(c)$ and $v=(0,-1)\in-N_B(c)$,
so $1=\scal{u}{v}\leq\bar\alpha\leq 1$.

On the other hand, let us now choose
$(\wt{A},\wt{B})=(\bd A,\bd B)$, which is justified by
Remark~\ref{r:joint}.
Then
\begin{subequations}
\begin{align}
&\nc{A}{\wt{B}}(c)=\menge{(u_1,u_2)\in\RR^2}{u_2+2|u_1|=0},\\
&\nc{B}{\wt{A}}(c)=\menge{(u_1,u_2)\in\RR^2}{-u_2+3|u_1|=0},
\end{align}
\end{subequations}
$\nc{A}{\wt{B}}(c)\cap(-\nc{B}{\wt{A}}(c))=\{0\}$ and
the $(A,\wt{A},B,\wt{B})$-CQ condition holds.
Hence, using also Theorem~\ref{t:CQ1}\ref{t:CQ1v},
Theorem~\ref{t:globalCvex} and
Theorem~\ref{t:jointdoubly}, we deduce linear convergence of the MAP.
\end{example}

However, even the choice
$(\wt{A},\wt{B})=(\bd A,\bd B)$ may not be applicable to
yield the desired linear convergence as the following shows.
In this example, we employ
the tightest possibility allowed by our framework, namely
$(\wt{A},\wt{B})=(P_A((\bd B)\smallsetminus A),P_B((\bd A)\smallsetminus
B))$.

\begin{example}[$(\wt{A},\wt{B})=(P_A((\bd B)\smallsetminus A),P_B((\bd A)\smallsetminus
B))$ ]
Suppose that $X=\RR^2$, that $A=\epi(|\cdot|)$, that $B=-A$, and that
$c=(0,0)$. Then
$\nc{A}{\bd B}(c)=\bd B = -\bd A$
and
$\nc{B}{\bd A}(c)=\bd A$;
hence, the $(A,\bd A,B,\bd B)$-CQ condition fails because
$\nc{A}{\bd B}(c) \cap (-\nc{B}{\bd A}(c)) = \bd B\neq \{0\}$.
On the other hand,
if $(\wt{A},\wt{B})=(P_A((\bd B)\smallsetminus A),P_B((\bd A)\smallsetminus
B))$, then $\nc{A}{\wt{B}} = \{0\} = \nc{B}{\wt{A}}=\{0\}$
because $\wt{A}=\{c\} = \wt{B}$.
Thus, the
$(A,\wt{A},B,\wt{B})$-CQ conditions holds.
(Note that the MAP converges in finitely many steps.)
\end{example}


\section*{Conclusion}

We have introduced restricted normal cones which generalize classical
normal cones. We have presented some of their basic properties and
shown their usefulness in describing interiority conditions,
constraint qualifications, and regularities.
The corresponding results were employed to yield new powerful
sufficient conditions for
linear convergence of the sequences generated by the method
of alternating projections applied to two sets $A$ and $B$.
A key ingredient were suitable restricting sets
$(\wt{A}$ and $\wt{B})$.
The least aggressive choice,
$(\wt{A},\wt{B})=(X,X)$,
recovers the framework by Lewis, Luke, and Malick.
The choice $(\wt{A},\wt{B})=(\aff(A\cup B),\aff(A\cup B))$
allows us to include basic settings from convex analysis into our framework.
Thus, the framework provided here unifies the recent nonconvex results
by Lewis, Luke, and Malick with classical convex-analytical settings.
When the choice $(\wt{A},\wt{B})=(\aff(A\cup B),\aff(A\cup B))$
fails, one may also try more aggressive choices such as
$(\wt{A},\wt{B})=(A,B)$ or
$(\wt{A},\wt{B})=(\bd A,\bd B)$ to guarantee linear convergence.
In a follow-up work \cite{BLPW12b} we demonstrate the power of 
these tools with the important 
problem of sparsity optimization with affine constraints.  Without any assumptions 
on the regularity of the sets or the intersection we achieve local
convergence results, with rates and radii of convergence, where all other 
sufficient conditions, particularly those of \cite{LM} and \cite{LLM}, fail.  
 
\subsection*{Acknowledgments}
HHB was partially supported by the Natural Sciences and Engineering
Research Council of Canada and by the Canada Research Chair Program.
This research was initiated when HHB visited the
Institut f\"ur Numerische und Angewandte Mathematik,
Universit\"at G\"ottingen because of his study leave
in Summer~2011. HHB thanks DRL and the Institut for their hospitality.
DRL was supported in part by the German Research Foundation grant SFB755-A4. 
HMP was partially
supported by the Pacific Institute for the Mathematical Sciences and
and by a University of British Columbia research grant. XW was
partially supported by the Natural Sciences and Engineering Research
Council of Canada.


\end{document}